\documentclass[amsfonts,12pt]{amsart}

\usepackage{amssymb}
\usepackage{etoolbox}
\usepackage{graphicx}
\usepackage{stmaryrd}
\usepackage{color}
\usepackage{pifont}
\usepackage{enumerate}
\DeclareMathOperator*{\esssup}{ess\,sup}
\DeclareMathOperator*{\essinf}{ess\,inf}
\newcommand{\ee}{\varepsilon}

\newcommand{\N}{{\mathbb N}}
\newcommand{\Q}{{\mathbb Q}}
\newcommand{\R}{{\mathbb R}}

\newcommand{\cH}{{\mathcal H}}

\newcommand{\cK}{\mathcal K}
\newcommand{\cL}{\mathcal L}
\newcommand{\cM}{\mathcal M}

\newcommand{\cS}{\mathcal S}

\newtheorem{thm}{Theorem}[section]

\newtheorem{lem}[thm]{Lemma}
\newtheorem{cor}[thm]{Corollary}

\newtheorem{prop}[thm]{Proposition}
\newtheorem{ex}[thm]{Example}

\newcommand{\supp}{{\mathrm{supp}}\,}

\newcommand{\inte}{{\mathrm{int}}\,}

\newcommand{\cl}{{\mathrm{cl}}\,}

\newcommand{\pa}{{\partial}}
\newcommand{\ga}{{\gamma}}

\newcommand{\cE}{{\mathcal E}}
\newcommand{\cV}{{\mathcal V}}

\newcommand{\di}{\diamondsuit}

\newcommand{\uomc}{\overline{\cM}^*}
\newcommand{\lomc}{\overline{\cM}_*}

\newcommand{\uomcC}{\overline{\cM}^*_C}
\newcommand{\lomcC}{\overline{\cM}_{*C}}
\newcommand{\omcC}{\overline{\cM}_C}

\newcommand{\lphi}{{L^\Phi(\R^n)}}
\newcommand{\lphiom}{{L^\Phi(\Omega)}}
\newcommand{\lphir}{{L^{\Phi_r}(\R^n)}}

\numberwithin{equation}{section}

\begin{document}
\hfill\today
\bigskip

\title{The P\'olya-Szeg\H{o} inequality for smoothing rearrangements}
\author[Gabriele Bianchi, Richard J. Gardner, Paolo Gronchi, and Markus Kiderlen]
{Gabriele Bianchi, Richard J. Gardner, Paolo Gronchi, and Markus Kiderlen}
\address{Dipartimento di Matematica e Informatica ``U. Dini", Universit\`a di Firenze, Viale Morgagni 67/A, Firenze, Italy I-50134} \email{gabriele.bianchi@unifi.it}
\address{Department of Mathematics, Western Washington University,
Bellingham, WA 98225-9063,USA} \email{Richard.Gardner@wwu.edu}
\address{Dipartimento di Matematica e Informatica ``U. Dini", Universit\`a di Firenze, Piazza Ghiberti 27, Firenze, Italy I-50122} \email{paolo.gronchi@unifi.it}
\address{Department of Mathematical Sciences, University of Aarhus,
Ny Munkegade, DK--8000 Aarhus C, Denmark} \email{kiderlen@math.au.dk}
\thanks{First and third author supported in part by the Gruppo
Nazionale per l'Analisi Matematica, la Probabilit\`a e le loro
Applicazioni (GNAMPA) of the Istituto Nazionale di Alta Matematica (INdAM).  Fourth author supported by the Centre for Stochastic Geometry and Advanced Bioimaging, funded by a grant from the Villum Foundation.}
\subjclass[2020]{Primary: 28A20, 52A20; secondary: 46E30, 52A38} \keywords{P\'olya-Szeg\H{o} inequality, convex body, Schwarz symmetrization, rearrangement, polarization, smoothing, modulus of continuity, Orlicz norm, anisotropic}

\maketitle

\begin{abstract}
A basic version of the P\'olya-Szeg\H{o} inequality states that if $\Phi$ is a Young function, the $\Phi$-Dirichlet energy---the integral of $\Phi(\|\nabla f\|)$---of a suitable function $f\in \cV(\R^n)$, the class of nonnegative measurable functions on $\R^n$ that vanish at infinity, does not increase under symmetric decreasing rearrangement. This fact, along with variants that apply to polarizations and to Steiner and certain other rearrangements, has numerous applications.  Very general versions of the inequality are proved that hold for all smoothing rearrangements, those that do not increase the modulus of continuity of functions.  The results cover all the main classes of functions previously considered: Lipschitz functions $f\in \cV(\R^n)$, functions $f\in W^{1,p}(\R^n)\cap\cV(\R^n)$ (when $1\le p<\infty$ and $\Phi(t)=t^p$), and functions $f\in W^{1,1}_{loc}(\R^n)\cap\cV(\R^n)$.  In addition, anisotropic versions of these results, in which the role of the unit ball is played by a convex body containing the origin in its interior, are established.  Taken together, the results bring together all the basic versions of the P\'olya-Szeg\H{o} inequality previously available under a common and very general framework.
\end{abstract}

\section{Introduction}

A familiar version of the P\'olya-Szeg\H{o} inequality states that if $1\le p\le \infty$ and  $f\in W^{1,p}(\R^n)\cap \cV(\R^n)$, then $f^{\#}\in W^{1,p}(\R^n)$ and
\begin{equation}\label{intro1}
\int_{\R^n}\|\nabla f^{\#}(x)\|^p\, dx\leq \int_{\R^n}\|\nabla f(x)\|^p\; dx.
\end{equation}
See, e.g., \cite[Theorem~3.20 and~p.~113]{Baer}; when $p=\infty$, the integrals of $p$th powers are replaced by the essential suprema over $\R^n$. Here $f^{\#}$ denotes the {\em symmetric decreasing rearrangement} of $f$ (another common notation is $f^{\star}$) and $\cV(\R^n)$ is the class of nonnegative measurable functions on $\R^n$ that vanish at infinity, a natural class for which this rearrangement is defined. Most of our definitions and terminology can be found in Sections~\ref{subsec:notations} and~\ref{Properties}.  However, since we have no need for a precise definition of $f^{\#}$ (see, e.g., \cite[Definition~1.29]{Baer}, \cite[p.~9]{Kaw}, \cite[p.~80]{LL}), we lean on a vivid description of Sperner \cite[Abstract]{Spe}:  Imagine the subgraph of $f$ as a lump of clay on a potter's wheel, which on turning is molded into a perfectly symmetrical shape, maintaining the height of each particle of clay.  The molded shape then represents the subgraph of $f^{\#}$.  The map that takes $f$ to $f^{\#}$ is the primary example of a rearrangement on $\cV(\R^n)$.  In general, if $X$ is a class of measurable functions on $\R^n$ containing the characteristic functions of sets in $\cL^n$, the $\cH^n$-measurable sets of finite measure, a {\em rearrangement} $T$ on $X$ is an essentially monotonic (i.e., monotonic up to sets of $\cH^n$-measure zero) and equimeasurable (preserving the $\cH^n$-measure of superlevel sets) map $T:X\to X$.

Inequality \eqref{intro1} has its roots in studies of symmetrization of sets and rearrangements of functions that go back to Jakob Steiner's work on the isoperimetric inequality around 1836.  (The isoperimetric inequality can actually be deduced from the case $p=1$ of \eqref{intro1}.)  With extra assumptions on $f$, it was first proved for $n=p=2$ by G.~Faber and E.~Krahn independently, and then for all $n$ and $p=2$ by Krahn, all in the 1920s. In fact, \eqref{intro1} was a key ingredient in the solution by Faber and Krahn of Lord Rayleigh's 1884 conjecture that the disk has the lowest fundamental frequency of vibration of all membranes of a given area.  References are given by Daners \cite{Dan} in his detailed commentary focusing on Krahn's solution, and by Mondino and Semola \cite{MS}, who provide a lucid account of this early history.  (The latter also describe extensions of \eqref{intro1} to non-Euclidean settings, but the present paper is set entirely in $\R^n$.)  The many sources that outline the recent history of \eqref{intro1} often contradict each other.  P\'olya and Szeg\H{o}'s classic text \cite{PS} on isoperimetric inequalities in mathematical physics is always cited, but \eqref{intro1} is not explicitly stated there, and the setting is different.  They work with smooth surfaces $A_0$ and $A_1$ in $\R^3$ with $A_0$ in the interior of $A_1$, and functions equal to 0 on $A_0$, $1$ on $A_1$, and between 0 and 1 in the region bounded by $A_0$ and $A_1$. Their arguments in \cite[pp.~154--156]{PS}, with additional work, lead to \eqref{intro4} below when $f$ is smooth with compact support, $n=3$, and $T$ is the $(2,3)$-Steiner rearrangement (explained below) with respect to the $xy$-plane.  This and the approximation of Schwarz rearrangement by a sequence of $(2,3)$-Steiner rearrangements sketched in \cite[pp.~157]{PS} yield \eqref{intro4}, and hence \eqref{intro1} for $1<p<\infty$, with the same restrictions on $f$ and $n$. We believe Baernstein \cite[Section~3.8]{Baer} is correct in giving credit to Sperner \cite{Spe} for proving \eqref{intro1} when $f$ is Lipschitz (though he cites the wrong paper) and Hild\'{e}n \cite{Hil} for the result as stated above when $p<\infty$.

Diverse variants and applications of the P\'olya-Szeg\H{o} inequality (often called the P\'olya-Szeg\H{o} principle) have generated a very substantial literature, surveyed by Talenti \cite[Sections~1.3 and~1.5]{Tal2}, \cite[Section~5]{Tal}, who in \cite[p.~126]{Tal} provides over fifty references. The main themes are: P\'olya-Szeg\H{o} inequalities on spheres, hyperbolic, or other spaces, and for other functionals of the gradient; weighted versions involving other measures; versions invariant under affine transformations; anisotropic inequalities; the examination of equality cases; connections with capacitary inequalities; and applications to mathematical physics, PDEs, and function spaces.

Like \eqref{intro1}, this paper has also arisen from earlier work on symmetrization and rearrangement, including our previous investigations \cite{BGG, BGG2, BGGK}.  As in those articles, the attention is less on particular symmetrizations or rearrangements than on general properties that allow results for those special cases to be extended and unified.    For the classes $X$ of measurable functions considered, such as $\cV(\R^n)$, each rearrangement $T$ is essentially determined by an associated map $\di_T:\cL^n\to \cL^n$ defined by $\di_TA=\{x:T1_A(x)=1\}$, where $1_A$ is the characteristic function of $A$, satisfying
$$\{x: Tf(x)\ge t\}=\di_T\{x: f(x)\ge t\},$$
essentially, for $t>\essinf f$; see Proposition~\ref{lemapril30} below. Another such map may then be defined by $\di^*_TA=(\di_TA)^*$, where $E^*$ denotes the set of density points of $E$.  We focus here on {\em smoothing} rearrangements, those for which
\begin{equation}\label{intro2}
(\di^*_TA)+dB^n\subset \di^*_T(A+dB^n),
\end{equation}
essentially, for each $d>0$ and bounded measurable set $A$. Several equivalent variations of this definition, which stems from that of Sarvas \cite[p.~11]{Sar72}, are given in Lemma~\ref{lemapr3}.  The use of density points on the left of \eqref{intro2} is crucial and a feature of our methods, which differ from those in related studies of rearrangements by Brock and Solynin \cite{BS} and Van Schaftingen and Willem \cite{VSW}.  See \cite[Appendix]{BGGK} and the remarks around \eqref{eqjan72} and at the beginning of Section~\ref{Smoothing rearrangements} below for commentary about the various approaches.

It turns out that for the main classes $X$ of interest, smoothing rearrangements are precisely the rearrangements $T:X\to X$ that {\em reduce the modulus of continuity} of functions in $X$, that is,  such that $\omega_d(Tf)\le \omega_d(f)$ for $d>0$ and $f\in X$, where
$$
\omega_d(f)=\esssup_{\|x-y\|\le d}|f(x)-f(y)|.
$$
This result, a consequence of Corollary~\ref{lemdec301}, relies on several others, such as Theorems~\ref{lemdec201} and~\ref{lemdec194}, which collectively generalize (even in the special case when $K=B^n$) the theorem of Brock and Solynin \cite[Theorem~3.3]{BS}.  Corollary~\ref{lemdec301} also shows that the equivalence of smoothing and reduction of the modulus of continuity is true when only the continuous functions, or indeed the contractions, in $X$, are considered.  Via this equivalence and results in the literature, it can be seen that the class of smoothing rearrangements on $\cV(\R^n)$ includes the symmetric decreasing rearrangement (see, e.g., \cite[Theorem~2.12]{Baer}); more generally, the {\em Schwarz} or {\em $(k,n)$-Steiner} rearrangement with respect to a $k$-dimensional subspace in $\R^n$ \cite[Theorem~6.10]{Baer}, \cite[Corollary~6.1]{BS} (here the axis of the potter's wheel is $(n-k+1)$-dimensional in $\R^{n+1}$, $1\le k\le n$, with $k=n$ corresponding to the symmetric decreasing rearrangement); polarization with respect to a hyperplane, defined by \eqref{pol} below \cite[Proposition~1.37]{Baer}, \cite[Lemma~5.1]{BS}; and others besides, such as the SC 1-symmetrizations \cite[Lemma~4.4 and~Definition~4.4]{Sol1} and their generalizations \cite[Section~9]{Sol1}, which we shall call {\em Solynin rearrangements}.

The P\'olya-Szeg\H{o} inequality \eqref{intro1} holds for each of the just-mentioned rearrangements.  (For the symmetric decreasing rearrangement, references were provided above, while proofs for the Schwarz rearrangement, polarization (when \eqref{intro1} becomes an equality), and Solynin rearrangements can be found in \cite[Theorem~6.19]{Baer}, \cite[Lemma~5.3]{BS}, and \cite[Theorem~10.2]{Sol1}, respectively.)  One of the main purposes of this paper is to prove \eqref{intro1} for all smoothing rearrangements on $\cV(\R^n)$; see Corollary~\ref{polyaszego_w12}.  The initial goal in this direction, achieved in Theorem~\ref{polyaszego_lip}, is to show that if
$T:\cV(\R^n)\to \cV(\R^n)$ is a smoothing rearrangement, $\Phi:[0,\infty)\to[0,\infty]$ is left-continuous and convex with $\Phi(0)=0$ (i.e., what we call a Young function), and $f\in \cV(\R^n)$ is Lipschitz, then
\begin{equation}\label{intro4}
\int_{\R^n}\Phi\left(\|\nabla Tf(x)\|\right)\, dx\leq \int_{\R^n}\Phi\left(\|\nabla f(x)\|\right)\, dx,
\end{equation}
where the integrals may be infinite. (For real-valued $\Phi$, this result for the special rearrangements discussed above can be found in \cite[Theorem~3.11]{Baer}, \cite[Theorem~6.16]{Baer}, \cite[Proposition~3.12]{Baer}, with polarization again giving an equality even when $\Phi:[0,\infty)\to[0,\infty)$ is an arbitrary Borel function, and \cite[Theorem~10.4]{Sol1}, respectively.)  The core of the proof of \eqref{intro4} is an inequality (see \eqref{introcore} below) between the $(n+1)$-dimensional measure of the Minkowski sum of the part of the subgraph $K_{Tf}$ of $Tf$ above a fixed height and a suitable convex body $C\subset \R^{n+1}$, on the one hand, and the corresponding quantity for the subgraph $K_f$ of $f$, on the other.  (The inequality follows from a containment relation in Lemma~\ref{lemdec303} between horizontal sections of these two sets, that comes from \eqref{intro2} and other properties of $T$.) This yields an inequality between the (upper) outer Minkowski contents of the two sets (Lemma~\ref{minkcont_diminishes}). Some results from geometric measure theory, in particular a formula of Lussardi and Villa \cite{LV}, allow us to express this inequality in terms of integrals over the graphs of $Tf$ and of $f$ of the support function of $C$ of the outer unit normal (Lemma~\ref{minkcontequalhausdorff}). The last main step is to prove that $C$ can be chosen so that it represents $\Phi$ (Lemma~\ref{lemoct20}), i.e., so that the mentioned inequalities transform into \eqref{intro4}. We also use the McShane-Whitney extension theorem for Lipschitz functions (Lemma~\ref{cor1may21}).

In Theorem~\ref{polyaszego_w12 new}, we present a far-reaching version of \eqref{intro4} in $W^{1,1}_{loc}(\R^n)$.  Specifically, we show that if $T:\cV(\R^n)\to \cV(\R^n)$ is a smoothing rearrangement, $\Phi$ is a Young function, $f\in W^{1,1}_{loc}(\R^n)\cap\cV(\R^n)$, and $\int_{\R^n}\Phi(\|\nabla f(x)\|)\, dx<\infty$, then $Tf\in W^{1,1}_{loc}(\R^n)$ and \eqref{intro4} holds.  This generalizes the results for Schwarz and Solynin rearrangements in \cite[Theorem~8.3]{BS} and \cite[Theorem~10.4]{Sol1}.  The passage from Theorem~\ref{polyaszego_lip} to Theorem~\ref{polyaszego_w12 new} requires overcoming some technical difficulties, made all the more challenging because we do not assume that $\Phi$ is an N-function.  In particular, we approximate $\Phi$ by a real-valued Young function $\Phi_r$ such that the Orlicz space $L^{\Phi_r}(\R^n)$ is equivalent to $L^1(\R^n)+L^\infty(\R^n)$, the largest Orlicz space, and use both the necessary and the sufficient condition of the so-called de La Vall\'ee-Poussin criterion. The necessary background on Orlicz spaces is provided at the beginning of Section~6.  Since $W^{1,p}(\R^n)\subset W^{1,1}_{loc}(\R^n)$ for $1\le p<\infty$, Theorem~\ref{polyaszego_w12 new} immediately yields Corollary~\ref{polyaszego_w12}, the classical version \eqref{intro1} of the P\'olya-Szeg\H{o} inequality, but now for every smoothing rearrangement.

Finally, anisotropic versions of Theorems~\ref{polyaszego_lip} and~\ref{polyaszego_w12 new}, in which the role of the unit ball $B^n$ is replaced by a convex body $K\subset \R^n$ containing the origin in its interior, are proved in Theorems~\ref{anisotropic polyaszego_lip} and~\ref{AnisotropicW11}.  Here the rearrangement $T:\cV(\R^n)\to \cV(\R^n)$ is assumed to be $K$-smoothing (i.e., \eqref{intro2} holds with $B^n$ replaced by $K$), and then, with $\Phi$ and $f$ as in Theorem~\ref{polyaszego_lip} or Theorem~\ref{polyaszego_w12 new}, respectively, the conclusion is that
\begin{equation}\label{intro5}
\int_{\R^n}\Phi\left(h_{-K}(\nabla Tf(x))\right)\, dx\leq \int_{\R^n}\Phi\left(h_{-K}(\nabla f(x))\right)\, dx,
\end{equation}
where $h_K$ is the support function of $K$.  When $K=B^n$, \eqref{intro5} becomes \eqref{intro4}.  This type of P\'olya-Szeg\H{o} inequality was introduced by Alvino, Ferone, Trombetti, and Lions \cite[Theorem~3.1]{AFTL} when $K$ is $o$-symmetric and $\Phi(t)=t^p$, $p\ge 1$, but for all $f\in W^{1,p}_0(\R^n)$.  In their anisotropic framework, Schwarz symmetrization is replaced by one they call convex symmetrization; in the potter's wheel description above, each horizontal slice of clay would be molded into a dilate of the convex body $K$.  See Example~\ref{exaug61}(ii) below, where we call the process when it is extended to functions a {\em $K$-Schwarz rearrangement} and note that it is $K$-smoothing.  These concepts, which align with that of Wulff shape in crystallography (see, e.g., \cite[Section~7.5]{Sch14}), were generalized by Van Schaftingen in a process he calls partial anisotropic symmetrization in \cite{VS2006}, where he proves a corresponding generalization of \cite[Theorem~3.1]{AFTL}.  In the rearrangements resulting from partial anisotropic symmetrizations, which we may consistently also call $K$-Schwarz rearrangements, $K$ is a $k$-dimensional convex body in $\R^k$ and the axis of the potter's wheel is $(n-k+1)$-dimensional in $\R^{n+1}$, $1\le k\le n$.  We also introduce in \eqref{eqmay255} below the $K$-modulus of continuity of a function, and show in Theorem~\ref{lemdec201} that when $X=\cV(\R^n)$, for example, a rearrangement $T: X\to X$ that reduces the $K$-modulus of continuity of functions in $X$ is $K$-smoothing.  The converse is true when $K$ is $o$-symmetric (see Theorem~\ref{lemdec194}), but Example~\ref{exmay205} shows that the $K$-Schwarz rearrangement does not generally reduce the $K$-modulus of continuity of functions in $X$ when $K$ is not $o$-symmetric.

Different anisotropic extensions of the P\'olya-Szeg\H{o} inequality were found by Klimov \cite{Kli99} and Van Schaftingen \cite{VS2006}.  Our methods can be used to prove Klimov's inequality for the $K$-Schwarz rearrangement of symmetrizable functions.

Known proofs of \eqref{intro1} and its variants seem to follow one of two approaches.  The first, adopted in the present paper, proceeds via isoperimetric inequalities applied to (super-) level sets, while the second uses approximation by special rearrangements, principally polarizations.  The second approach does not provide information about the cases of equality and moreover does not help with the anisotropic case, but otherwise can be extremely efficient.  The standard {\em polarization} process, sometimes called {\em two-point symmetrization}, with respect to an {\em oriented} $(n-1)$-dimensional (linear) subspace $H$, takes a function $f:\R^n\to \R$ and replaces it by
\begin{equation}\label{pol}
P_Hf(x)=\begin{cases}
\max\{f(x),f(x^{\dagger})\},& {\text{if $x\in H^+$,}}\\
\min\{f(x),f(x^{\dagger})\},& {\text{if $x\in H^-$}},
\end{cases}
\end{equation}
where $^{\dagger}$ denotes the reflection in $H$ and where $H^+$, $H^-$, are the two closed half-spaces bounded by $H$ and determined by its orientation. For background and references, see \cite[Introduction]{BGGK}, where it is explained in exactly what sense all Schwarz rearrangements, including the symmetric decreasing rearrangement, can be approximated by polarizations, a result due to Brock and Solynin \cite{BS} and refined by Van Schaftingen \cite{VS1, VS2}. Solynin \cite[Lemmas~7.4 and~9.2]{Sol1} proved that his rearrangements can also be approximated by polarizations.  This raises a natural question:  Can all smoothing rearrangements be so approximated by those in a subclass for which \eqref{intro1} is easy to prove?  We have begun to study this question but do not yet have an answer.  Another obvious question is whether the smoothing or $K$-smoothing assumptions are necessary for our P\'olya-Szeg\H{o} inequalities, though Example~\ref{exaug721} shows that this is not the case when $p=1$.  (Example~\ref{exaug721} also shows that the smoothing assumption cannot generally be omitted in Theorems~\ref{polyaszego_lip} and~\ref{polyaszego_w12 new} and Corollary~\ref{polyaszego_w12}.)  Also left for a future investigation are the cases of equality.  Even for the symmetric decreasing rearrangement, this is challenging; see \cite{BF} and the references given there, which go back to the initial study of Brothers and Ziemer \cite{BZ}.

When $p=2$, the P\'olya-Szeg\H{o} inequality \eqref{intro1} can be derived from the Riesz-Sobolev inequality (also called the Riesz rearrangement inequality) for the symmetric decreasing rearrangement \cite[Theorem~8.4]{Baer}, \cite[Theorem~3.7]{LL}; for a proof of this fact, see \cite[Lemma~7.17]{LL}.  This inequality holds for all Schwarz rearrangements but not for polarizations (see \cite[Corollary~4.3]{VS2006}), and therefore is not true for all smoothing rearrangements.  On the other hand, it is pointed out in \cite[p.~1763]{BS} that a very special case of the Riesz-Sobolev inequality, the Hardy-Littlewood inequality \cite[p.~54]{Baer}, \cite[Theorem~3.4]{LL}, is a simple consequence of Proposition~\ref{contraction} below, and therefore holds for all rearrangements.

\section{Preliminaries}\label{subsec:notations}

As usual, $S^{n-1}$ denotes the unit sphere and $o$ the origin in Euclidean $n$-space $\R^n$.  Unless stated otherwise, we assume throughout that $n\ge 2$.   The standard orthonormal basis for $\R^n$ is $\{e_1,\dots,e_n\}$ and the Euclidean norm is denoted by $\|\cdot\|$.  The term {\em ball} in $\R^n$ will always mean a {\em closed} $n$-dimensional ball unless otherwise stated.  The unit ball in $\R^n$ will be denoted by $B^n$ and $B(x,r)$ is the ball with center $x$ and radius $r$.  We write $D^n$ for the open unit ball in $\R^n$. If $x,y\in \R^n$ we write $x\cdot y$ for the inner product and $[x,y]$ for the line segment with endpoints $x$ and $y$. If $x\in \R^n\setminus\{o\}$, then $x^{\perp}$ is the $(n-1)$-dimensional subspace orthogonal to $x$ and $\langle x\rangle$ is the 1-dimensional subspace spanned by $x$. Throughout the paper, the term {\em subspace} means a linear subspace.

If $A$ is a set,  we denote by $\cl A$, $\inte A$, and $\dim A$ the {\it closure}, {\it interior}, and {\it dimension} (that is, the dimension of the affine hull) of $A$, respectively.  If $H$ is a subspace of $\R^n$, then $A|H$ is the (orthogonal) projection of $A$ on $H$ and $x|H$ is the projection of a vector $x\in \R^n$ on $H$.

If $A$ and $B$ are sets in $\R^n$ and $t\in \R$, then we denote by $tA=\{tx:x\in A\}$ the {\em dilate} of $A$ by the factor $t$, and by
$$A+B=\{x+y: x\in A, y\in B\}$$
the {\em Minkowski sum} of $A$ and $B$.  We write $-A=(-1)A$ for the reflection of $A$ in the origin and call $A$ {\em origin symmetric} or {\it $o$-symmetric} if $-A=A$.

We write ${\mathcal H}^k$ for $k$-dimensional Hausdorff measure in $\R^n$, where $k\in\{1,\dots, n\}$.  When dealing with relationships between sets in $\R^n$ or functions on $\R^n$, the term {\em essentially} means up to a set of $\cH^n$-measure zero.

The Grassmannian of $k$-dimensional subspaces in $\R^n$ is denoted by ${\mathcal{G}}(n,k)$.

We denote by ${\mathcal C}^n$, ${\mathcal G}^n$, ${\mathcal B}^n$, ${\mathcal M}^n$, and ${\mathcal L}^n$ the class of nonempty compact sets, open sets, bounded Borel sets, ${\mathcal H}^n$-measurable sets, and ${\mathcal H}^n$-measurable sets of finite ${\mathcal H}^n$-measure, respectively, in $\R^n$.

Let ${\mathcal K}^n$ be the class of nonempty compact convex subsets of $\R^n$ and let ${\mathcal K}^n_n$ be the class of {\em convex bodies}, i.e., members of ${\mathcal K}^n$ with interior points.  We write ${\mathcal K}^n_{(o)}$ for subclass of ${\mathcal K}^n_n$ whose members contain the origin in their interiors. If $K\in {\mathcal K}^n$, then
\begin{equation}\label{support}
h_K(x)=\sup\{x\cdot y: y\in K\}
\end{equation}
for $x\in\R^n$, defines the {\it support function} $h_K$ of $K$.  The texts by Gruber \cite{Gru07} and Schneider \cite{Sch14} contain a wealth of useful information about convex sets and related concepts such as the {\em intrinsic volumes} $V_j$, $j\in\{1,\dots, n\}$ (see also \cite[Appendix~A]{Gar06}).  In particular, if $K\in {\mathcal K}^n$ and  $\dim K=n$ then $2V_{n-1}(K)$ is the {\em surface area} of $K$. If $\dim K=k$, then $V_k(K)={\mathcal H}^k(K)$ is the {\em volume} of $K$.  By $\kappa_n$ we denote the volume ${\mathcal H}^n(B^n)$ of the unit ball in $\R^n$.

If $K\in{\mathcal K}^n_{(o)}$, the {\em polar body} $K^\circ$ of $K$ is defined by
\begin{equation}\label{polarbody}
K^\circ=\{x\in\R^n: x\cdot y\leq1 \ \text{for}\  y\in K \}.
\end{equation}
Then $(K^\circ)^\circ=K$ and (see \cite[(1.52), p.~57]{Sch14})
\begin{equation} \label{polarrad}
\rho_K(x) h_{K^\circ}(x)=h_K (x) \rho_{K^\circ}(x)=1\ \ \ \mathrm{for~} x\in \R^n\setminus\{o\},
\end{equation}
where
\begin{equation} \label{radialfunction}
\rho_K(x)=\max\{\lambda\ge 0: \lambda x\in K\}
\end{equation}
for $x\in \R^n\setminus\{o\}$, is the {\em radial function} of $K$.  We shall also find use for the {\em gauge function} of $K$, defined by
\begin{equation} \label{gaugefunction}
\|x\|_K=\inf\{\lambda \ge 0: x\in \lambda K\}=h_{K^{\circ}}(x)
\end{equation}
for $x\in \R^n$.  The previous equality follows from \eqref{polarrad} and \eqref{radialfunction}, or see \cite[Lemma~1.7.13]{Sch14}.  Despite the notation, $\|\cdot\|_K$ is a norm if and only if $K$ is $o$-symmetric; in general it is sublinear but does not satisfy $\|-x\|_K=\|x\|_K$ for all $x\in \R^n$.  When $K=B^n$, $\|\cdot\|_K$ is the Euclidean norm.

It will be convenient to call a function $f:\R^n\to \R$ a {\em $K$-contraction} if
\begin{equation}\label{Kcontraction}
|f(x)-f(y)|\le \|x-y\|_K
\end{equation}
for all $x,y\in \R^n$.  Note that when $K=B^n$, a $K$-contraction is a contraction in the usual sense of the term.  Note also that $f$ is a $K$-contraction if and only if it is a $-K$-contraction, since \eqref{Kcontraction} is equivalent to $|f(y)-f(x)|\le \|x-y\|_K$ for all $x,y\in \R^n$, and hence to $|f(x)-f(y)|\le \|y-x\|_K=\|x-y\|_{-K}$ for all $x,y\in \R^n$.  Clearly, every $K$-contraction is continuous.

From \eqref{support} and \eqref{polarrad}, it is easy to see that
\begin{equation}\label{CS_polar}
x\cdot y\leq h_{K}(x)\, h_{K^\circ}(y)
\end{equation}
for $x,y\in \R^n$ (see \cite[(1.40), p.~54]{Sch14}) and that equality holds when $x,y\neq o$ if and only if $x$ is an outer normal to $K$ at $\rho_K(y)y=y/h_{K^\circ}(y)\in \partial K$.

Given $A\in\cM^n$, let $\uomc(A)$ and $\lomc(A)$ denote, respectively, its {\em upper} and {\em lower outer Minkowski content}, i.e.,
\begin{equation}\label{eqapr31}
\uomc(A)=\limsup_{\ee\to0+} \frac{\cH^n( A+\ee B^n)-\cH^n(A)}{\ee}\quad\text{and}\quad
\lomc(A)=\liminf_{\ee\to0+} \frac{\cH^n( A+\ee B^n)-\cH^n(A)}{\ee}.
\end{equation}
See \cite[p.~69]{BurZ80} and \cite{ACV}, whose notation and terminology differs from ours, and note that the limits in \eqref{eqapr31} are unchanged if $B^n$ is replaced by $D^n$.  We shall also need the following generalization of these concepts.  If $C\in {\mathcal K}^n_{(o)}$ and $A\in\cM^n$, let $\uomcC(A)$ and $\lomcC(A)$ denote, respectively, the {\em upper} and {\em lower anisotropic outer Minkowski content} of $A$ with respect to $C$, obtained by replacing $B^n$ in \eqref{eqapr31} by $C$.  When the two limits coincide we denote them by $\omcC(A)$, and again, the limits are unchanged if $C$ is replaced by $\inte C$.

Let $A\in\cM^n$.  We shall write $S(A)$ for the {\em perimeter} of $A$. For the definition of this widely-used term, see, for example, \cite[p.~170]{EG}, \cite[p.~107]{KP}, \cite[p.~122]{Mag}, or \cite[p.~34]{Pfe}.  When $K$ is a convex body, its perimeter is equal to its surface area, defined in that case as
$$
S(K)=\lim_{\ee\to0+} \frac{\cH^n( K+\ee B^n)-\cH^n(K)}{\ee},
$$
its outer Minkowski content.  It is for this reason that we prefer not to use the more common $P(A)$ for the perimeter of $A$.

Let ${\mathcal{M}}(\R^n)$ (or ${\mathcal{M}}_+(\R^n)$) denote the set of real-valued (or nonnegative, respectively) measurable functions on $\R^n$ and let ${\mathcal{S}}(\R^n)$ denote the set of functions $f$ in ${\mathcal{M}}(\R^n)$ such that ${\mathcal{H}}^{n}(\{x:f(x)>t\})<\infty$ for $t>\essinf f$.  By ${\mathcal{V}}(\R^n)$, we denote the set of functions $f$ in ${\mathcal{M}}_+(\R^n)$ such that ${\mathcal{H}}^{n}(\{x:f(x)>t\})<\infty$ for $t>0$.  The four classes of functions satisfy
${\mathcal{V}}(\R^n)\subset {\mathcal{S}}(\R^n)\subset {\mathcal{M}}(\R^n)$ and ${\mathcal{V}}(\R^n)\subset {\mathcal{M}}_+(\R^n)\subset {\mathcal{M}}(\R^n)$.  Members of ${\mathcal{S}}(\R^n)$ have been called {\em symmetrizable} (see, e.g., \cite{BS}) and those of ${\mathcal{V}}(\R^n)$ are often said to {\em vanish at infinity}.  Note that the constant functions are symmetrizable but do not vanish at infinity unless they are identically zero.

We shall define a {\em Young function} as a left-continuous and convex function $\Phi:[0,\infty)\to[0,\infty]$ such that $\Phi(0)=0$, and say that such a function is {\em nontrivial} if $\Phi\not\equiv 0$ and $\Phi\not\equiv\infty$ on $(0,\infty)$.  Note that a real-valued Young function is both continuous and increasing (which will always mean non-decreasing in this paper). In \cite[Definition~2.1.1]{ES}, the term {\em Orlicz function} is used for a nontrivial Young function.  Both terms have other definitions in the literature.

Our notation for Sobolev spaces such as $W^{1,p}(\R^n)$ is standard. Definitions can be found in many texts, such as \cite{LL}.

If $f\in {\mathcal{M}}(\R^n)$, we denote its {\em graph} by $G_f$ and define its {\em subgraph} $K_f\subset \R^{n+1}$ by
\begin{equation}\label{eqdec304}
K_f=\{(x,t)\in \R^n\times \R: f(x)\ge t\}.
\end{equation}

If $T:X\to X$, where $X$ is one of the function classes given above, we shall usually write $Tf$ instead of $T(f)$.  If $T_0,T_1:X\to X$ are maps, we say that $T_0$ is {\em essentially equal} to $T_1$ if for $f\in X$, $T_0f(x)=T_1f(x)$ for ${\mathcal{H}}^{n}$-almost all $x\in \R^n$, where the exceptional set may depend on $f$.

If $f$ is a locally integrable function on $\R^n$, define
\begin{equation}\label{eqdec141}
  f^*(x)=\lim_{r \to 0^+}\frac{1}{\cH^n(B(x,r))}\int_{B(x,r)}f(y)\, dy
\end{equation}
when the limit exists and $f^*(x)=0$ otherwise.  The limit exists and equals $f(x)$ ${\mathcal{H}}^{n}$-almost everywhere in $\R^n$, by the Lebesgue differentiation theorem (see, e.g., \cite[Proposition~3.5.4]{KP}).  Evans and Gariepy \cite[p.~46]{EG} call $f^*$ the {\em precise representative} of $f$. If $A$ is a measurable set,
\begin{equation}\label{eqjuly151}
\Theta(A,x)=1_A^*(x)=\lim_{r\to0^+}\frac{\cH^{n}(A\cap B(x,r))}{\cH^n(B(x,r))},
\end{equation}
is the {\em density} of $A$ at $x$, provided the limit exists.

If $A\in \cM^n$, define
$$
A^*=\{x\in \R^n: \Theta(A,x)=1_A^*(x)=1\}.
$$
Elements of $A^*$ are called {\em Lebesgue density points}, or simply density points, of $A$.  Note that $A^*=A$, essentially, by the Lebesgue density theorem (see, e.g., \cite[Theorem~1.5.2]{Pfe}).  Since it follows immediately from the definition of perimeter as a supremum of integrals of divergences (see, for example, \cite[p.~34]{Pfe}) that if two measurable sets are essentially equal, their perimeters are equal, we have
\begin{equation}\label{eqapr2}
S(A^*)=S(A)
\end{equation}
for $A\in \cM^n$.

\begin{lem}\label{lemjan61}
Let $A, B\in \cM^n$.

\noindent{\rm{(i)}} If $A\subset B$, essentially, then $A^* \subset B^*$.

\noindent{\rm{(ii)}} If $A=B$, essentially, then $A^*= B^*$.

\noindent{\rm{(iii)}} $(A^*)^*=A^*$.

\noindent{\rm{(iv)}} If $K\subset \R^n$ is a convex body, then $A^*+K=A^*+\inte K$ is open.
\end{lem}

\begin{proof}
(i) If $A\subset B$, essentially, then $1_A \leq 1_B$, essentially.  Therefore, if $x\in A^*$, then
\begin{eqnarray*}
1=1_A^*(x)&=&\lim_{r \to 0^+}\frac{1}{\cH^n(B(x,r))}\int_{B(x,r)}1_A(y)\, dy
\le  \liminf_{r \to 0^+}\frac{1}{\cH^n(B(x,r))}\int_{B(x,r)}1_B(y)\, dy\\
&\le & \limsup_{r \to 0^+}\frac{1}{\cH^n(B(x,r))}\int_{B(x,r)}1_B(y)\, dy\le 1
\end{eqnarray*}
and hence $1_B^*(x)=1$.  Therefore $A^* \subset B^*$.

Parts (ii) and (iii) follow easily.

(iv)  Let $x\in  A^*+K$ and choose $y\in A^*$ such that $x\in y+K$. Choose $r>0$ and an open cone $C$ with vertex at $o$ such that $x+(C\cap rD^n)\subset y+\inte K$.  Note that if $w\in C\cap rD^n$, then $x\in y-w+\inte K$.  Since ${\cH}^n(-C\cap rD^n)>0$, $y\in A^*$, and $A=A^*$, essentially, there is a $w_0\in C\cap rD^n$ such that $y-w_0\in A^*$.  Hence $x\in y-w_0+\inte K\subset A^*+\inte K$ and it follows from the definition of $A^*$ that $A^*+\inte K$ is open since $\inte K$ is open.
\end{proof}

A function  $f\in\cM(\R^n)$ is \emph{approximately continuous} at $x\in\R^n$ if for each $\ee>0$,
\begin{equation}\label{eqjuly282}
\lim_{r\to 0+}\frac{\cH^n\left(B(x,r)\cap\{y:|f(x)-f(y)|<\ee\}\right)}{\cH^n(B(x,r))}=1.
\end{equation}
We shall use the fact that each $f\in\cM(\R^n)$ is approximately continuous at $\cH^n$-almost all $x\in \R^n$; see, for example, \cite[Theorem~3, Section~1.7.2]{EG}.

Note that a measurable characteristic function $1_A$ is approximately continuous at $x$ if and only if the limit in \eqref{eqdec141} with $f=1_A$ exists and equals $1_A(x)$, that is, if and only if either $x\in A$ and $1_A^*(x)=1$ or $x\notin A$, the limit in \eqref{eqdec141} exists, and $1_A^*(x)=0$.  This means that the set of points of approximate continuity can change even when a function is only changed on a set of measure zero.
Moreover, if $1_A$ is approximately continuous at $x$, then $1_A(x)=1_{A^*}(x)$ and $1_{A^*}$ is also approximately continuous at $x$.  Hence, the set of approximate continuity points of $1_{A^*}$ is the largest set of approximate continuity points of any $1_B$ for which $A$ and $B$ essentially coincide. In particular, if $A=B^n$, essentially, then $1_A$ is not approximately continuous at any unit vector. This precludes the possibility of finding a representative of each $f\in\cM(\R^n)$ that  is approximately continuous everywhere and agrees with $f$ for $\cH^n$-almost all $x\in \R^n$.

\section{Properties of maps}\label{Properties}

Let $i\in \{1,\dots,n-1\}$ and let $H\in {\mathcal{G}}(n,i)$ be fixed. We consider a map $\di:{\mathcal{E}}\subset{\mathcal{L}}^n\to {\mathcal{L}}^n$ and define
\begin{equation}\label{eqaug31}
\di^*A=(\di A)^*
\end{equation}
for each $A\in {\mathcal{E}}$.  We assume (here and throughout the paper) that the properties listed below hold for all $A,B\in {\mathcal{E}}$ and that the class ${\mathcal{E}}$ is appropriate for the property concerned.

\smallskip

1. ({\em Monotonic} or {\em strictly monotonic}) \quad $A\subset B\Rightarrow \di A\subset \di B$, essentially (or $\Rightarrow \di A\subset \di B$, essentially, and $A\neq B \Rightarrow \di A\neq \di B$, essentially, respectively).

2. ({\em Measure preserving}) \quad ${\mathcal H}^n(\di A)={\mathcal H}^n(A)$.

3. ({\em Maps balls to balls})\quad If $K=B(x,r)$, then $\di K=B(x',r')$, essentially.

4. ({\em Continuous from the inside}) \quad If $(A_m)$ is an increasing sequence of sets in ${\mathcal{E}}$ such that $\cup_{m\in\N}A_m\in {\mathcal{E}}$, then
$\di(\cup_{m\in\N}A_m)=\cup_{m\in\N}\, \di A_m$, essentially.

5. ({\em Continuous from the outside}) \quad If $(A_m)$ is a decreasing sequence of sets in ${\mathcal{E}}$ such that $\cap_{m\in\N}\,A_m\in {\mathcal{E}}$, then $\di(\cap_{m\in\N}A_m)=\cap_{m\in\N}\, \di A_m$, essentially.

6. ({\em Smoothing and $K$-smoothing}) \quad If $K\in {\mathcal K}^n_{(o)}$, we say that $\di$ is {\em $K$-smoothing} if whenever $d>0$,
\begin{equation}\label{eq:smothDef}
(\di^*A)+d K\subset \di^*(A+d K)=\di(A+d K),
\end{equation}
essentially, for each bounded $A\in \cE$ with $A+dK\in {\mathcal{E}}$, where $\di^*A$ is defined by \eqref{eqaug31}.  Then $\di$ is called {\em smoothing} if it is $K$-smoothing with $K=B^n$.

\smallskip

Information concerning relations between the first three properties listed above and others besides may be found in \cite[Sections~3 and~6]{BGGK}. The terms ``continuous from the inside," ``continuous from the outside," and ``smoothing" are employed by Sarvas \cite[p.~11]{Sar72}, although his definitions differ slightly from those above.

In the definition of $K$-smoothing, one can equivalently require a pointwise inclusion in \eqref{eq:smothDef}.  To see this, note that by Lemma~\ref{lemjan61}(iv) with $A$ and $K$ replaced by $\di A$ and $dK$, respectively, $(\di^*A)+dK=(\di^*A)+d\,\inte K$ is open.  Then the essential inclusion in \eqref{eq:smothDef} and parts (i) and (iii) of Lemma~\ref{lemjan61} give
$$(\di^*A)+d K=\left((\di^*A)+d K\right)^*\subset \left(\di^*(A+d K)\right)^*=\di^*(A+d K).$$	

\begin{lem}\label{lemdec302}
If $\di:{\mathcal{E}}\subset{\mathcal{L}}^n\to {\mathcal{L}}^n$ is monotonic and measure preserving, then $\di$ is continuous from the inside and from the outside.
\end{lem}

\begin{proof}
Let $(A_m)$ be an increasing sequence of sets in ${\mathcal{E}}$ such that $\cup_{m\in\N}A_m\in {\mathcal{E}}$.  Since $\di$ is monotonic, we have $\di A_m\subset \di(\cup_{m\in\N}A_m)$ for $m\in \N$, essentially, and hence $\cup_{m\in\N}\,\di A_m\subset \di(\cup_{m\in\N}A_m)$, essentially.  The continuity of measures of increasing sequences and the fact that $\di$ preserves measure yield $$\cH^n\left(\di(\cup_{m\in\N}A_m)\right)=\cH^n\left(\cup_{m\in\N}A_m\right)=
\lim_{m\to\infty}\cH^n(A_m)=\lim_{m\to\infty}\cH^n(\di A_m)=
\cH^n\left(\cup_{m\in\N}\, \di A_m\right).$$
It follows that $\di(\cup_{m\in\N}A_m)=\cup_{m\in\N}\, \di A_m$, essentially, and hence that $\di$ is continuous from the inside.

The proof of the continuity from the outside is similar.
\end{proof}

Let $X\subset {\mathcal{M}}(\R^n)$, where we assume henceforth that $X$ contains the characteristic functions of sets in ${\cL}^n$.  Let $T:X\to X$ and if $A\in \cL^n$, let
$$
\di_T A=\{x: T1_A(x)=1\}
$$
and let
\begin{equation}\label{eqdec201}
\di_T^*A=(\di_TA)^*.
\end{equation}
By Proposition~\ref{fromSetToFct}(i) below, the induced map $\di_T:\cL^n\to \cL^n$ is well defined when $X={\mathcal{M}}(\R^n)$, ${\mathcal{M}}_+(\R^n)$, ${\mathcal{S}}(\R^n)$, or $\cV(\R^n)$. Of course, $\di_T^*:\cL^n\to \cL^n$ is well defined whenever $\di_T:\cL^n\to \cL^n$ is.

If $X\subset {\mathcal{M}}(\R^n)$, we consider the following properties of a map $T:X\to X$, where the first four properties are assumed to hold for all $f,g\in X$:

\smallskip

1. ({\em Equimeasurable})\quad
$$
{\mathcal{H}}^{n}(\{x: Tf(x)>t\})={\mathcal{H}}^{n}(\{x: f(x)>t\})
$$
for $t\in \R$.

2. ({\em Monotonic}) \quad $f\le g$, essentially, implies $Tf\le Tg$, essentially.

3. ({\em $L_p$-contracting}) \quad $\|Tf-Tg\|_p\le \|f-g\|_p$ when $f-g\in L^p(\R^n)$.

4. ({\em Modulus of continuity reducing and $K$-modulus of continuity reducing})  If $d>0$ and $K\in {\mathcal K}^n_{(o)}$, we define the {\em $K$-modulus of continuity} of $f\in X$ by
\begin{equation}\label{eqmay255}
\omega_{K,\,d}(f)=\esssup_{\|x-y\|_{K}\le d}|f(x)-f(y)|=\esssup_{x-y\in dK}|f(x)-f(y)|.
\end{equation}
The equivalence of these two expressions follows easily from the left-hand equality in \eqref{gaugefunction}.  Then $T$ {\em reduces the $K$-modulus of continuity} if
$\omega_{K,d}(Tf)\le \omega_{K,d}(f)$ for all $d>0$ and $f\in X$.  When $K=B^n$, we refer simply to the {\em modulus of continuity} of $f\in X$ and drop the suffix $K$, i.e.,
$$
\omega_d(f)=\esssup_{\|x-y\|\le d}|f(x)-f(y)|,
$$
and say that $T$ {\em reduces the modulus of continuity} if $\omega_{d}(Tf)\le \omega_{d}(f)$ for all $d>0$ and $f\in X$.

5. ({\em Continuous on the inside (or outside)}) \quad  The induced map $\di_T$ is well defined on ${\cL}^n$ and continuous from the inside (or outside, respectively) when ${\mathcal{E}}={\cL}^n$.

6. ({\em Smoothing and $K$-smoothing}) \quad If $K\in {\mathcal K}^n_{(o)}$, we say that $T$ is {\em $K$-smoothing} if the induced map $\di_T$ is well defined on ${\cL}^n$ and $K$-smoothing when ${\mathcal{E}}={\cL}^n$, i.e.,
\begin{equation}\label{eqjuly312}
(\di_T^*A)+d K\subset \di_T^*(A+d K)=\di_T(A+d K),
\end{equation}
essentially, for each $d>0$ and bounded $A\in \cM^n$.  Then $T$ is called {\em smoothing} if it is $K$-smoothing with $K=B^n$.

\smallskip

The map $T$ is called a {\em rearrangement} if it is equimeasurable and monotonic.

In their somewhat different setting, versions of Properties~5 and~6 (for $K=B^n$) were also considered by Brock and Solynin \cite[p.~1764]{BS}.  In particular, their definition of a smoothing rearrangement $T:{\mathcal{S}}(\R^n)\to {\mathcal{S}}(\R^n)$ corresponds to requiring $(\di_TA)+dD^n\subset \di_T(A+dD^n)$, for each $d>0$ and $A\in \cL^n$. However, $A+dD^n\not\in \cL^n$, in general, when $A\in \cL^n$.  Moreover, their definition is sensitive to changing $T$ on a set of ${\cH}^n$-measure zero.  For example, if $T_0f=f$ is the identity map and $T_1f=\max\{f,1_{\Q^n}\}$, then $T_0=T_1$, essentially, while $\di_{T_0} A=A$ and $\di_{T_1}A=A\cup \Q^n$ implies that $T_0$ is smoothing but $T_1$ is not under their definition.

Our definitions of smoothing and $K$-smoothing are examined further in Lemma~\ref{lemapr3} below.  See also the remarks at the beginning of Section~\ref{Smoothing rearrangements}.

For the convenience of the reader, we now state five results proved in \cite{BGGK} as Lemmas~4.1, 4.5, 4.7, Theorem~4.8 and the remarks that follow it, and Theorem~4.9, respectively.

\begin{prop}\label{may8lem}
\noindent{\rm{(i)}}  If $T:{\mathcal{S}}(\R^n)\to {\mathcal{S}}(\R^n)$ is equimeasurable, then $\essinf Tf=\essinf f$ for $f\in {\mathcal{S}}(\R^n)$.

\noindent{\rm{(ii)}} If $T:{\mathcal{M}}(\R^n)\to {\mathcal{M}}(\R^n)$ is a rearrangement, then $\essinf Tf\ge \essinf f$ for $f\in {\mathcal{M}}(\R^n)$.   Hence, $T:{\mathcal{S}}(\R^n)\to {\mathcal{S}}(\R^n)$.

\noindent{\rm{(iii)}} In either case, $T:\cV(\R^n)\to \cV(\R^n)$ and $T$ is essentially the identity on constant functions.
\end{prop}

\begin{prop}\label{fromSetToFct}
Let $X={\mathcal{M}}(\R^n)$, ${\mathcal{M}}_+(\R^n)$, ${\mathcal{S}}(\R^n)$, or $\cV(\R^n)$, and let $T:X\to X$ be equimeasurable.

\noindent{\rm{(i)}} The induced map $\di_T:\cL^n\to \cL^n$ given by
\begin{equation}\label{eqIndic}
\di_T A=\{x: T1_A(x)=1\}
\end{equation}
for $A\in\cL^n$ is well defined and measure preserving.

\noindent{\rm{(ii)}} If $X={\mathcal{M}}_+(\R^n)$, ${\mathcal{S}}(\R^n)$, or $\cV(\R^n)$, then $T$ essentially maps characteristic functions of sets in $\cL^n$ to characteristic functions of sets in $\cL^n$, in the sense that for each $A\in \cL^n$,
\begin{equation}\label{eqoct72}
T1_A=1_{\di_TA},
\end{equation}
essentially.
\end{prop}

\begin{prop}\label{lemoct6}
Let $X={\mathcal{S}}(\R^n)$ or $\cV(\R^n)$ and let $T:X\to X$ be a rearrangement.  For $X={\mathcal{S}}(\R^n)$, $A\in \cL^n$, and $\alpha,\beta\in \R$ with $\alpha\ge 0$, we have
\begin{equation}\label{eqoct61}
T(\alpha 1_A+\beta)=\alpha\, T1_A+\beta,
\end{equation}
essentially.  When $X=\cV(\R^n)$, \eqref{eqoct61} holds, essentially, if $\beta=0$.
\end{prop}

\begin{prop}\label{lemapril30}
Let $X={\mathcal{M}}(\R^n)$, ${\mathcal{M}}_+(\R^n)$, ${\mathcal{S}}(\R^n)$, or $\cV(\R^n)$ and let $T:X\to X$ be a rearrangement.

\noindent{\rm{(i)}}  The map  $\di_T:\cL^n\to \cL^n$ defined by \eqref{eqIndic} is monotonic.

\noindent{\rm{(ii)}} If $X={\mathcal{S}}(\R^n)$ or $\cV(\R^n)$ and $f\in X$, then
\begin{equation}\label{eqnov11}
\{x: Tf(x)\ge t\}=\di_T\{x: f(x)\ge t\}\quad{\text{and}}\quad\{x: Tf(x)> t\}=\di_T\{x: f(x)> t\},
\end{equation}
essentially, for $t>\essinf f$.  Moreover, $T$ is essentially determined by $\di_T$, since
\begin{equation}\label{eqoct62}
Tf(x)=\max\left\{\sup\{t\in \Q,\, t>\essinf f: x\in \di_T \{z: f(z)\ge t\}\},\essinf f\right\},
\end{equation}
essentially.
\end{prop}

\begin{prop}\label{propoct14}
Let $T:{\mathcal{S}}(\R^n)\to {\mathcal{S}}(\R^n)$ be a rearrangement and let $f\in {\mathcal{S}}(\R^n)$.  If $\varphi:\R\to\R$ is right-continuous and increasing (i.e., non-decreasing), then $\varphi\circ f\in {\mathcal{S}}(\R^n)$ and
\begin{align}\label{eqoct149}
\varphi(Tf)=T(\varphi\circ f),
\end{align}
essentially.
\end{prop}

The following result was first proved, without the assumption that $j$ is nonnegative, by Crowe, Zweibel, and Rosenbloom \cite{CZR} for Schwarz rearrangement.  Versions of it have been stated for general rearrangements in \cite[Theorem~3.1]{BS}, \cite[Proposition~3.3.9]{VSPhD}, and \cite[Corollary~1]{VSW}; however, these works take a different approach to rearrangements, so we provide a proof and brief commentary in the Appendix.

\begin{prop}\label{contraction}
Let $j:\R\to [0,\infty)$ be convex with $j(0)=0$. If $T:\cV(\R^n)\to \cV(\R^n)$ is a rearrangement, then
\begin{equation}\label{eq_contraction}
\int_{\R^n}j(Tf(x)-Tg(x))\, dx\leq \int_{\R^n}j(f(x)-g(x))\, dx
\end{equation}
for $f,g\in\cV(\R^n)$ such that either integral exists.  In particular, $T$ has the $L^p$-contracting property.
\end{prop}

\begin{lem}\label{lemaug11}
Let $X={\mathcal{M}}(\R^n)$, ${\mathcal{M}}_+(\R^n)$, ${\mathcal{S}}(\R^n)$, or $\cV(\R^n)$.  If $T:X\to X$ is a rearrangement, then the maps $\di_T, \di_T^*:{\cL}^n\to {\cL}^n$ are well defined, measure preserving, and monotonic (pointwise monotonic, in the case of $\di_T^*$, i.e., $A\subset B$ $\Rightarrow$ $\di_T^*A\subset \di_T^*B$). Moreover, $T$ and $\di_T^*$ are continuous from the inside and from the outside.
\end{lem}

\begin{proof}
The induced map $\di_T:{\cL}^n\to {\cL}^n$ is well defined, measure preserving, and monotonic by Propositions~\ref{fromSetToFct}(i) and~\ref{lemapril30}(i).  It follows that $\di_T^*:{\cL}^n\to {\cL}^n$ is well defined and measure preserving by the Lebesgue density theorem, and pointwise monotonic by Lemma~\ref{lemjan61}(i).

Lemma~\ref{lemdec302} with ${\mathcal{E}}={\cL}^n$ and $\di$ replaced by $\di_T$ shows that $T$ is continuous from the inside and from the outside.  The fact that $\di_T^*$ is also continuous from the inside and from the outside is then an easy consequence of the Lebesgue density theorem.
\end{proof}

It is convenient to state the following lemma for the induced maps $\di_T$ of a rearrangement $T$, but it holds more generally for any monotonic map $\di:{\mathcal{E}}\subset{\mathcal{L}}^n\to {\mathcal{L}}^n$ and $A, B\in {\mathcal{E}}$.

\begin{lem}\label{lemjan63}
Let $X={\mathcal{M}}(\R^n)$, ${\mathcal{M}}_+(\R^n)$, ${\mathcal{S}}(\R^n)$, or $\cV(\R^n)$, let $T:X\to X$ be a rearrangement, and let $A, B\in {\cL}^n$.

\noindent{\rm{(i)}} If $A\subset B$, essentially, then $\di^*_TA\subset \di^*_TB$.

\noindent{\rm{(ii)}} If $A=B$, essentially, then $\di^*_TA=\di^*_TB$.

\noindent{\rm{(iii)}} $\di^*_TA=\di^*_TA^*$.

\noindent{\rm{(iv)}} $\di_T^*A=\di_TA=\di_T A^*$, essentially.

\noindent{\rm{(v)}} For $f\in X$ and $s\ge t> \essinf f$, we have $\di_T^*\{z:f(z)\geq s\} \subset \di_T^*\{z:f(z)\geq t\}$.
\end{lem}

\begin{proof}
(i) Since $\di_T$ is monotonic by Lemma~\ref{lemaug11}, we have $\di_TA\subset \di_TB$, essentially. The conclusion follows from \eqref{eqdec201} and Lemma~\ref{lemjan61}(i).

Parts (ii) and (iii) follow easily, the latter using the fact that $A=A^*$, essentially.  The latter equality and the monotonicity of $\di_T$ yield the second equality in (iv), while the first is a consequence of \eqref{eqdec201} and $\di_T A=(\di_T A)^*$, essentially.

Part (v) follows from (i) and the fact that $\{z:f(z)\geq s\} \subset \{z:f(z)\geq t\}$ for $s\ge t> \essinf f$.
\end{proof}

Let $X={\mathcal{S}}(\R^n)$ or $\cV(\R^n)$ and let $T:X\to X$ be a rearrangement. Using \eqref{eqoct62}, the Lebesgue density theorem and the fact that the supremum is over a countable set of values, and Lemma~\ref{lemjan63}(v), we obtain
\begin{eqnarray}\label{eqjan72}
Tf(x)&=&\max\left\{\sup\{t\in \Q,\, t>\essinf f: x\in \di_T \{z: f(z)\ge t\}\},\essinf f\right\}\nonumber \\
&=&\max\left\{\sup\{t\in \Q,\, t>\essinf f: x\in \di_T^* \{z: f(z)\ge t\}\},\essinf f\right\}\nonumber \\
&=& \max\left\{\sup\{t\in \R,\, t>\essinf f: x\in \di_T^* \{z: f(z)\ge t\}\},\essinf f\right\},
\end{eqnarray}
essentially.  This shows that by substituting $\di^*_T$ for $\di_T$ in \eqref{eqoct62}, we may take the supremum over $\R$ and thus bring the formula into line with those in \cite{BS} and \cite{VSW}; see the discussion in \cite[Appendix]{BGGK}.

\section{Smoothing rearrangements and reduction of the modulus of continuity}\label{Smoothing rearrangements}

Before embarking on the main goal of this section, we prove the inequalities \eqref{eqapr33} below for smoothing rearrangements.  These will not be needed for the sequel, but seem interesting and follow fairly easily from what we know so far.

\begin{lem}\label{lemjuly312}
Suppose that ${\cK}_n^n\subset {\mathcal{E}}\subset{\mathcal{L}}^n$ and that $\di:{\mathcal{E}}\to {\mathcal{L}}^n$ is measure preserving and smoothing.  Then $S(\di K)\le S(K)$ for $K\in {\cK}_n^n$ and hence $\di$ maps balls to balls.
\end{lem}

\begin{proof}
Let $K\in {\cK}_n^n$. For $\ee>0$, the assumed properties of $\di$ imply that
$$\cH^n((\di^* K)+\ee B^n)\le \cH^n(\di(K+\ee B^n))= \cH^n(K+\ee B^n)$$
and $\cH^n(\di^* K)=\cH^n(K)$. It follows, using \eqref{eqapr2} with $A=\di K$ and the relation between the (lower) outer Minkowski content (defined by \eqref{eqapr31}) and perimeter \cite[Theorem~14.2.1]{BurZ80}, that
\begin{eqnarray*}
S(\di K)=S(\di^* K)\le \lomc(\di^* K) &=& \liminf_{\ee\to 0+}\frac{\cH^n((\di^* K)+\ee B^n)-\cH^n(\di^* K)}{\ee}\\
&\le& \lim_{\ee\to 0+}\frac{\cH^n(K+\ee B^n)-\cH^n(K)}{\ee}=S(K).
\end{eqnarray*}
Consequently, $\di K$ has finite perimeter.  The fact that $\di$ maps balls to balls is now a direct consequence of the isoperimetric inequality for sets in $\cL^n$ of finite perimeter and its equality condition (see \cite[p.~165]{Mag}).
\end{proof}

Recall that $\uomc(A)$ and $\lomc(A)$ are the upper and lower outer Minkowski content of $A$, defined by \eqref{eqapr31}.

\begin{thm}\label{lemdec28}
Let $X={\mathcal{M}}(\R^n)$, ${\mathcal{M}}_+(\R^n)$, ${\mathcal{S}}(\R^n)$, or $\cV(\R^n)$ and suppose that $T:X\to X$ is a rearrangement.  If $T$ is smoothing, then
\begin{equation}\label{eqapr33}
\uomc(\di^*_T A)\le \uomc(A^*)\quad{\text{and}}\quad\lomc(\di^*_T A)\le \lomc(A^*)
\end{equation}
for bounded $A\in\cM^n$.  Moreover, $S(\di_T K)\le S(K)$ for $K\in {\cK}_n^n$ and hence $\di_T$ maps balls to balls.
\end{thm}

\begin{proof}
Let $A\in\cM^n$ be bounded and let $\ee>0$.  Then the fact that $\di^*_T A=\di^*_T A^*$ by Lemma~\ref{lemjan63}(iii), \eqref{eqjuly312} with $A$ replaced by $A^*$, Lemma~\ref{lemjan63}(iv) with $A$ replaced by $A^*+\ee B^n$, and the equimeasurability of $T$ imply that
\begin{equation}\label{eqdec291}
\cH^n((\di_T^*A)+\ee B^n)\le \cH^n(\di_T^*(A^*+\ee B^n))=\cH^n(\di_T(A^*+\ee B^n))=\cH^n(A^*+\ee B^n)
\end{equation}
and
\begin{equation}\label{eqdec292}
\cH^n(\di^*_TA)=\cH^n(\di_TA)=\cH^n(A)=\cH^n(A^*).
\end{equation}
We obtain the inequalities \eqref{eqapr33} directly from \eqref{eqapr31},  \eqref{eqdec291}, and \eqref{eqdec292}.

The second statement in the lemma follows directly from Lemma~\ref{lemjuly312} with $\di$ replaced by $\di_T$, which is valid by Lemma~\ref{lemaug11}.
\end{proof}

\begin{ex}\label{exaug61}
{\rm
\noindent{(i)} Let $H=e_n^\perp$ and define $\di:{\mathcal{K}}^n_n\rightarrow \cK^n_n$ by $\di K=M_HK$, the Minkowski symmetral of $K$ (see \cite[Section~3]{BGG}).  Then $\di$ is smoothing but not measure preserving and does not satisfy $S(\di K)\le S(K)$ for each $K\in \cK^n_n$. This shows that the measure-preserving assumption in Lemma~\ref{lemjuly312} cannot be dropped.

\noindent{(ii)} Let $K\in {\mathcal K}^n_{(o)}$ and define $\di:\cL^n\rightarrow \cK^n_n$ by $\di A=r_A K$, where $r_A=({\cH^n}(A)/{\cH^n}(K))^{1/n}$.  This map, which corresponds to the convex symmetrization in \cite{AFTL}, is monotonic and measure preserving and, using \eqref{eq:smothDef} and the Brunn-Minkowski inequality \cite{Gar02}, it is easy to see that it is $K$-smoothing.  Clearly, $\di$ maps balls to balls if and only if $K$ is a ball, and we claim that it is smoothing if and only if $K$ is an $o$-symmetric ball.  To see this, note that by definition, $\di$ is smoothing if
$${\cH^n}(A)^{1/n}K+d {\cH^n}(K)^{1/n}B^n\subset {\cH^n}(A+dB^n)^{1/n}K$$
for all $A\in \cL^n$.  If $A=K$, this implies that
$${\cH^n}(K)^{1/n}(K+dB^n)\subset{\cH^n}(K+dB^n)^{1/n}K.$$
Since the sets on both sides of this inclusion have the same volume, the inclusion must be an equality and hence
$${\cH^n}(K)^{1/n}K+{\cH^n}(K)^{1/n}dB^n={\cH^n}(K)^{1/n}K+\left({\cH^n}(K+dB^n)^{1/n}
-{\cH^n}(K)^{1/n}\right)K.$$
The cancelation law \cite[p.~139]{Sch14} for Minkowski addition yields
$${\cH^n}(K)^{1/n}dB^n=\left({\cH^n}(K+dB^n)^{1/n}
-{\cH^n}(K)^{1/n}\right)K,$$
which holds if and only if $K=rB^n$ for some $r\ge 0$. Thus, the smoothing assumption cannot be omitted in Lemma~\ref{lemjuly312}.

Now suppose that $X={\mathcal{S}}(\R^n)$ or $\cV(\R^n)$ and $T:X\to X$ is the rearrangement given, for all $x\in \R^n$, by \eqref{eqoct62} or \eqref{eqjan72} with $\di_T=\di$.  Then the superlevel sets of $Tf$ are dilates of $K$, so $T$ is the {\em $K$-Schwarz rearrangement} mentioned in the Introduction, and again, unless $K$ is an $o$-symmetric ball, $T$ is $K$-smoothing but not smoothing. Therefore the smoothing assumption in Theorem~\ref{lemdec28} cannot be replaced by $K$-smoothing for any non-spherical convex body $K$.

See Example~\ref{exmay205} for more information about the rearrangement $T$.
\qed }
\end{ex}

In the rest of this section, we study the relationship between the $K$-smoothing and reduction of the $K$-modulus of continuity properties of a rearrangement $T:X\to X$, where $X={\mathcal{M}}(\R^n)$, ${\mathcal{M}}_+(\R^n)$, ${\mathcal{S}}(\R^n)$, or $\cV(\R^n)$.  When $K=B^n$, some information of this type was obtained by Brock and Solynin in \cite[Theorem~3.3]{BS}, which states that a rearrangement (in their sense of the term) that is continuous from the inside is smoothing if and only if it reduces the modulus of continuity of continuous functions in ${\mathcal{S}}(\R^n)$.  A comparison of their approach to rearrangements and ours can be found in \cite[Appendix]{BGGK}.

When $K=B^n$, it is possible to use \cite[Theorem~3.3]{BS} to obtain the same result for our rearrangements, i.e., the equivalence (ii)$\Leftrightarrow$(iii) of Corollary~\ref{lemdec301} below.  To see this, note firstly that the continuity from the inside assumption is not necessary in our context, by Lemma~\ref{lemaug11}.  If $T:{\mathcal{S}}(\R^n)\to {\mathcal{S}}(\R^n)$ is a rearrangement, then by Lemma~\ref{lemaug11}, the set transformation $\di_T^*:{\cL}^n\to {\cL}^n$ is a rearrangement in the sense of \cite[p.~1762]{BS}, since it is pointwise monotonic.  If $\overline{T}:{\mathcal{S}}(\R^n)\to {\mathcal{S}}(\R^n)$ denotes the rearrangement map induced by $\di_T^*$ via \cite[(3.1), p.~1762]{BS}, then \eqref{eqjan72} yields $\overline{T}=T$, so \cite[Theorem~3.3]{BS} is valid for $T$.

However, even when $K=B^n$, Theorems~\ref{lemdec201} and~\ref{lemdec194} below are more general than \cite[Theorem~3.3]{BS}, since they apply to much wider classes of functions.

In (ii) of the following lemma, we assume that $A\in \cM^n$ is bounded to ensure that $A+d\,\inte K\in \cL^n$ when $d>0$.  This seems unavoidable since $A+d\,\inte K\not\in\cL^n$ when $A$ is unbounded, in which case $\di_T^*(A+d\,\inte K)$ is not defined.

\begin{lem}\label{lemapr3}
Let $X={\mathcal{M}}(\R^n)$, ${\mathcal{M}}_+(\R^n)$, ${\mathcal{S}}(\R^n)$, or $\cV(\R^n)$, let $T:X\to X$ be a rearrangement, and let $K\in {\mathcal K}^n_{(o)}$.    The following statements are equivalent.

\noindent{\rm{(i)}}  $T$ is $K$-smoothing.

\noindent{\rm{(ii)}} For each $d>0$ and bounded $A\in \cM^n$, we have
\begin{equation}\label{eqdec221}
(\di_T^*A)+d\,\inte K\subset \di_T^*(A+d\,\inte K).
\end{equation}

\noindent{\rm{(iii)}} For each $d>0$ and bounded $A\in \cM^n$, \eqref{eqdec221} holds essentially.

\noindent{\rm{(iv)}} For each $d>0$ and $A\in \cL^n$, we have
\begin{equation}\label{eqdec192}
(\di_T^*A)+d\,\inte K\subset \cup\{\di_T^*E: E\in \cL^n,~ E\subset A+d\,\inte K\},
\end{equation}
essentially.
\end{lem}

\begin{proof}
(i)$\Rightarrow$ (iii)  Let $d>0$, let $A\in {\cM}^n$ be bounded, and choose $N\in \N$ so that $1/N<d$.  Using the $K$-smoothing property of $T$ and the continuity of $\di_T^*$ from the inside provided by Lemma~\ref{lemaug11}, we obtain
\begin{eqnarray*}
(\di_T^*A)+d\,\inte K&= & \cup_{m=N}^{\infty}\,\left((\di_T^*A)+(d-1/m)K\right)
\subset \cup_{m=N}^{\infty}\,\di_T^*(A+(d-1/m)K)\\
&=& \di_T^*\left(\cup_{m=N}^{\infty}(A+(d-1/m) K)\right)
=\di_T^*(A+d\,\inte K),
\end{eqnarray*}
essentially. This proves \eqref{eqdec221}.

(iii)$\Rightarrow$ (ii)  Let $d>0$ and let $A\in {\cM}^n$ be bounded.  From the fact that $(\di_T^* A)+d\,\inte K$ is open, the essential inclusion \eqref{eqdec221}, and parts (i) and (iii) of Lemma~\ref{lemjan61}, we obtain
$$(\di_T^* A)+d\,\inte K=\left((\di_T^* A)+d\,\inte K)\right)^*\subset \di_T^*(A+d\,\inte K),$$
proving (ii).

(ii)$\Rightarrow$(i) Let $d>0$, let $A\in {\cM}^n$ be bounded. We use \eqref{eqdec221} and the continuity of $\di_T^*$ from the outside from Lemma~\ref{lemaug11} to get
\begin{eqnarray*}
(\di_T^*A)+dK&\subset & \cap_{m\in \N}\,\left((\di_T^*A)+(d+1/m)\,\inte K\right)
\subset \cap_{m\in \N}\,\di_T^*(A+(d+1/m)\,\inte K)\\
&= &\di_T^*\left(\cap_{m\in \N}(A+(d+1/m)\,\inte K)\right)
= \di_T^*(A+d K),
\end{eqnarray*}
essentially, so $T$ is $K$-smoothing.

(iii)$\Rightarrow$(iv) Let $d>0$, let $A\in {\cL}^n$, and define $A_m=A\cap m\,\inte K$ for $m\in \N$.  Since $\di_T^*$ is continuous from the inside by Lemma~\ref{lemaug11}, we obtain
\begin{eqnarray*}
(\di_T^*A)+d\,\inte K&=&\left(\di_T^*(\cup_{m\in\N} A_m)\right)+d\,\inte K\\
&=&\left(\cup_{m\in\N}\,\di_T^*A_m\right)+d\,\inte K
=\cup_{m\in\N}(\di_T^*A_m+d\,\inte K)\\
&\subset&\cup_{m\in\N}\,\di_T^*(A_m+d\,\inte K)\subset \cup\{\di_T^*E: E\in \cL^n,~ E\subset A+d\,\inte K\},
\end{eqnarray*}
essentially, and \eqref{eqdec192} follows.

(iv)$\Rightarrow$(iii)  Let $d>0$ and let $A\in {\cM}^n$ be bounded.  Applying \eqref{eqdec192}, we immediately obtain \eqref{eqdec221}.
\end{proof}

Recall the definition \eqref{eqjuly282} of approximate continuity.

\begin{lem}\label{modulus_appr_contin}
Let $d>0$, let $f\in\cM(\R^n)$, and let $C$ be the set of points of approximate continuity of $f$.  If $K\in {\mathcal K}^n_{(o)}$, then
$$
\omega_{K,d}(f)=\esssup_{\|x-y\|_K\le d}|f(x)-f(y)|=\sup_{\|x-y\|_K\le d;~ x,y\in C}|f(x)-f(y)|.
$$
\end{lem}

\begin{proof}
Let $d$, $f$, and $C$ be as in the statement of the lemma.  Since $C=\R^n$, essentially, we have
$$
\omega_{K,d}(f)=\esssup_{\|x-y\|_K\le d;~ x,y\in C}|f(x)-f(y)|.
$$
Moreover, it is clear that
$$
\esssup_{\|x-y\|_K\le d;~ x,y\in C}|f(x)-f(y)|\leq \sup_{\|x-y\|_K\le d;~ x,y\in C}|f(x)-f(y)|=s,
$$
say. To prove the reverse of the previous inequality, we may assume that $s>0$, since if $s=0$, it is trivial.  Let $0<\ee<s/2$. It suffices to show that
\begin{equation}\label{positive_meas}
\cH^{2n}\left(\{(x,y)\in C\times C: \|x-y\|_K\le d~{\text{and}}~ |f(x)-f(y)|>s-\ee\}\right)>0.
\end{equation}
To this end, choose $\bar{x},\bar{y}\in C$ with $\bar d=\|\bar{x}-\bar{y}\|_{K}\le d$, such that
$$
|f(\bar{x})-f(\bar{y})|>s-\ee/2,
$$
and note that $\bar d>0$. For $x\in\R^n$ and $r>0$, let
$$
A_{x,r}=\{y\in B(x,r)\cap C :|f(x)-f(y)|<\ee/8\}.
$$
If $(x,y)\in A_{\bar{x},r}\times A_{\bar{y},r}$, then
\begin{equation}\label{fact0}
|f(x)-f(y)|\geq |f(\bar{x})-f(\bar{y})|-\ee/4> s-\ee.
\end{equation}

Define $E_d=\{(x,y)\in \R^{2n}: \|x-y\|_{K}\le d\}$.  We aim to prove that $\cH^{2n}\left(\big(A_{\bar{x},r}\times A_{\bar{y},r}\big)\cap E_d\right)>0$ for small $r>0$. Since
$$
(A_{\bar{x},r}\times A_{\bar{y},r})\cap E_d=(A_{\bar{x},r}\times A_{\bar{y},r})\setminus \big(( A_{\bar{x},r}\times A_{\bar{y},r})\setminus E_d\big)
$$
and
$$
( A_{\bar{x},r}\times A_{\bar{y},r})\setminus E_d
\subset(B(\bar{x},r)\times B(\bar{y},r))\setminus E_d,
$$
we have
\begin{multline}\label{ineq_1}
\cH^{2n}\left((A_{\bar{x},r}\times A_{\bar{y},r})\cap E_d\right)\geq
\cH^{2n}(A_{\bar{x},r}\times A_{\bar{y},r})\\
- \cH^{2n}(B(\bar{x},r)\times B(\bar{y},r))+\cH^{2n}\left((B(\bar{x},r)\times B(\bar{y},r))\cap E_d\right).
\end{multline}
The approximate continuity of $f$ at $\bar{x}$ and $\bar{y}$ yields
\begin{equation}\label{fact3}
 \lim_{r\to0}
\frac
{\cH^{2n}(A_{\bar{x},r}\times A_{\bar{y},r})}
{\cH^{2n}(B(\bar{x},r)\times B(\bar{y},r))}
=
\lim_{r\to0}
\frac{\cH^n(A_{\bar{x},r})}{\cH^n(B(\bar{x},r))}
\lim_{r\to0}
\frac{\cH^n(A_{\bar{y},r})}{\cH^n(B(\bar{y},r))}
=1.
\end{equation}
Let $\Delta=\{(x,x): x\in \R^n\}$ be the diagonal in $\R^{2n}$. Now $\|x-y\|_K\le \bar{d}$ if and only if $x\in \bar{d}K+y$, which holds if and only if $(x,y)\in \left( \bar{d}K\times \{o\}\right)+(y,y)$.  It follows that
\begin{equation}\label{eqmay177}
E_{\bar d}=\left(\bar{d}K\times \{o\}\right)+\Delta
\end{equation}
is a $2n$-dimensional convex cylinder in $\R^{2n}$ as $\bar d>0$.  Since $\bar{d}=\|\bar{x}-\bar{y}\|_{K}$, we have $(\bar{x},\bar{y})\in \left(\partial\left(\bar{d}K\right)\times \{o\}\right)+(\bar{y},\bar{y})$ and hence $(\bar{x},\bar{y})\in \partial E_{\bar d}$.  By the convexity of $E_{\bar d}$,
\begin{align}\nonumber
\lim_{r\to0}\frac{\cH^{2n}\left((B(\bar{x},r)\times B(\bar{y},r))\cap E_{\bar d}\right)}{\cH^{2n}\left(B(\bar{x},r)\times B(\bar{y},r)\right)}
&\ge\lim_{r\to0} \frac{\cH^{2n}\left( B((\bar{x},\bar{y}),r)\cap E_{\bar d}\right)}{\cH^{2n}\left(B(\bar{x},r)\times B(\bar{y},r)\right)}
\\&=a=a\left(\bar{x},\bar{y}\right)>0,\label{fact1}
\end{align}
as the density of $E_{\bar d}$ at the boundary point $(\bar{x},\bar{y})$ is positive.  Since $\bar{d}\le d$, we have
$E_{\bar d}\subset E_{d}$ and \eqref{fact1} implies that
\begin{equation}\label{fact2}
\lim_{r\to0}
\frac
{\cH^{2n}\left((B(\bar{x},r)\times  B(\bar{y},r))\cap E_{d}\right)}{\cH^{2n}(B(\bar{x},r)\times B(\bar{y},r))}\geq a.
\end{equation}
From \eqref{ineq_1}, \eqref{fact3}, and \eqref{fact2}, we conclude that
\begin{equation}\label{fact4}
\lim_{r\to0}
\frac
{\cH^{2n}\left((A_{\bar{x},r}\times  A_{\bar{y},r})\cap E_{d}\right)}{\cH^{2n}(B(\bar{x},r)\times B(\bar{y},r))}\geq a>0.
\end{equation}
Finally, \eqref{fact0} and \eqref{fact4} imply \eqref{positive_meas}.
\end{proof}

\begin{lem}\label{lemdec191}
Let $A\subset\R^n$, let $d>0$, and let $K\in {\mathcal K}^n_{(o)}$. For $x\in \R^n$, define $$d_K(x,A)=\inf\{\|x-y\|_{K}:y\in A\}=\inf\{\lambda\ge 0:x\in \lambda K+y{\mathrm{~for~some~}}y\in A\}$$
and
\begin{equation}\label{eqdec203}
f_A(x)=(d-d_K(x,A))^+,
\end{equation}
where $s^+$ is the nonnegative part of $s\in \R$.  Then $f_A$ is a $K$-contraction as defined in \eqref{Kcontraction}.
\end{lem}

\begin{proof}
Let $x,y\in \R^n$.  If $x,y\not\in A+dK$, then \eqref{eqdec203} implies that $|f_A(x)-f_A(y)|=0\le \|x-y\|_{K}$.  Otherwise, we may, by relabeling if necessary, assume that $x\in A+dK$ and $d_K(x,A)\le d_K(y,A)$.
Let $\delta>0$ and choose $x'\in A$ such that $\|x-x'\|_{K}< d_K(x,A)+\delta$.  Then
$$|f_A(x)-f_A(y)|\le d_K(y,A)-d_K(x,A)< \|y-x'\|_{K}-\|x-x'\|_{K}+\delta\le \|y-x\|_{K}+\delta.$$
Therefore $|f_A(x)-f_A(y)|\le \|y-x\|_{K}=\|x-y\|_{-K}$.  As was noted directly after \eqref{Kcontraction}, this proves that $f_A$ is a $K$-contraction.
\end{proof}

\begin{lem}\label{lemdec311}
\noindent{\rm{(i)}}  If $T:{\mathcal{M}}(\R^n)\to {\mathcal{M}}(\R^n)$ is a rearrangement, then \eqref{eqoct72} holds, essentially.

\noindent{\rm{(ii)}}  Let $X={\mathcal{M}}(\R^n)$ or ${\mathcal{M}}_+(\R^n)$ and suppose that $T:X\to X$ is a rearrangement.  If $\alpha>0$ and $\beta=0$, then \eqref{eqoct61} holds, essentially.

\noindent{\rm{(iii)}}  Let $X={\mathcal{M}}(\R^n)$ or ${\mathcal{M}}_+(\R^n)$ and suppose that $T:X\to X$ is a rearrangement.  Then \eqref{eqoct61} holds, essentially, if $\alpha=0$ and $\beta\in \R$ when $X={\mathcal{M}}(\R^n)$, and also if $\alpha=0$ and $\beta\ge 0$ when $X={\mathcal{M}}^+(\R^n)$ and $T$ reduces the $K$-modulus of continuity for some $K\in {\mathcal K}^n_{(o)}$.
\end{lem}

\begin{proof} (i) By Proposition~\ref{may8lem}(iii), $T:{\mathcal{V}}(\R^n)\to {\mathcal{V}}(\R^n)$.  Let $A\in \cL^n$.  Since $1_A\in {\mathcal{V}}(\R^n)$, the result follows from Proposition~\ref{fromSetToFct}(ii).

(ii) By (i), \eqref{eqoct72} holds, essentially, when $X={\mathcal{M}}(\R^n)$.  With this in hand, the second paragraph of the proof of Proposition~\ref{lemoct6} can be followed verbatim.

(iii) If $X={\mathcal{M}}(\R^n)$, the result follows from Proposition~\ref{may8lem}(iii).  The latter fails when $X={\mathcal{M}}^+(\R^n)$, by \cite[Example~4.3]{BGGK}, but if $\beta\ge 0$, $T$ reduces the $K$-modulus of continuity for some $K\in {\mathcal K}^n_{(o)}$, and $f=\beta$, essentially, then $\omega_{K,d}(Tf)=\omega_{K,d}(f)=0$ for all $d>0$.  This and the equimeasurability of $T$ give $Tf=\beta$, essentially.
\end{proof}

\begin{thm}\label{lemdec201}
Let $X={\mathcal{M}}(\R^n)$, ${\mathcal{M}}_+(\R^n)$, ${\mathcal{S}}(\R^n)$, or $\cV(\R^n)$, and let $K\in {\mathcal K}^n_{(o)}$.  If $T:X\to X$ is a rearrangement that reduces the $K$-modulus of continuity of each $K$-contraction in $X$, then $T$ is $K$-smoothing.
\end{thm}

\begin{proof}
Let $d>0$, let $A\in\cM^n$ be bounded, and let $f_A$ be defined by \eqref{eqdec203}. By Lemma~\ref{lemdec191}, $f_A\in \cV(\R^n)$ is a $K$-contraction, so $\omega_{K,d}(f_A)\le d$ for all $d>0$, by the definition \eqref{eqmay255} of the $K$-modulus of continuity. If $N_0$ is the complement of the set of points of approximate continuity of $Tf_A$, then $\cH^n(N_0)=0$.  Lemma~\ref{modulus_appr_contin} implies that
\begin{eqnarray*}\label{eqdec252}
|Tf_A(x)-Tf_A(y)|&\le& \omega_{K,\,\|x-y\|_{K}}(Tf_A)\le\|x-y\|_K
\end{eqnarray*}
for $x,y\in \R^n\setminus N_0$, so $Tf_A$ is a $K$-contraction on $\R^n\setminus N_0$.  Since $T$ is monotonic and $d 1_{A+d\,\inte K}\ge f_A \ge d 1_A$, we have
$$T(d 1_{A+d\,\inte K})(x)\ge Tf_A(x)\ge T(d 1_A)(x)$$
for $x\in \R^n\setminus N_1$, where $\cH^n(N_1)=0$.

By Proposition~\ref{fromSetToFct}(ii) and Lemma~\ref{lemdec311}(i), there is a set $N_2$ with $\cH^n(N_2)=0$ such that
\begin{equation}\label{eqdec251}
T1_{A+d\,\inte K}(x)=1_{\di_T(A+d\,\inte K)}(x)
\end{equation}
for $x\in \R^n\setminus N_2$. By Proposition~\ref{lemoct6} and Lemma~\ref{lemdec311}(ii), there is a set $N_3$ with $\cH^n(N_3)=0$ such that \eqref{eqoct61} holds everywhere on $\R^n\setminus N_3$ when $\alpha=d$, $\beta=0$, and $A=A$ or $A=A+d\,\inte K$.  Let $N=\cup\{N_i:i=0,1,2,3\}$.

Let $x_0\in ((\di_T^*A)+d\,\inte K)\setminus N$. There is a $y_0\in (\di_TA)\setminus N$ such that $\|x_0-y_0\|_{K}=d'<d$. As $Tf_A$ is a $K$-contraction on $\R^n\setminus N$,  $Tf_A(x_0)\ge Tf_A(y_0)-d'>Tf_A(y_0)-d$.  Using \eqref{eqoct61}, we obtain
\begin{eqnarray*}
d\, T1_{A+d\,\inte K}(x_0)&=&T(d 1_{A+d\,\inte K})(x_0)\ge Tf_A(x_0)\\
&>& Tf_A(y_0)-d\ge T(d 1_A)(y_0)-d=d\, T1_A(y_0)-d=d-d=0,
\end{eqnarray*}
since $T1_A(y_0)=1$ due to $y_0\in \di_TA$ and \eqref{eqIndic}.  Thus $T1_{A+d\,\inte K}(x_0)>0$ and then $x_0\in  \di_T(A+d\,\inte K)$, by \eqref{eqdec251}.  This shows that $(\di_T^*A)+d\,\inte K\subset \di_T(A+d\,\inte K)=\di_T^*(A+d\,\inte K)$, essentially. Therefore \eqref{eqdec221} holds and $T$ is $K$-smoothing by Lemma~\ref{lemapr3}.
\end{proof}

\begin{lem}\label{lemdec192}
Let $d>0$, let $f\in {\mathcal{M}}(\R^n)$, let $K\in {\mathcal K}^n_{(o)}$, and let $t\in \R$.  Then
$$\{x:f(x)\ge t+\omega_{K,\,d}(f)\}^*+d\,\inte K\subset \{x:f(x)\ge t\}^*.$$
\end{lem}

\begin{proof}
If $N$ is the complement of the set of points of approximate continuity of $f$, then $\cH^n(N)=0$ and, by Lemma~\ref{modulus_appr_contin},
$$\omega_{K,\,d}(f)=\esssup_{\|x-y\|_K\le d}|f(x)-f(y)|=\sup_{\|x-y\|_K\le d;~ x,y\not\in N}|f(x)-f(y)|$$
for any $d>0$. Now fix $d>0$ and let $y\in \{x:f(x)\ge t+\omega_{K,\,d}(f)\}^*+d\,\inte K$.  By \eqref{gaugefunction}, we can choose $z\in \{x:f(x)\ge t+\omega_{K,\,d}(f)\}^*$ and $\ee=\ee(y)>0$ such that $\|y-z\|_K+2\ee<d$.  Let $y'\in (\ee K+y)\setminus N$ and $z'\in (\{x:f(x)\ge t+\omega_{K,\,d}(f)\}\cap (-\ee K+z))\setminus N$.  Then $\|y'-y\|_K\le \ee$ and $\|z-z'\|_K\le \ee$, so $\|y'-z'\|_K\le \|y-z\|_K+2\ee <d$.  Since $y',z'\not\in N$, we have $|f(y')-f(z')|\le \omega_{K,\,d}(f)$.  Hence
$$f(y')\ge f(z')-\omega_{K,\,d}(f)\ge t.$$
This implies that $(\ee K+y)\subset \{x:f(x)\ge t\}$, essentially. Thus the latter set essentially contains an open neighborhood of $y$, and the desired conclusion follows easily.
\end{proof}

\begin{lem}\label{lemdec193}
If $\alpha>0$, $d>0$, $g\in {\mathcal{M}}(\R^n)$, and an $o$-symmetric $K\in {\mathcal K}^n_{(o)}$ are such that
\begin{equation}\label{eqdec191}
\{x:g(x)\ge t+\alpha\}^*+d\,\inte K\subset \{x:g(x)\ge t\}^*,
\end{equation}
essentially, for $t>\essinf g$, then $\omega_{K,\,d}(g)\le\alpha$.
\end{lem}

\begin{proof}
Note firstly that the inclusion in \eqref{eqdec191} actually holds pointwise, a fact we shall use later in the proof.  This is because the set $G=\{x:g(x)\ge t+\alpha\}^*+d\,\inte K$ is open, so $G^*=G$, and the pointwise inclusion then follows from parts (i) and (iii) of Lemma~\ref{lemjan61}.
	
Suppose that $\omega_{K,\,d}(g)>\alpha$. Let $E_d=\{(y,z)\in \R^n\times\R^n: \|y-z\|_K\le d\}$, let
$$F=\{(y,z)\in \R^n\times\R^n: g(y),~g(z)\ge \essinf g~~{\text{and}}~~ |g(y)-g(z)|>\alpha\},$$
and let $A=E_d\cap F$.  Then $\omega_{K,\,d}(g)>\alpha$ implies that $\cH^{2n}(A)>0$.   The $o$-symmetry of $K$ yields $\|z-y\|_K=\|y-z\|_K$, so $(z,y)\in A$ if and only if $(y,z)\in A$. For $k\in \N$ and $q\in\Q$ with $q>\essinf g$, let
$$V_{k,q}=E_{d-1/k}\cap \{(y,z)\in \R^n\times \R^n: g(y)\ge q+\alpha,~{\text{and}}~ \essinf g\le g(z)<q\},$$
let $W_{k,q}=V_{k,q}\cup\{(z,y): (y,z)\in V_{k,q}\}$,
and let $W=\cup\{W_{k,q}: k\in \N,~q\in \Q\}$.  Clearly $W\subset A$. Let $Z=\{(y,z)\in \R^n\times \R^n: \|y-z\|_K=d\}$.  By Fubini's theorem, $\cH^{2n}(Z)=0$. For each $(y,z)\in A\setminus Z$, we can find $k\in \N$ such that $\|y-z\|_K\le d-1/k$ and $|g(y)-g(z)|>\alpha$, and hence also a $q\in \Q$ such that $q>\essinf g$ and
$$\essinf g\le g(z)<q<q+\alpha\le g(y)$$
(or the same with $y$ and $z$ interchanged). This means that $A\setminus Z\subset W$ and therefore $\cH^{2n}(W)=\cH^{2n}(A)>0$.  It follows that there are $k_0\in \N$ and $q_0\in \Q$ with $q_0>\essinf g$ such that $\cH^{2n}(W_{k_0,q_0})>0$ and hence, without loss of generality, $\cH^{2n}(V_{k_0,q_0})>0$.  By the Lebesgue density theorem, there exists $(y_0,z_0)\in V_{k_0,q_0}$ such that
$y_0\in \{x:g(x)\ge q_0+\alpha\}^*$ and $z_0\in \{x:g(x)<q_0\}^*$.  Since $\|z_0-y_0\|_K=\|y_0-z_0\|_K<d$, we have $z_0\in \inte dK+y_0$.  The pointwise inclusion in \eqref{eqdec191} with $t=q_0$ then implies that $z_0\in \{x:g(x)\ge q_0\}^*$, a contradiction that completes the proof.
\end{proof}

\begin{thm}\label{lemdec194}
Let $X={\mathcal{S}}(\R^n)$ or $\cV(\R^n)$, let $K\in \cK^n_{(o)}$, and let $T:X\to X$ be a $K$-smoothing rearrangement.

\noindent{\rm{(i)}} If $K$ is $o$-symmetric, then $T$ reduces the $K$-modulus of continuity of each $f\in X$.

\noindent{\rm{(ii)}} If $d>0$, $f\in X$, and $rB^n\subset K\subset RB^n$ for $0<r\le R$, then $\omega_{dr}(Tf)\le \omega_{dR}(f)$.
\end{thm}

\begin{proof}
{\rm{(i)}}  Let $d>0$, let $f\in X$, and let $K$ be $o$-symmetric. If $\omega_{K,d}(f)=0$, then clearly $f$ is essentially constant, and hence, by Proposition~\ref{may8lem}(iii), $Tf$ is also essentially constant.  Then $\omega_{K,d}(Tf)=0=\omega_{K,d}(f)$.  Therefore we may assume that $\omega_{K,d}(f)>0$.

Let $t>\essinf f$ and recall that $\essinf Tf=\essinf f$, by Proposition~\ref{may8lem}(i).  Then, using \eqref{eqnov11}, Lemma~\ref{lemjan63}(iii), \eqref{eqdec192} with $A=\{x:f(x)\ge t+\omega_{K,d}(f)\}^*$, Lemma~\ref{lemjan63}(i), Lemma~\ref{lemdec192}, and Lemma~\ref{lemjan63}(iii) and \eqref{eqnov11} again, we obtain
\begin{eqnarray*}\label{eqdec195}
\lefteqn{\{x:Tf(x)\ge t+\omega_{K,d}(f)\}^*+d\,\inte K
=(\di_T^*\{x:f(x)\ge t+\omega_{K,d}(f)\})+d\,\inte K}\\
&=& (\di_T^*\{x:f(x)\ge t+\omega_{K,d}(f)\}^*)+d\,\inte K\\
&\subset & \cup\{\di_T^*E: E\in \cL^n,~ E\subset \{x:f(x)\ge t+\omega_{K,d}(f)\}^*+d\,\inte K\}\\
&\subset & \di_T^*\{x:f(x)\ge t\}^*=
\di_T^*\{x:f(x)\ge t\}=\{x:Tf(x)\ge t\}^*,
\end{eqnarray*}
essentially. The conclusion $\omega_{K,d}(Tf)\le \omega_{K,d}(f)$ follows from the $o$-symmetry of $K$ and Lemma~\ref{lemdec193} with $g=Tf$ and $\alpha=\omega_{K,d}(f)$.

{\rm{(ii)}}  The proof is an easy modification of that of part {\rm{(i)}}.
Let $d>0$ and let $f\in X$. If $\omega_{dR}(f)=0$, then as at the beginning of the proof of (i), we conclude that $\omega_{dr}(Tf)=0=\omega_{dR}(f)$.  Therefore we may assume that $\omega_{dR}(f)>0$.  The inclusion $K\subset RB^n$ implies that $\omega_{K,\,d}(f)\le \omega_{RB^n,\,d}(f)=\omega_{dR}(f)$, so Lemma~\ref{lemdec192} yields
\begin{equation}\label{eqmay185}
\{x:f(x)\ge t+\omega_{dR}(f)\}^*+d\,\inte K\subset \{x:f(x)\ge t\}^*
\end{equation}
for $t\in \R$. Arguing as above with $A=\{x:f(x)\ge t+\omega_{dR}(f)\}^*$ and using \eqref{eqmay185} instead of Lemma~\ref{lemdec192}, we obtain
$$\{x:Tf(x)\ge t+\omega_{dR}(f)\}^*+d\,\inte K\subset \{x:Tf(x)\ge t\}^*,$$
essentially, as this does not require $K$ to be $o$-symmetric. Since $rdD^n\subset d\,\inte K$, the conclusion $\omega_{dr}(Tf)\le \omega_{dR}(f)$ follows from Lemma~\ref{lemdec193} with $K=B^n$, $g=Tf$,  $\alpha=\omega_{dR}(f)$, and $d$ replaced by $dr$.
\end{proof}

\begin{cor}\label{lemdec301}
Let $X={\mathcal{S}}(\R^n)$ or $\cV(\R^n)$, let $K\in \cK^n_{(o)}$ be $o$-symmetric, and let $T:X\to X$ be a rearrangement. The following are equivalent.

\noindent{\rm{(i)}} $T$ reduces the $K$-modulus of continuity.

\noindent{\rm{(ii)}} $T$ reduces the $K$-modulus of continuity of each $K$-contraction in $X$.

\noindent{\rm{(iii)}} $T$ is $K$-smoothing.
\end{cor}

\begin{proof}
The implication (i)$\Rightarrow$(ii) is obvious on noting that each $K$-contraction is continuous, while (ii)$\Rightarrow$(iii) and (iii)$\Rightarrow$(i) follow from Theorems~\ref{lemdec201} and~\ref{lemdec194}(i), respectively.
\end{proof}

The following example shows that the $o$-symmetry assumption on $K$ in Theorem~\ref{lemdec194}(i) and Corollary~\ref{lemdec301} (the implications (iii)$\Rightarrow$ (i) and (iii)$\Rightarrow$ (ii)) cannot be omitted.

\begin{ex}\label{exmay205}
{\rm Let $K\in {\mathcal K}^n_{(o)}$, let $X={\mathcal{S}}(\R^n)$ or $\cV(\R^n)$, and let $T:X\to X$ be the $K$-Schwarz rearrangement defined in Example~\ref{exaug61}(ii). If $f\in X$, the superlevel sets of $Tf$ are dilates of $K$.  Since $T$ is $K$-smoothing, it follows from Theorem~\ref{lemdec194} that when $K$ is $o$-symmetric, $T$ reduces the $K$-modulus of continuity of functions in $X$.  However, this is not generally the case if $K$ is not $o$-symmetric.  To see this, let $o\neq x_0\in D^n=\inte B^n$, let $K=B^n+x_0$, and let $f$ be defined by $f(x)=1-\|x\|$ if $x\in B^n$ and $f(x)=0$ otherwise.  Define  $v=x_0/\|x_0\|$ and
\[
M=\max_{u\in S^{n-1}}\rho_K(u)=\rho_K(v)=1+\|x_0\|.
\]
Then $\omega_{K,d}(f)=dM$ for small $d>0$. (For $\omega_{K,d}(f)\ge dM$, it suffices to take $y=o$ and $x=dMv$; then $\|x-y\|_K=d$, because this is equivalent to $x\in y+d \partial K$, and $|f(x)-f(y)|=\|x\|=dM$ for small $d>0$.  The reverse inequality comes from the observation that $f$ is Lipschitz with Lipschitz constant $1$, and from  $\|x-y\|\le M\|x-y\|_K$.)

Now we claim that $\omega_{K,d}(Tf)> dM$.  It suffices to prove $|Tf(o)-Tf(-dMv)|>dM$, since $\|o-(-dMv)\|_K=d$. For $t\in(0,1)$, the statement
\[
 -dMv\in (1-t)(B^n+x_0)=\di_T\{z: f(x)\geq t\}=\di_T^*\{z: f(x)\geq t\}
\]
is true if and only if $t\leq 1-dM/(1-\|x_0\|)$. By \eqref{eqjan72}, this implies that $Tf(-dMv)=1-dM/(1-\|x_0\|)$. With similar but simpler arguments we argue that $Tf(o)=1$. These two facts yield $|Tf(o)-Tf(-dMv)|=dM/(1-\|x_0\|)>d M$, as required.
\qed }
\end{ex}

When $X={\mathcal{M}}(\R^n)$ or ${\mathcal{M}}_+(\R^n)$, the implication (ii)$\Rightarrow$(iii) in Corollary~\ref{lemdec301} remains true, by Theorem~\ref{lemdec201}, but the following example, a modification of \cite[Example~4.4]{BGGK}, shows that (iii)$\Rightarrow$(ii) does not hold generally.

\begin{ex}\label{aug26ex}
{\rm If $K\in \cK^n_{(o)}$ and $X={\mathcal{M}}(\R^n)$ or ${\mathcal{M}}_+(\R^n)$, there are $K$-smoothing rearrangements $T:X\to X$ that do not reduce the $K$-modulus of continuity.  To see this, call $f\in X$ {\em of type I} if $\cH^n(\{x : f(x)>t\})=\infty$ for $t \ge\essinf f$ and {\em of type II} otherwise, i.e., if there is a $t_0 \ge\essinf f$ such that $\cH^n(\{x : f(x)>t\})<\infty$ for $t > t_0$.  Then define
$$Tf=\begin{cases}
f+1_{B^n},&\text{ if $f$ is of type I},\\
f,&\text{ if $f$ is of type II}.\\
\end{cases}
$$
Clearly, $T: X\to X$ is equimeasurable. If $f\le g$, then either $f$ and $g$ are of the same type, or $f$ is of type II and $g$ is of type I.  It follows that $Tf \le Tg$ and hence that $T$ is a rearrangement.  The associated mapping $\di_T$ is the identity on $\cL^n$, so $T$ is $K$-smoothing. The function $f_0(x)=e^{x_1}$ is continuous but its image $Tf_0$ is not. Hence, $T$ does not reduce the $K$-modulus of continuity of continuous functions.}
\qed
\end{ex}

\section{The P\'olya-Szeg\H{o} inequality for Lipschitz functions}\label{Polya-Szego}

Recall that the subgraph $K_f\subset \R^{n+1}$ of a function $f\in {\mathcal{M}}(\R^n)$ is defined by \eqref{eqdec304}.

\begin{lem}\label{lemapr8}
If $f\in {\mathcal{M}}(\R^n)$ and $s\in \R$, then
\begin{equation}\label{eqapr91}
K_{f}^*\cap \{x_{n+1}=s\} \subset \left(K_{f}\cap \{x_{n+1}=s\}\right)^*,
\end{equation}
where the set of Lebesgue density points on the right is formed with respect to the hyperplane $\{x_{n+1}=s\}=\R^n+se_{n+1}$, identified with $\R^n$.
\end{lem}

\begin{proof}
Let $x\in \{x_{n+1}=s\}$.  If $x\not\in \left(K_{f}\cap \{x_{n+1}=s\}\right)^*$, there exists $0<a<1$ such that if $r_0>0$, there is an $0<r<r_0$ such that
\begin{equation}\label{eqapr101}
\cH^n\left(K_{f}\cap \{x_{n+1}=s\}\cap B(x,r)\right)<(1-a)\kappa_n r^n,
\end{equation}
where $B(x,r)$ is the $(n+1)$-dimensional ball with center $x$ and radius $r$.  If $y\in e_{n+1}^{\perp}$ and $y+se_{n+1}\not\in K_{f}\cap \{x_{n+1}=s\}\cap B(x,r)$, then by the definition \eqref{eqdec304} of $K_f$, $y+te_{n+1}\not\in K_{f}$ for each $t>s$.  Therefore, by \eqref{eqapr101} and Fubini's theorem, $\cH^{n+1}(K_f\cap B(x,r))$ is largest when
$$K_{f}\cap \{x_{n+1}=s\}\cap B(x,r)\subset \{x_{n+1}=s\}\cap B\left(x,(1-a)^{1/n}r \right),$$
in which case
\begin{equation}\label{eqapr28}
K_f\cap B(x,r)\subset \left(\left(\left((1-a)^{1/n} r B^n\right)\times [0,s)\right)\cup (r B^n \times (-\infty,s])\right)\cap B(x,r).
\end{equation}
Let
$$E(a,n,r)=r B^{n+1}\setminus \left(\left((1-a)^{1/n} r B^n\right)\times \R\right)$$
be the region in $\R^{n+1}$ between the sphere with center $o$ and radius $r$ and an infinite $o$-symmetric cylinder with radius $(1-a)^{1/n}r$.  From \eqref{eqapr28}, we see that
$$\cH^{n+1}(B(x,r)\setminus K_{f})\ge \cH^{n+1}(E(a,n,r))/2\ge c(a,n)r^{n+1},$$
where $c(a,n)=\cH^{n+1}(E(a,n,1))/2>0$.  It follows that $x\not\in K_f^*$.
\end{proof}

Let $X={\mathcal{M}}(\R^n)$, ${\mathcal{M}}_+(\R^n)$, ${\mathcal{S}}(\R^n)$, or $\cV(\R^n)$, let $T:X\to X$ be a rearrangement, and let $E$ be a subset of the hyperplane $\{x_{n+1}=t\}=\R^n+te_{n+1}$ in $\R^{n+1}$, such that $E\,|\,\R^n\in \cL^n$.  Slightly abusing notation, we shall define
\begin{equation}\label{eqjan53}
\di_T E=(\di_T(E\,|\,\R^n))+te_{n+1},
\end{equation}
thereby extending the action of $\di_T$ to horizontal hyperplanes in $R^{n+1}$.  The action of $\di^*_T$ can be extended in a similar fashion. Note that by \eqref{eqnov11}, \eqref{eqdec304}, and \eqref{eqjan53}, for $X={\mathcal{S}}(\R^n)$ or $\cV(\R^n)$ we have
\begin{eqnarray*}
K_{Tf}\cap \{x_{n+1}=t\}&=&\{x\in \R^n: Tf(x)\ge t\}+te_{n+1}\nonumber\\
&=&(\di_T\{x\in \R^n: f(x)\ge t\})+te_{n+1}\nonumber\\
&=&\di_T(K_f\cap\{x_{n+1}=t\}),
\end{eqnarray*}
essentially, for $t>\essinf f$.  By Lemma~\ref{lemjan61}(ii), this yields the pointwise identity
\begin{equation}\label{eqapr51}
\left(K_{Tf}\cap \{x_{n+1}=t\}\right)^*=\di_T^*(K_f\cap\{x_{n+1}=t\})
\end{equation}
for $t>\essinf f$, where here and below, sets of Lebesgue density points are taken with respect to the appropriate horizontal hyperplane identified with $\R^n$.

The following lemma is stated in a general form required for Section~\ref{anisotropic}.  The reader interested only in the results of this section may focus on the special case corresponding to $K=B^n$, $\inte K=D^n$, when $C\subset\R^{n+1}$ is an $o$-symmetric convex body of revolution about the $x_{n+1}$-axis.

\begin{lem}\label{lemdec303}
Let $X={\mathcal{S}}(\R^n)$ or $\cV(\R^n)$, let $T:X\to X$ be a rearrangement, and let $d>0$. Let $K\in {\mathcal K}^n_{(o)}$ and let $C\subset\R^{n+1}$ be a convex body supported by the hyperplanes $\{x_{n+1}=\pm 1\}$ and all of whose sections $C\cap \{x_{n+1}=t\}$, $t\in [-1,1]$, are dilates of $K$.   If $T$ is $K$-smoothing, $a>d+\essinf f$, and $f\in X$ is such that $\{x: f(x)\ge a\}$ is bounded, then
\begin{eqnarray}\label{eqjune101}
\lefteqn{\left((K_{Tf}\cap \{x_{n+1}\ge a\})^*+d\,\inte C\right)\cap \{x_{n+1}=t\}}\nonumber\\
&\subset &\di_T^*\left(((K_f\cap\{x_{n+1}\ge a\}) +d\,\inte C)\cap \{x_{n+1}=t\}\right)
\end{eqnarray}
for $t>\essinf f$.
\end{lem}

\begin{proof}
Let $d>0$ and let $C=\{(x,x_{n+1})\in\R^n\times\R: x\in g(x_{n+1})K,\, |x_{n+1}|\leq 1\}$, for a suitable concave function $g$ defined on $[-1,1]$.  For $t\in \R$, denote by $\Pi_t$ the orthogonal projection onto $\{x_{n+1}=t\}$. If $L$ is any set in $\R^{n+1}$, then
\begin{equation}\label{eqjan51}
(L+d\,\inte C)\cap \{x_{n+1}=t\}=\bigcup_{t-d< s< t+d}
	 \Pi_t\left((L\cap \{x_{n+1}=s\})+r_s\,\inte K\right),
\end{equation}
where $r_s=d\,g((t-s)/d)$.  Indeed, $p\in (L+d\,\inte C)\cap \{x_{n+1}=t\}$ if and only if $p\,|\,\langle e_{n+1}\rangle =te_{n+1}$ and there is a $z\in L$ such that $p\in z+d\,\inte C$.  If $z\,|\,\langle e_{n+1}\rangle =se_{n+1}$, then this holds if and only if $t-d< s< t+d$ and
$$p-\Pi_t z\in\ d\,g\left( \frac{t-s}{d}\right)\inte K,$$
that is, $p\in \Pi_t(z+r_s \,\inte K)$.

Applying \eqref{eqjan51} with $L$ replaced by $L\cap \{x_{n+1}\ge a\}$, we obtain
\begin{eqnarray}\label{eqjune91}
\lefteqn{\left((L\cap \{x_{n+1}\ge a\})+d\, \inte C\right)\cap \{x_{n+1}=t\}}\nonumber\\
&=&\bigcup_{t-d< s< t+d}\Pi_t\left((L\cap \{x_{n+1}\ge a\}\cap \{x_{n+1}=s\})+r_s \,\inte K\right)\nonumber\\
&=&\bigcup_{t-d< s< t+d,~s\ge a}\Pi_t\left((L\cap \{x_{n+1}=s\})+r_s \,\inte K\right).
\end{eqnarray}

Let $f\in X$ satisfy the hypotheses of the lemma. By Lemma~\ref{lemjan61}(i), we have
$$
(K_{Tf}\cap \{x_{n+1}\ge a\})^*\subset K_{Tf}^*\cap \{x_{n+1}\ge a\}^*\subset
K_{Tf}^*\cap \{x_{n+1}\ge a\}.
$$
From this and \eqref{eqapr91} with $f$ replaced by $Tf$, we obtain
\begin{equation}\label{eqjun92}
(K_{Tf}\cap \{x_{n+1}\ge a\})^*\cap \{x_{n+1}=s\}\subset K_{Tf}^*\cap \{x_{n+1}=s\}\subset (K_{Tf} \cap \{x_{n+1}=s\})^*,
\end{equation}
whenever $s\ge a$, while the set on the left is clearly empty if $s<a$.  We use \eqref{eqjan51} with $L=(K_{Tf}\cap \{x_{n+1}\ge a\})^*$, \eqref{eqjun92}, \eqref{eqapr51}, \eqref{eqdec221} applied (via \eqref{eqjan53}) with $E$ replaced by the bounded set $K_f\cap \{x_{n+1}=s\}$, $s\ge a$, the fact that the action of $\di_T^*$ as extended by \eqref{eqjan53} is the same for each $t$, the pointwise monotonicity of $\di_T^*$ provided by Lemma~\ref{lemjan63}(i), and \eqref{eqjune91} with $L=K_{f}$, to obtain
\begin{align*}
\lefteqn{\left((K_{Tf}\cap \{x_{n+1}\ge a\})^*+d\,\inte C\right)\cap \{x_{n+1}=t\}}\\
&= \bigcup_{t-d<s<t+d}
\Pi_t\left([(K_{Tf}\cap \{x_{n+1}\ge a\})^*\cap \{x_{n+1}=s\}]+r_s\,\inte K\right) \\
&\subset  \bigcup_{t-d<s<t+d,~s\ge a}
      \Pi_t\left([K_{Tf}\cap \{x_{n+1}=s\}]^*+r_s\,\inte K\right) \\
&= \bigcup_{t-d< s< t+d,~s\ge a}\Pi_t\left(
           \left[ \di_T^*\left(K_f\cap\{x_{n+1}=s\}\right)\right]+r_s\,\inte K\right)\\
&\subset\bigcup_{t-d<s<t+d,~s\ge a}\Pi_t\left(\di_T^*\left[(K_f\cap\{x_{n+1}=s\})+r_s\,\inte K
\right]\right)\\
&= \bigcup_{t-d< s< t+d,~s\ge a} \di_T^*\left(\Pi_t\left[
            (K_f\cap \{x_{n+1}=s\})+r_s\,\inte K\right]\right) \\
&\subset  \di_T^*\Big( \bigcup_{t-d< s< t+d,~s\ge a}
            \Pi_t\left[ (K_f\cap \{x_{n+1}=s\})+r_s\,\inte K\right]\Big) \\
&=  \di_T^*\left(((K_f\cap \{x_{n+1}\ge a\})+d\,\inte C)\cap \{x_{n+1}=t\}\right).
\qedhere
\end{align*}
\end{proof}

Recall that $\uomcC(A)$ is the upper anisotropic outer Minkowski content of $A\in\cM(\R^n)$ with respect to a convex body $C\in {\mathcal K}^n_{(o)}$, obtained via the left-hand limit in \eqref{eqapr31} with $B^n$ replaced by $C$. We will apply this notion in $\R^{n+1}$.

\begin{lem}\label{minkcont_diminishes}
Let $X={\mathcal{S}}(\R^n)$ or $\cV(\R^n)$, let $T:X\to X$ be a rearrangement, and let $C$ be as in Lemma~\ref{lemdec303}.  If $T$ is smoothing, $a> \essinf f$, and $f\in X$ is such that $\{x: f(x)\ge a\}$ is bounded, then
\begin{eqnarray*}
\uomcC\left((K_{Tf}\cap \{x_{n+1}\ge a\})^*\right)\le\uomcC(K_f\cap\{x_{n+1}\ge a\}).
\end{eqnarray*}
\end{lem}

\begin{proof}
Taking the ${\cH}^n$-measures of both sides of \eqref{eqjune101}, integrating with respect to $t$, and using Fubini's theorem and the fact that $\di$ is measure preserving, we obtain
\begin{equation}\label{introcore}
{\cH}^{n+1}\left((K_{Tf}\cap \{x_{n+1}\ge a\})^*+d\, \inte C\right)\le {\cH}^{n+1}\left((K_{f}\cap \{x_{n+1}\ge a\})+d\, \inte C\right)
\end{equation}
for $0<d\le a-\essinf f$.  By the equimeasurability of $T$,
$${\cH}^{n+1}\left((K_{Tf}\cap \{x_{n+1}\ge a\})^*\right)={\cH}^{n+1}(K_{Tf}\cap \{x_{n+1}\ge a\})={\cH}^{n+1}(K_{f}\cap \{x_{n+1}\ge a\}).$$
The desired inequality now follows directly from the definition of $\uomcC$ (in $\R^{n+1}$ and with $C$ replaced by $\inte C$).
\end{proof}

\begin{lem}\label{lemoct17}
Let $f\in{\mathcal{S}}(\R^n)$ be Lipschitz.  If $a>\essinf f$, then $\{x: f(x)\ge a\}$ is bounded.
\end{lem}

\begin{proof}
Let $\ee>0$ be such that $a-\ee>\essinf f$ and let $L$ be the Lipschitz constant of $f$.

Suppose that $\{x: f(x)\ge a\}$ is unbounded. Then there are points $x_k$ in this set with $\|x_{k+1}\|>\|x_k\|+2\ee/(1+L)$ for $k\in \N$. The Lipschitz property implies that $f(x)\ge a-\ee$ whenever $x\in B\left(x_k,\ee/(1+L)\right)$, $k\in \N$. As these balls are disjoint, ${\cH}^n(\{x:f(x)\geq a-\ee\})=\infty$, contradicting $f\in {\cS}(\R^n)$.
\end{proof}

Recall that $h_C$ is the support function of $C$ and $G_f$ denotes the graph of $f\in\cM(\R^n)$.   A result in the spirit of the following lemma was proved by Zhang \cite[Lemma~3.1]{Zhang}.

\begin{lem}\label{minkcontequalhausdorff}
Let $f\in{\mathcal{S}}(\R^n)$ be Lipschitz and let $C$ be as in Lemma~\ref{lemdec303}. Let $a>\essinf f$ be such that $\cH^{n}(\{x :  f(x)= a\})=0$. Then
\begin{eqnarray}\label{eq_minkcontequalhausdorff}
\lefteqn{\omcC\big((K_f\cap \{x_{n+1}\ge a\})^*\big)=\omcC\big(K_f\cap \{x_{n+1}\ge a\}\big)}\nonumber\\
&=&\int_{G_f\cap \{x_{n+1}> a\}}h_C(\nu(x))\,d \cH^n(x)+\cH^n(\{x :  f(x)\geq a\}),
\end{eqnarray}
where $\nu(x)$ denotes the outer unit normal to $K_f$ at $x$.
\end{lem}

\begin{proof}
Recall that if $E\subset \R^{n+1}$ is $\cH^{n+1}$-measurable, its density $\Theta(E,x)$ at $x$ is defined by \eqref{eqjuly151} with $n$ replaced by $n+1$. For $t\in[0,1]$, define
$$E^t=\{x\in\R^{n+1} : \Theta(E,x)=t\}.$$
Let $\pa^{\,e}\!E=\R^{n+1}\setminus (E^0\cup E^1)$ denote the essential boundary of $E$. If $E$ has locally finite perimeter, then by Federer's theorem, $\nu(x)$ exists for $\cH^n$-almost all $x\in\pa^{\,e}\!E$; see, for example, \cite[Theorem~16.2]{Mag}.

Lussardi and Villa \cite[Remark~4.2, Theorem~4.4, and Remark~4.5]{LV} prove the following result. If $E\subset\R^{n+1}$ is a Borel set whose boundary is countably $\cH^n$-rectifiable and bounded, $\cH^n(\pa E\cap E^0)=0$, and $E$ has the property that there exist $\ga>0$ and a probability measure $\mu$ in $\R^{n+1}$ absolutely continuous with respect to $\cH^n$, such that for each $x\in\pa E$ and $r\in(0,1)$,
\begin{equation}\label{villa_condition}
\mu(B(x,r))\geq  \gamma r^n,
\end{equation}
then $E$ has finite perimeter, the anisotropic outer Minkowski content of $E$ with respect to $C$ is defined, and
\begin{equation}\label{eqjuly152}
    \omcC(E)=\int_{\pa^{\,e}\! E}h_C(\nu(x))\,d \cH^n(x).
\end{equation}
Recall that a set $E\subset\R^{n+1}$ is {\em countably $\cH^n$-rectifiable} if there exist countably many Lipschitz maps $g_i: \R^n\to \R^{n+1}$ such that $\cH^n\left(E\setminus \cup_i g_i(\R^n)\right)=0$.

Let $K_{f,a}=K_f\cap\{x_{n+1}\geq a\}$. We have
\begin{equation}\label{boundary_Kfa}
\pa K_{f,a}=\big(G_f\cap\{x_{n+1}>a\}\big)\cup \big(K_f\cap \{x_{n+1}=a\}\big).
\end{equation}
If $x\in \R^{n+1}$, write $x'=(x_1,\dots,x_n)$.  It is clear that if $x\in K_f\cap \{x_{n+1}=a\}$ and $f(x')>a$ (or $f(x')=a$), then $\Theta(K_{f,a},x)=1/2$ (or $\Theta(K_{f,a},x)\leq 1/2$, respectively).  We claim that if $x\in G_f\cap\{x_{n+1}>a\}$ and $\Theta(K_{f,a},x)$ exists, then $\Theta(K_{f,a},x)\in(0,1)$.  Indeed, let $x\in G_f\cap\{x_{n+1}>a\}$. If $L$ denotes the Lipschitz constant of $f$ and $\ee>0$ is sufficiently small, then
$$\{(y',y_{n+1}): y_{n+1}\leq x_{n+1}-L\|y'-x'\|\}\cap B(x,\ee)\subset K_{f,a},$$
from which it is easy to see that $\Theta(K_{f,a},x)>0$.  Similarly, $\Theta(K_{f,a},x)<1$ follows easily from
$$
\{(y',y_{n+1}): y_{n+1}> x_{n+1}+L\|y'-x'\|\}\cap K_{f,a}=\emptyset.
$$
This proves the claim.

Since $\cH^{n}(\{x : f(x)=a\})=0$, by assumption, the observations in the previous paragraph imply that
\begin{equation}\label{summary_Kfa}
\cH^n\left(\pa K_{f,a}\cap (K_{f,a})^0\right)=0\quad\text{ and }\quad \pa^{\,e}\! K_{f,a}=\pa K_{f,a},
\end{equation}
up to a set of $\cH^{n}$-measure zero.  The assumption that $\cH^{n}(\{x : f(x)=a\})=0$ also implies that $K_{f,a}^*=\inte K_{f,a}$ and $\pa \big(K_{f,a}^*\big)=\pa K_{f,a}$, up to a set of $\cH^{n}$-measure zero. Arguments similar to those used for $K_{f,a}$ prove that
\begin{equation}\label{summary_Kfa_star}
\cH^n\left(\pa \big(K_{f,a}^*\big)\cap (K_{f,a}^*)^0\right)=0\quad\text{ and }\quad \pa^{\,e}\!\big( K_{f,a}^*\big)=\pa K_{f,a},
\end{equation}
up to a set of $\cH^{n}$-measure zero.

We claim that $K_{f,a}$ and $K_{f,a}^*$ satisfy the hypotheses of Lussardi and Villa's result. Towards this goal, note firstly that by Lemma~\ref{lemoct17}, $K_{f,a}$ and $K_{f,a}^*$ are bounded, and by definition, their boundaries are countably $\cH^n$-rectifiable. Let
\[
D=\{x'\in\R^n :f(x')\geq a\}+B^n\subset \R^n,
\]
let $A_1=\{(x',f(x')): x'\in D\}$, and let $A_2=\{(x',a): x'\in D\}$. For $E\subset\R^{n+1}$, define
$$
\mu(E)=\big(\cH^n\left(\pi(E\cap A_1)\right)+\cH^n(E\cap A_2)\big)/c,
$$
where $c=\cH^n(\pi(A_1))+\cH^n(A_2)$ and $\pi: \R^{n+1}\to\R^n$ is defined by $\pi((x',x_{n+1}))=x'$.  It is clear that $\mu$ is a probability measure in $\R^{n}$ and that since $\pi$ is a contraction, $\mu$ is absolutely continuous with respect to $\cH^n$.

It now suffices to prove that $\mu$ satisfies \eqref{villa_condition} for each $x\in \pa K_{f,a}$. To this end, let $y=(y',y_{n+1})\in\pa K_{f,a}$ and let $r\in(0,1)$. If $y_{n+1}=a$ then
$$\mu(B(y,r))\geq \cH^n\left(B(y,r)\cap A_2\right)/c=\cH^n\left(B(y,r)\cap \{x_{n+1}=a\}\right)/c=\kappa_nr^n/c.$$
If $y_{n+1}=f(y')$, then $\pi\left(B(y,r)\cap A_1\right)$ contains $\{ x\in\R^n : \|x-y'\|\leq r/(1+L)\}$. Therefore
\begin{align*}
\mu(B(y,r))&\geq \cH^n\left(\pi\left(B(y,r)\cap A_1\right)\right)/c \geq\cH^n\left(\{ x\in\R^n : \|x-y'\|\leq r/(1+L)\}\right)/c\\
&=\frac{\kappa_n r^n}{c(1+L)^n}.
\end{align*}
Thus $\mu$ satisfies \eqref{villa_condition} with $\gamma=\kappa_n/(c(1+L)^n)$.

Once we observe that
\[
\int_{K_f\cap\{x_{n+1}=a\}}h_C(\nu(x))\,d \cH^n(x)=h_C(-e_{n+1})\,\cH^n(\{x :  f(x)\geq a\})=\cH^n(\{x :  f(x)\geq a\}),
\]
where we used the fact that $h_C(-e_{n+1})=1$ for our choice of $C$,
formula~\eqref{eqjuly152} with $E=K_{f,a}$ and $K_{f,a}^*$, \eqref{boundary_Kfa}, \eqref{summary_Kfa}, and \eqref{summary_Kfa_star} give \eqref{eq_minkcontequalhausdorff}.
\end{proof}

The following lemma is stated in a general form necessary for Section~\ref{anisotropic}.  The reader interested only in the main results of this section may choose to focus on the special case when $K=B^n$, in which case $h_K(y)=\|y\|$ and $C\subset \R^{n+1}$ is a $o$-symmetric convex body of revolution about the $x_{n+1}$-axis.

\begin{lem}\label{lemoct20}
Let $\Phi:[0,\infty)\to[0,\infty)$ be convex with $\Phi(0)=0$ and $\Phi\not\equiv0$, let $M>0$, and let $K\in {\mathcal K}^n_{(o)}$. Then there exist $b>0$ and a convex body $C\subset \R^{n+1}$, supported by the hyperplanes $\{x_{n+1}=\pm 1\}$ and all of whose sections $C\cap \{x_{n+1}=t\}$, $t\in [-1,1]$, are dilates of $K$, such that
\begin{equation}\label{eqoct201}
h_C(y,1)=1+b\,\Phi(h_K(y)),
\end{equation}
for $y\in \R^n$ with $h_K(y)\le M$.  In particular, $C$ satisfies the conditions in Lemma~\ref{lemdec303}.
\end{lem}

\begin{proof}
Define
$$\Psi(t)=\begin{cases}
\Phi(t),&\text{ if $0\le t\le M$},\\
mt+q,&\text{ if $t\ge M$},\\
\end{cases}
$$
where $m>0$ and $q\le 0$ are such that $\Psi:[0,\infty)\to[0,\infty)$ is convex.  Then, for $y\in \R^n$ and $t\in \R$, define
\begin{eqnarray}
h(y,t)&=&\begin{cases}
|t|\left(1+b\,\Psi(h_K(y)/|t|\right),&\text{ if $t\neq 0$},\\
b\,m\,h_K(y),&\text{ if $t=0$},\\
\end{cases}\label{eqoct202 new}\\
&=&\begin{cases}
|t|\left(1+b\,\Phi(h_K(y)/|t|\right),&\text{ if $|t|\ge h_K(y)/M$},\\
b\,m\,h_K(y)+(1+b\,q)|t|,&\text{ if $|t|\le h_K(y)/M$},\\
\end{cases}\label{eqoct203 new}
\end{eqnarray}
where $b>0$.  We show that $b$ can be chosen so that $h=h_C$ is the support function of a convex body $C$. To this end, note that from \eqref{eqoct202 new}, the positive homogeneity of $h$ follows immediately and the subadditivity of $h$ for $t>0$ or for $t<0$ is a routine exercise using the triangle inequality and the convexity of $\Psi$.  It is then enough to observe that if $b$ is small enough to ensure that $1+b\,q>0$, then the function $b\,m\,h_K(y)+(1+b\,q)|t|$ in \eqref{eqoct203 new} coincides with the support function of the cylinder $b\,m\,K\times[-(1+b\,q),1+b\,q]\subset\R^n\times\R$.  This proves \eqref{eqoct201}.

It remains to prove that all sections $C\cap\{x_{n+1}=t\}$, $t\in[-1,1]$, are dilates of $K$. Clearly $o\in \inte C$, since $h_C>0$.  From \eqref{polarrad} and \eqref{eqoct202 new}, it follows easily that the sublevel sets of $h_C(y,1)$, considered as a function of $y$, are dilates of the polar body $K^\circ$ of $K$.  If $t\in (0,1]$, and $\rho_{C^\circ}$ denotes the radial function of the polar body $C^{\circ}$ of $C$ (cf. \eqref{radialfunction}), then
\begin{eqnarray*}
C^\circ\cap\{ x_{n+1}=t \}&=& \{ (z,t): \rho_{C^\circ}(z,t)\geq1 \}=\{ t(z/t,1): \rho_{C^\circ}(z/t,1)\geq t \}\\
&=&\{ t(y,1): \rho_{C^\circ}(y,1)\geq t \}=
t \{(y,1): h_{C}(y,1)\leq 1/t \}
\end{eqnarray*}
is also a dilate of $K^\circ$. This argument can be repeated for the sections of $C^\circ$ corresponding to $t\in [-1,0)$ and, by continuity, for $C^\circ\cap\{x_{n+1}=0\}$.  Thus there exists a concave function $f$ on $[-1,1]$ such that
\[
C^\circ=\{(f(t)z,t)\in \R^n\times \R : t\in[-1,1], z\in K^\circ\}.
\]
Let $s\in[-1,1]$. A point $(y,s)\in \R^n \times \R$ belongs to $C$ if and only if
\[
 f(t)(y\cdot z)+s\,t \leq 1\quad\forall z\in K^\circ,\quad\forall t\in[-1,1],
\]
that is, if and only if
\[
 y\cdot z\leq \min_{t\in-[1,1]}\frac{1-s\,t}{f(t)}\quad\forall z\in K^\circ.
\]
The last formula shows that $C\cap\{x_{n+1}=s\}$ equals $K$ dilated by the factor $\min_{t\in-[1,1]}(1-s\,t)/f(t)$.
\end{proof}

The special case $r=R$ of the following lemma will be needed for the main results of this section, while the general case is applied in Section~\ref{anisotropic}.

\begin{lem}\label{cor1may21}
Let $X={\mathcal{S}}(\R^n)$ or $\cV(\R^n)$, let $K\in \cK^n$ satisfy $rB^n\subset K\subset RB^n$ for $0<r\le R$, and let $T:X\to X$ be a $K$-smoothing rearrangement.  If $f\in X$ is Lipschitz with Lipschitz constant $L$, then there is a Lipschitz function $F:\R^n\to \R$, with Lipschitz constant at most $LR/r$, such that $F(x)=Tf(x)$ for $\cH^n$-almost all $x\in \R^n$.
\end{lem}

\begin{proof}
Let $f\in X$ be Lipschitz with Lipschitz constant $L$ and let $A$ be the set of points of approximate continuity of $Tf$. Since $f$ is continuous, Lemma~\ref{modulus_appr_contin} and Theorem~\ref{lemdec194}(ii) yield
$$
\sup_{\|x-y\|\le dr;~ x,y\in A}|Tf(x)-Tf(y)|\le \sup_{\|x-y\|\le dR}|f(x)-f(y)|.$$
Let $x,y\in A$, $x\neq y$, and choose $d>0$ so that $\|x-y\|=dr$. Then, by the previous inequality,
\begin{eqnarray*}
|Tf(x)-Tf(y)|&\le &\sup_{\|w-z\|\le dr;~ w,z\in A}|Tf(w)-Tf(z)|\\
&\le &\sup_{\|w-z\|\le dR}|f(w)-f(z)|\leq LdR=(LR/r)\|x-y\|.
\end{eqnarray*}
Therefore $Tf$ is Lipschitz on $A$ with Lipschitz constant at most $LR/r$.  By the McShane-Whitney extension theorem (see, e.g., \cite[p.~202]{Fed}) there is a function $F:\R^n\to \R$ with the same Lipschitz properties as $Tf$ on the entire space $\R^n$, such that $F=Tf$ on $A$.  Since $\cH^n(\R^n\setminus A)=0$, the proof is complete.
\end{proof}

Recall that a Young function is a left-continuous and convex function $\Phi:[0,\infty)\to[0,\infty]$ with $\Phi(0)=0$.

\begin{thm}\label{polyaszego_lip}
Let $X={\mathcal{S}}(\R^n)$ or $\cV(\R^n)$, let $T:X\to X$ be a rearrangement, and let $\Phi$ be a Young function.  If $T$ is smoothing and $f\in X$ is Lipschitz, then $Tf$ coincides with a Lipschitz function $\cH^n$-almost everywhere on $\R^n$, and
\begin{equation}\label{eq_polyaszego_lip}
\int_{\{x:~Tf(x)\geq a\}}\Phi\left(\|\nabla Tf(x)\|\right)\, dx\leq \int_{\{x:~f(x)\geq a\}}\Phi\left(\|\nabla f(x)\|\right)\, dx
\end{equation}
for each $a> \essinf f$.
Hence
\begin{equation}\label{eqoct171}
\int_{\R^n}\Phi\left(\|\nabla Tf(x)\|\right)\, dx\leq \int_{\R^n}\Phi\left(\|\nabla f(x)\|\right)\, dx,
\end{equation}
where the integrals may be infinite.
\end{thm}

\begin{proof}
Without loss of generality, we may assume that $\Phi$ is a nontrivial real-valued function. Indeed, the result is obvious if $\Phi\equiv 0$, and if $\Phi$ attains the value $\infty$, it does so on some maximal interval $(t_0,\infty)$, $t_0\ge 0$.  Suppose first that $t_0>0$. For $0<t<t_0$, the right derivative $\Phi^{\prime+}(t)$ of $\Phi$ at $t$ is increasing, so we may define $c=\lim_{t\to t_0-} \Phi^{\prime+}(t)$.  If $c=\infty$, let $(t_k)$ be a strictly increasing sequence in $(0,t_0)$ converging to $t_0$, and define the real-valued Young functions
$$
\Phi_k(t)=\begin{cases}
\Phi(t),& \text{if } 0\le t<t_k,\\
\Phi^{\prime+}(t_k)(t-t_k)+\Phi(t_k), & \text{otherwise.}
\end{cases}
$$
If $c<\infty$, we must have $\Phi(t_0)<\infty$ by left continuity, and may define
$$
\Phi_k(t)=\begin{cases}
\Phi(t),& \text{if } 0\le t\le t_0,\\
(c+k)(t-t_0)+\Phi(t_0), & \text{otherwise.}
\end{cases}
$$
When $t_0=0$, we put $\Phi_k(t)=kt$.  In each case, we have $\Phi_k\le \Phi$ and $(\Phi_k)$ is an increasing sequence of real-valued Young  functions converging pointwise to $\Phi$. Thus, if \eqref{eq_polyaszego_lip} holds for real-valued Young functions, it holds with $\Phi$ replaced by $\Phi_k$, and hence, by the monotone convergence theorem, for $\Phi$ itself.

The set of values $t$ such that $\cH^{n}(\{x :  f(x)= t\})=0$ is dense in $(\essinf f, \infty)$, so there is an increasing sequence $\{a_m\}$ contained in $(\essinf f,a)$ and converging to $a$ such that $\cH^{n}(\{x :  f(x)= a_m\})=0$ for each $m$.  Fix $m\in \N$.  By Lemma~\ref{lemoct17}, $\{x: f(x)\ge a_m\}$ is bounded, and the equimeasurability of $T$ implies that
\[
 \cH^{n}(\{x :  Tf(x)= a_m\})=\cH^{n}(\{x :  f(x)= a_m\})=0.
\]
Assume that $L$ is the Lipschitz constant for $f$.  As $T$ reduces the modulus of continuity by Corollary~\ref{lemdec301}, Lemma~\ref{cor1may21} with $r=R$ implies that there is a Lipschitz function $F$ on $\R^n$ with Lipschitz constant at most $L$ such that $F(x)=Tf(x)$ for $\cH^n$-almost all $x\in \R^n$.  Weak derivatives and the remainder of this proof are unaffected by changing $Tf$ on a set of measure zero, so we may assume that $Tf$ itself is Lipschitz on $\R^n$ with Lipschitz constant at most $L$.  Then, by Lemmas~\ref{minkcont_diminishes}  and~\ref{minkcontequalhausdorff} (the latter applied to both $f$ and $Tf$) with $a$ replaced by $a_m$, and the equimeasurability of $T$, we obtain
\begin{equation*}
    \int_{G_{Tf}\cap \{x_{n+1}> a_m\}}h_C(\nu_{Tf}(x))\,d \cH^n(x)\leq \int_{G_f\cap \{x_{n+1}> a_m\}}h_C(\nu_{f}(x))\,d \cH^n(x),
\end{equation*}
where $C$ is any convex body as in Lemma~\ref{lemdec303}, and $\nu_g(x)$ is the outer unit normal of $K_g$ at $x$ for a Lipschitz function $g$.  The integral on the right can be written as
\begin{equation*}
    \int_{\{y: f(y)> a_m\}}h_C\left(\frac{(-\nabla f(y),1)}{\sqrt{1+\|\nabla f(y)\|^2}}\right){\sqrt{1+\|\nabla f(y)\|^2}}\,dy=\int_{\{y: f(y)> a_m\}}h_C(-\nabla f(y),1)\,dy,
\end{equation*}
where we used the $1$-homogeneity of $h_C$.  Similarly, the integral on the left can be rewritten in the same form, with $f$ replaced by $Tf$.  Consequently,
\begin{equation*}
    \int_{\{y: Tf(y)> a_m\}}h_C(-\nabla Tf(y),1)\,dy\leq \int_{\{y: f(y)> a_m\}}h_C(-\nabla f(y),1)\,dy,
\end{equation*}
which also yields
\begin{equation}\label{eqoct111}
     \int_{\{y: Tf(y)> a_m\}}h_C(-\nabla Tf(y),1)-1\,dy\leq \int_{\{y: f(y)> a_m\}}h_C(-\nabla f(y),1)-1\,dy,
\end{equation}
since $\cH^n(\{y: Tf(y)> a_m\})=\cH^n(\{y: f(y)> a_m\})$.  If $\ell$ is the Lipschitz constant for some Lipschitz function $g$, we have $\|\nabla g(x)\|\le \ell$ for $x\in \R^n$ where the derivative of $g$ exists. Applying this to $g=f$ and $g=Tf$ implies that
\begin{equation}\label{eqoct111a}
\max\left\{\|\nabla f(x)\|,\|\nabla Tf(x)\|\right\}\le L
\end{equation}
for ${\mathcal{H}}^{n}$-almost all $x\in \R^n$.  As $\Phi$ is real-valued and not identically $0$, we may define the convex body $C$ as in Lemma~\ref{lemoct20} corresponding to $K=B^n$, $M=L$, and $\Phi$. Then $C$ satisfies the conditions stated in Lemma~\ref{lemdec303} and by \eqref{eqoct201},
$$
h_C(y,1)-1=b\,\Phi(\|y\|),
$$
for $y\in \R^n$ with $\|y\|\le M$ and some $b>0$.  Substituting in \eqref{eqoct111} and taking limits as $m\to\infty$, we obtain \eqref{eq_polyaszego_lip}.

By Proposition~\ref{may8lem}, we have $\essinf Tf=\essinf f$.  Letting $a\to \essinf f$ in \eqref{eq_polyaszego_lip}, we arrive at \eqref{eqoct171}.
\end{proof}

Theorem~\ref{polyaszego_lip} yields P\'olya-Szeg\H{o} inequalities with $\Phi(t)=t^p$, $1\le p<\infty$, and it is possible to extend these to functions in $W^{1,p}(\R^n)\cap \cV(\R^n)$ by means of standard techniques, utilized for example in the proof of \eqref{intro1} in \cite[Theorem~3.20]{Baer} (though the case $p=1$ needs extra work). We shall not state this extension here, however, but instead derive it from a still more general result, Theorem~\ref{polyaszego_w12 new}, proved in Section~\ref{Orlicz}.  See Corollary~\ref{polyaszego_w12}.  We also postpone to Section~\ref{Orlicz} a proof that the assumption in Theorem~\ref{polyaszego_lip} that $T$ is smoothing cannot be dropped in general; see Example~\ref{exaug721}.  This example also shows that when $\Phi(t)=t$ in Theorem~\ref{polyaszego_lip}, the smoothing property is not necessary for the stated inequalities to hold.

\begin{cor}\label{area_diminishes}
Let $X={\mathcal{S}}(\R^n)$ or $\cV(\R^n)$, and let $T:X\to X$ be a rearrangement.  If $T$ is smoothing, $f\in X$ is Lipschitz, and $a> \essinf f$, then
\begin{equation*}
   \cH^n(G_{Tf}\cap \{x_{n+1}\ge a\})\leq \cH^n(G_f\cap \{x_{n+1}\ge a\}).
\end{equation*}
\end{cor}

\begin{proof}
Since $f$ is Lipschitz, we have the familiar formula (see, e.g., \cite[p.~101]{EG})
$${\cH}^n(G_f\cap \{x_{n+1}\ge a\})=\int_{\{x: f(x)\ge a\}}\left(1+\|\nabla f(x)\|^2\right)^{1/2}\,dx.$$
Noting that $T$ reduces the modulus of continuity, by Corollary~\ref{lemdec301}, we may, as in the proof of Theorem~\ref{polyaszego_lip}, assume that $Tf$ is also Lipschitz, so the same formula holds with $f$ replaced by $Tf$.  The result now follows from \eqref{eq_polyaszego_lip} on setting $\Phi(t)=\sqrt{1+t^2}-1$ for $t\ge 0$, and recalling that $\cH^n(\{x: Tf(x)\geq a\})=\cH^n(\{x: f(x)\geq a\})$.
\end{proof}

\section{Extension of Theorem~\ref{polyaszego_lip} to $W^{1,1}_{loc}(\R^n)$}\label{Orlicz}

We begin by recalling some notions from the theory of Orlicz spaces.  Much of the literature is based on N-functions (nice Young functions), that is, functions $\Phi:[0,\infty)\to[0,\infty)$ that are continuous, convex, and such that $\Phi(t)=0$ if and only if $t=0$, $\lim_{t\to0}\Phi(t)/t=0$, and $\lim_{t\to\infty} \Phi(t)/t=\infty$.   However, this restriction is sometimes unnecessary, so we prefer to follow \cite{ES} and \cite{Zaanen} by working with nontrivial Young functions $\Phi$, those such that $\Phi\not\equiv 0$ and $\Phi\not\equiv\infty$ on $(0,\infty)$.

Let $\Phi:[0,\infty)\to[0,\infty]$ be a nontrivial Young function.  The \emph{Orlicz space} $\lphiom$, where $\Omega\subset\R^n$ is an open set, is the set of all real-valued measurable functions $f$ on $\Omega$  such that
\begin{equation}\label{A28eq1}
\|f\|_{\lphiom}=\inf \left\{ \lambda>0 : \int_{\Omega}\Phi\left(\frac{|f(x)|}{\lambda}\right)\, dx\leq 1\right\}<\infty.
\end{equation}
The norm defined by \eqref{A28eq1} is called the \emph{Luxemburg norm}.

Note that $f\in \lphiom$ if and only if there is a $c>0$ such that $\int_{\Omega}\Phi(c|f(x)|)\, dx$ is finite.  Indeed, if the latter condition holds, we can choose $\lambda>0$ large enough that $1/\lambda < c$ and $\int_{\Omega}\Phi(c |f(x)|)\, dx \le c\lambda$. Then, since $\Phi(t)/t$ is increasing, we have $\lambda \Phi(|f(x)|/ \lambda) \le \Phi(c|f(x)|) /c$ and hence
$$\int_{\Omega} \Phi\left(\frac{|f(x)|}{\lambda}\right)\, dx \le \frac{1}{c\lambda}\int_{\Omega} \Phi(c|f(x)|)\,dx \le 1.$$

If $f, f_k\in \lphiom$, $k\in \N$, then
\begin{equation}\label{eq:LPhiConv}
\lim_{k\to\infty}\|f_k-f\|_{\lphiom}=0\iff
\lim_{k\to\infty}\int_{\Omega} \Phi\left(c|f_k(x)-f(x)|\right)\, dx =0,
\text{ for all } c>0,
\end{equation}
by \cite[Proposition~2.1.10(5)]{ES}.

The {\em complementary function $\Psi:[0,\infty)\to [0,\infty]$ to $\Phi$}, defined by $\Psi(t)=\sup_{s\geq0} (st-\Phi(s))$, is also a nontrivial Young function, sometimes called the conjugate function of $\Phi$.  Another norm on $\lphiom$ is the \emph{Orlicz norm}
\[
 \|g\|'_{\lphiom}=\sup \left\{ \int_{\Omega}|g(x) h(x)|\, dx: \int_{\Omega}\Psi(|h(x)|)\, dx\leq 1\right\};
\]
by \cite[Theorem~132.2]{Zaanen}, this norm is equivalent to the Luxemburg norm.

The set
\begin{equation}\label{Hdefin}
H^{\Phi}(\Omega)=\left\{f\in L^\Phi(\Omega) : \int_{\Omega}\Phi(c|f(x)|)\, dx<\infty \text { for all }c>0\right\}
\end{equation}
is called the \emph{heart} of $\lphiom$.

We say that the Young function $\Phi$ \emph{satisfies a $\Delta_2$ condition at infinity} if there exist $c, t_0>0$ such that
$$
\Phi(2t)\le c\,\Phi(t)
$$
for $t\geq t_0$. If $\Phi$ is in addition nontrivial and real-valued, and  if $\cH^{n}(\Omega)<\infty$, then $H^{\Phi}(\Omega)=\lphiom$; see \cite[Theorem~2.1.17(2)]{ES}.

The \emph{Orlicz-Sobolev space} $W^{1,\Phi}(\Omega)$ is defined as
\[
W^{1,\Phi}(\Omega)=\{f : f\in \lphiom, f\text{  is weakly differentiable in $\Omega$, and }
 \|\nabla f\|\in\lphiom\}.
\]
The norm of $f\in W^{1,\Phi}(\Omega)$ is defined by
\begin{equation}\label{Wnorm}
\|f\|_{W^{1,\Phi}(\Omega)}=\|f\|_\lphiom+\big\| \|\nabla f\| \big\|_\lphiom.
\end{equation}

Let $\Phi_1$ and $\Phi_2$ be nontrivial Young functions. We say that $\Phi_1$ \emph{dominates $\Phi_2$ globally} and write $\Phi_1\succ \Phi_2$ if there are constants $a,b>0$ such that
\begin{equation}\label{eq:dominate}
b\,\Phi_1(at)\ge \Phi_2(t)
\end{equation}
for all $t\ge 0$. If $\Phi_1\succ \Phi_2$, then $L^{\Phi_1}(\Omega)\subset L^{\Phi_2}(\Omega)$ by \cite[Theorem~2.2.3(1)]{ES}.  If both $\Phi_1\succ \Phi_2$ and $\Phi_2\succ \Phi_1$, we say that $\Phi_1$ and $\Phi_2$ are \emph{equivalent}. In this case, $L^{\Phi_1}(\Omega)= L^{\Phi_2}(\Omega)$ and by \cite[Proposition~2.2.1]{ES}, the norms $\|\cdot\|_{L^{\Phi_1}(\Omega)}$ and $\|\cdot\|_{L^{\Phi_2}(\Omega)}$ are equivalent.

Consider the complementary pair of nontrivial Young functions
\[
\Phi_{\text{min}}(t)=\begin{cases} 0, &\text{ if } 0\le t\le 1,\\t-1, &\text{ if } 1<t<\infty,
\end{cases}
\qquad
\Phi_{\text{max}}(t)=\begin{cases} t, &\text{ if } 0\le t\le 1,\\\infty, &\text{ if } 1<t<\infty;
\end{cases}
\]
see \cite[p.~54]{ES}.  The following facts are gathered in \cite[Proposition~2.2.4]{ES} and the remarks that follow it.  We have
\[
L^{\Phi_{\text{min}}}(\R^n)=\{f_1+f_\infty: f_1\in L^1(\R^n),f_\infty\in L^\infty(\R^n)\},
\]
with Luxemburg norm
\begin{equation}\label{eq:LuxLargest} \|f\|_{L^{\Phi_{\text{min}}}(\R^n)}=
\inf\{\max\{\|f_1\|_{L^1(\R^n)},\|f_\infty\|_{L^\infty(\R^n)}\}: f=f_1+f_\infty\}
\end{equation}	
for $f\in L^{\Phi_{\text{min}}}(\R^n)$, and $L^{\Phi_{\text{max}}}(\R^n)=L^1(\R^n)\cap L^\infty(\R^n)$, with 	
Luxemburg norm
\begin{equation}\label{eq:LuxSmallest}
	 \|f\|_{L^{\Phi_{\text{max}}}(\R^n)}
=\max\{\|f\|_{L^1(\R^n)},\|f\|_{L^\infty(\R^n)}\}
\end{equation}	
for $f\in L^{\Phi_{\text{max}}}(\R^n)$. We will often write $L^{\Phi_{\text{min}}}(\R^n)=L^1(\R^n)+L^\infty(\R^n)$.  Any nontrivial Young function $\Phi$ satisfies $\Phi_{\text{max}}\succ \Phi\succ \Phi_{\text{min}}$ and hence all Orlicz spaces with nontrivial Young functions contain $L^1(\R^n)\cap L^\infty(\R^n)$ and are contained in $L^1(\R^n)+L^\infty(\R^n)$.

The proof of the following lemma is a variant of arguments in \cite[p.~77]{RaoRen}.

\begin{lem}\label{lemM}
Let $Z$ be the closure of
\begin{equation}\label{eq:Prop3'RenRao}
\left\{\sum_{j=1}^k a_j1_{A_j}: a_1,\ldots,a_k\in \R, A_1,\ldots,A_k\in {\mathcal L}^n, k\in \N\right\}
\end{equation}
in $L^1(\R^n)+L^\infty(\R^n)$. Then $Z=L^1(\R^n)+L^\infty(\R^n)$.
\end{lem}

\begin{proof}
The step functions in the set \eqref{eq:Prop3'RenRao} are clearly bounded and integrable, so we only have to show that $L^1(\R^n)+L^\infty(\R^n)\subset Z$.
	
Let $f\in L^1(\R^n)+L^\infty(\R^n)$ and suppose that $|f(x)|<M$ for $\cH^n$-almost all $x\in \R^n$. For $k\in \N$ define $I_j=[jM/k,(j+1)M/k)$, $A_j=f^{-1}(I_j)$, and
\[
a_j=\begin{cases}
\min I_j, & \text{ if } j\ge 0,\\
\sup I_j, & \text{ otherwise,}
\end{cases}
\]
for $j=-k,\ldots,k-1$.  For $j\in \{-1,0\}$ we have $a_j=0$, and for all other $j$,
\[
|a_j|\cH^n(A_j)\le \int_{A_j} |f(x)|dx<\infty,
\]
as $f\in L^1(\R^n)$.  Therefore
\[
f_k=\sum_{j=-k}^{k-1} a_j1_{A_j}\in Z.
\]
By construction, $|f_k(x)-f(x)|\le M/k$ for $\cH^n$-almost all $x\in \R^n$, so $\|f_k-f\|_\infty\to 0$ as $k\to \infty$.  Since $|f_k(x)|\le |f(x)|$ for $\cH^n$-almost all $x\in \R^n$, the dominated convergence theorem yields $\|f_k-f\|_1\to 0$ as $k\to\infty$.  It follows that $f_k\to f$ as $k\to\infty$ in the norm \eqref{eq:LuxSmallest}, proving the lemma.
\end{proof}

By definition, a subset $S$ of $L^1(\R^n)+L^\infty(\R^n)$ is relatively sequentially compact in the weak topology $\sigma\left(L^1(\R^n)+L^\infty(\R^n),L^1(\R^n)\cap L^\infty(\R^n)\right)$ if any sequence $(f_j)$ of functions in $S$ has a subsequence $(f_{j_k})$ converging to some $f\in L^1(\R^n)+L^\infty(\R^n)$, in the sense that
\[
\lim_{k\to \infty} \int_{\R^n} (f_{j_k}(x)-f(x))h(x)\,dx=0
\]
for all $h\in L^1(\R^n)\cap L^\infty(\R^n)$. As $\sigma\left(L^1(\R^n)+L^\infty(\R^n),L^1(\R^n)\cap L^\infty(\R^n)\right)$ is the only weak topology used in this section, henceforth weak convergence and compactness will always refer to this topology.  If $S$ is bounded in the norm \eqref{eq:LuxLargest}, it is relatively weakly sequentially compact if and only if there is a real-valued Young function $\widetilde{\Phi}$ with $\lim_{t\to \infty} \widetilde{\Phi}(t)/t=\infty$ and
\begin{equation}\label{eq:ValleePoussin}
\sup_{f\in S}	\int_{\R^n} \widetilde{\Phi}(f(x))\,dx <\infty.
\end{equation}
This is essentially the criterion of de La Vall\'ee Poussin \cite[Theorem~2.3.5]{ES} combined with \cite[Proposition~2.3.16]{ES}, where the latter uses the $\sigma$-finiteness of the underlying measure in the proof.

For the reader's convenience we provide a proof of the following approximation lemma, mainly to stress that $\Phi$ need not be an N-function.

\begin{lem}\label{approxx_by_C_infty}
If $\Phi:[0,\infty)\to [0,\infty)$ is a real-valued Young function and $f\in H^{\Phi}(\R^n)$, there is a sequence $(f_j)$ of $C^\infty(\R^n)$ functions with compact supports that converges to $f$ in $L^{\Phi}(\R^n)$.
	
If, in addition, $f$ is nonnegative, or has bounded support, or $\|\nabla f\|\in H^{\Phi}(\R^n)$, then $f_j$, $j\in \N$, can be chosen to be nonnegative, or have uniformly bounded supports, or such that $\lim_{j\to 0}\|f_j-f\|_{W^{1,\Phi}(\R^n)}=0$, respectively.
\end{lem}

\begin{proof}
We may assume that $f\in H^{\Phi}(\R^n)$ has compact support. Indeed, let $0\le \phi\le 1$ be a $C^\infty(\R^n)$ function with support in $2B^n$ and $\phi(x)=1$ for all $x\in B^n$, and let $\phi_m(x)=\phi(x/m)$ for $m\in \N$ and $x\in \R^n$.  Then $\phi_mf\in H^{\Phi}(\R^n)$ has compact support, and using \eqref{eq:LPhiConv} and \eqref{Hdefin}, it is easy to see that $\phi_mf$ converges to $f$ as $m\to\infty$ in $\lphir$.   If $\|\nabla f\|\in H^{\Phi}(\R^n)$ is assumed in addition, this convergence even holds in $W^{1,\Phi}(\R^n)$, in view of \eqref{eq:LPhiConv} and \eqref{Wnorm}.  Indeed, since $\Phi$ is convex, we have, for $s,t\geq0$,
\[
\Phi(s+t)\leq (1/2)\left(\Phi(2s)+ \Phi(2t)\right).
\]
Therefore
\[
\Phi\left(c\|\nabla (f(1-\phi_m)\right)\|)\leq (1/2) \left( \Phi(2c\|\nabla f\|) +\Phi\left( (2c/m)|f|\sup_{\mathbb{R}^n}\|\nabla \phi\| \right) \right).
\]
Since $\|\nabla f\|\in H^{\Phi}(\R^n)$, the integrals of both terms on the right-hand side are finite, and the desired conclusion follows from the dominated convergence theorem. Whenever $f$ is nonnegative, we also have $\phi_m f\ge 0$.  Thus if the result holds for functions with compact support, a standard diagonal-type argument shows that it holds generally.
	
Suppose that $f\in H^{\Phi}(\R^n)$ has compact support.  Let $\rho:\R^n\to [0,\infty)$ be a $C^\infty(\R^n)$ function with support in $B^n$ and  integrating to $1$, and define $\rho_j(x)=j^{n}\rho(jx)$ for $j\in \N$ and $x\in \R^n$.  The convolution $f_j=f *\rho_j$ is $C^\infty(\R^n)$, satisfies $f_j\ge 0$ when $f\ge 0$, and has support in $M+B^n$ when the support of $f$ is $M\subset \R^n$.  We claim that $\|f- f_j\|_{L^{\Phi}(\R^n)}\to 0$ as $j\to \infty$.  To see this, note that since $\Phi$ is real-valued, there is a sequence of integrable simple functions (i.e., finite weighted sums of integrable characteristic functions) converging to $f$ in the $L^{\Phi}(\R^n)$ norm, by \cite[Theorem~2.1.14(b)]{ES}. Using monotone convergence, the proof of the latter theorem can easily be modified to show that these simple functions can be assumed to have bounded supports. Thus, for each $\varepsilon>0$ there is a simple function $h$ with $|h(x)|\le b 1_{RB^n}(x)$ for some $b,R>0$ and all $x\in \R^n$, such that $\|f- h\|_{L^{\Phi}(\R^n)}\le \varepsilon/4$. If $h_j=h*\rho_j$, Jensen's inequality (see, e.g., \cite[p.~62, Proposition~5]{RaoRen}) implies that
\begin{equation*}
\|f_j- h_j\|_{L^{\Phi}(\R^n)}=\|(f- h)*\rho_j\|_{L^{\Phi}(\R^n)}\le \|f- h\|_{L^{\Phi}(\R^n)}\le \varepsilon/4,
\end{equation*}
so
\begin{align}\nonumber
\|f- f_j\|_{L^{\Phi}(\R^n)}&\le \|f- h\|_{L^{\Phi}(\R^n)}+\|h- h_j\|_{L^{\Phi}(\R^n)}+\|h_j- f_j\|_{L^{\Phi}(\R^n)}
\\&\label{eqgg_k}
\le \frac{\varepsilon}2+\|h- h_j\|_{L^{\Phi}(\R^n)}.
\end{align}
We have
\[
|h(x)-h_j(x)|\le \int_{\R^n} |h(x)-h(y)|\rho_j(x-y) dy\le \|\rho_j\|_\infty\int_{B(x,1/j)} |h(x)-h(y)|\,dy.
\]
As $h$ is integrable, the Lebesgue differentiation theorem (see, e.g., \cite[Proposition~3.5.4]{KP}) implies that $h_j(x)\to h(x)$ for ${\mathcal{H}}^{n}$-almost all $x\in \R^n$. Since $|h_j(x)|\le b 1_{(R+1)B^n}(x)$ for all $x\in \R^n$ and $\Phi$ is real-valued, the dominated convergence theorem gives $\lim_{j\to \infty}\int_{\R^n}\Phi(c|h(x)- h_j(x)|)\,dx=0$ for all $c>0$.  Therefore, by \eqref{eq:LPhiConv}, there is a $j_0\in \N$ such that $\|h- h_j\|_{L^{\Phi}(\R^n)}\le \varepsilon/2$ for all $j\ge j_0$, and inserting this into \eqref{eqgg_k} proves the claim.
	
Since $f_j\ge 0$ when $f\ge 0$, it only remains to deal with the statement involving the extra assumption that $\|\nabla f\|\in H^{\Phi}(\R^n)$.  For this, we may proceed exactly as in the proof of \cite[Theorem~2.1]{DT}. If $f$ and $\|\nabla f\|$ are functions in $H^{\Phi}(\R^n)$ with compact support, we have $\partial f_j/\partial x_i =(\partial f/\partial x_i)*\rho_j$ (where the weak derivative is used on the right-hand side). Applying the above, both to $f$ and with $f$ replaced by $\partial f/\partial x_i$, $i=1,\ldots,n$, yields $\lim_{j\to 0}\|f-f_j\|_{W^{1,\Phi}(\R^n)}=0$, as required.
\end{proof}

\begin{thm}\label{polyaszego_w12 new}
Let $T:\cV(\R^n)\to \cV(\R^n)$ be a smoothing rearrangement and let $\Phi$ be a Young function.  If $f\in W^{1,1}_{loc}(\R^n)\cap\cV(\R^n)$ and $\int_{\R^n}\Phi(\|\nabla f(x)\|)\, dx<\infty$, then $Tf\in W^{1,1}_{loc}(\R^n)$ and
\begin{equation}\label{eq_polyaszego_w1phi}
\int_{\R^n}\Phi(\|\nabla T f(x)\|)\, dx\leq \int_{\R^n}\Phi(\|\nabla f(x)\|)\, dx.
\end{equation}
\end{thm}

\begin{proof}
As at the beginning of the proof of Theorem~\ref{polyaszego_lip}, we may assume that $\Phi$ is a nontrivial real-valued function. We first aim to prove \eqref{eq_polyaszego_w1phi} when $\cH^n(\supp{f})<\infty$. Let $\Omega$ be an open set such that $\supp f\subset \Omega$ and $\cH^n(\Omega)<\infty$.  The proof will proceed via a succession of claims.

For $r>\sup\{t\ge 0:\Phi(t)=0\}$, define
\begin{equation*}
\Lambda_r(t)=\begin{cases}
\Phi(t),& \text{if } 0\le  t<r,\\
\Phi^{\prime+}(r)(t-r)+\Phi(r), & \text{otherwise,}
\end{cases}
\end{equation*}
where $\Phi^{\prime+}(r)$ is the right derivative of $\Phi$ at $r$.
Then $\Phi_r=\max\{0,\Lambda_r-1/r\}$ is a nontrivial real-valued (and hence continuous) Young function with $\Phi_r\le \Phi$. By construction, there is a $\delta>0$ such that
$\delta\le \Phi_r^{\prime+}(t)\le 1/\delta$ for all $t>t_0=\sup\{t\ge 0:\Phi_r(t)=0 \}>0$.  It follows that for $t\ge t_0$, we have
\[
\delta(t-t_0)=\int_{t_0}^t \delta\,ds\le \Phi_r(t)\le \int_{t_0}^t \frac1{\delta}\,ds=
\frac1{\delta}(t-t_0).
\]
Comparing \eqref{eq:dominate}, we see that  $\Phi_r$ is equivalent to $\Phi_{\text{min}}$. Therefore, as we remarked before Lemma~\ref{approxx_by_C_infty}, the $L^{\Phi_r}(\R^n)$ norm is equivalent to the norm \eqref{eq:LuxLargest} and hence $L^{\Phi_r}(\R^n)=L^{1}(\R^n)+L^{\infty}(\R^n)$. If $\Psi_r$ denotes the complementary Young function to $\Phi_r$, then
$L^{\Psi_r}(\R^n)=L^{1}(\R^n)\cap L^{\infty}(\R^n)$ and the $L^{\Psi_r}(\R^n)$ norm is equivalent to \eqref{eq:LuxSmallest}.

Our first claim is that $f$ can be approximated in the $W^{1,\Phi_r}(\R^n)$ norm by a sequence $(f_j)$ of nonnegative $C^\infty(\R^n)$ functions with compact support.  To see this, note that since $0\le\Phi_r\le \Phi$, our assumptions give $\int_{\R^n}\Phi_r(\|\nabla f(x)\|)\, dx<\infty$.  By \cite[Lemma~3]{Talenti}, this yields the existence of a $c>0$ such that $\int_{\R^n}\Phi_r(c|f(x)|)\, dx$ is finite, and hence $f\in L^{\Phi_r}(\R^n)$, as was explained in the remarks at the beginning of this section. Therefore $f\in L^{\Phi_r}(\Omega)$ and our assumptions and $0\le\Phi_r\le \Phi$ imply that $\|\nabla f\|\in L^{\Phi_r}(\Omega)$.
As was mentioned before Lemma~\ref{lemM}, $L^{\Phi_r}(\Omega)=H^{\Phi_r}(\Omega)$ because $\Phi_r$ satisfies a $\Delta_2$ condition at infinity and $\cH^n(\Omega)<\infty$. Since $f$ vanishes on $\R^n\setminus\Omega$, we have $f, \|\nabla f\|\in H^{\Phi_r}(\R^n)$. By Lemma~\ref{approxx_by_C_infty},  there exists a sequence $(f_j)$ of  nonnegative $C^\infty(\R^n)$ functions with compact support that converges to $f$ in the $W^{1,\Phi_r}(\R^n)$ norm.  This completes the proof of the first claim.

Our second claim is that $(T f_j)$ converges to $Tf$ in $\lphir$, $Tf\in W^{1,1}_{loc}(\R^n)$, and a subsequence of $(\nabla T f_j)$ converges weakly to the weak gradient of $Tf$. Indeed, Proposition~\ref{contraction} implies that $\|Tf_j-Tf\|_\lphir\leq\|f_j-f\|_\lphir$. Thus $Tf_j\to Tf$ in $\lphir$ and, in particular, $Tf\in\lphir$ and $Tf\in L^1_{loc}(\R^n)$.  We also have
\begin{equation}\label{lim_grad_approx}
 \lim_{j\to\infty} \big\|\|\nabla f_j\|\big\|_\lphir =\big\|\|\nabla f\|\big\|_\lphir.
\end{equation}
The map $T$ and each $f_j$ satisfy the hypotheses of Theorem~\ref{polyaszego_lip}, so for $j\ge 1$, the latter theorem implies that the nonnegative function $Tf_j$ agrees with a Lipschitz function $\cH^n$-almost everywhere, and
\begin{equation}\label{ineq_grad_approx}
 \big\|\|\nabla Tf_j\|\big\|_\lphir\leq \big\|\|\nabla f_j\|\big\|_\lphir.
\end{equation}
Due to \eqref{lim_grad_approx} and \eqref{ineq_grad_approx}, the set $\{\|\nabla T f_j\|, j\geq 1\}$ is bounded in ${L^{\Phi_r}(\R^n)}$ and therefore also in ${L^{\Phi_{\text{min}}}(\R^n)}$ by norm equivalence.
Since  $\|\nabla f_j\|$ converges in $\lphir$, it also converges in ${L^{\Phi_{\text{min}}}(\R^n)}$, and the criterion of de La Vall\'ee Poussin stated before Lemma~\ref{approxx_by_C_infty} yields the existence of a real-valued Young function $\widetilde{\Phi}$ with
$\lim_{t\to \infty} \widetilde{\Phi}(t)/t=\infty$ such that \eqref{eq:ValleePoussin} holds with $S=\{\|\nabla f_j\|:j\ge 1\}$. Theorem~\ref{polyaszego_lip}, applied with $\widetilde{\Phi}$ instead of $\Phi$, shows that \eqref{eq:ValleePoussin} also holds for
$S=\{\|\nabla Tf_j\|:j\ge 1\}$ with the same $\widetilde{\Phi}$. De La Vall\'ee Poussin's criterion now shows the relative weak compactness of $\{\|\nabla Tf_j\|:j\ge 1\}$. Hence, there is a subsequence of $(\nabla T f_j)$, also denoted $(\nabla T f_j)$, and a vector field $g\in\left(L^1(\R^n)+ L^\infty(\R^n)\right)^n=\left(\lphir\right)^n$, such that $\nabla T f_j$ converges weakly to $g$, i.e., for each $h\in \left(L^1(\R^n)\cap L^\infty(\R^n)\right)^n=\left(L^{{\Psi_r}}(\R^n)\right)^n$,
$$
\lim_{j\to\infty} \int_{\R^n} (\nabla Tf_j)(x)h(x)\,dx=\int_{\R^n}g(x)h(x)\,dx.
$$
If $h\in C^\infty(\R^n)$ has compact support, then $h, \partial h/\partial x_i\in L^{{\Psi_r}}(\R^n)$ and
\begin{equation}\label{weak_derivative}
\begin{aligned}
 \int_{\R^n}g_i(x) h(x)\, dx&=\lim_{j\to\infty} \int_{\R^n} \frac{\partial T f_j}{\partial x_i}(x) h(x)\, dx
 =-\lim_{j\to\infty} \int_{\R^n} T f_j(x) \frac{\partial h}{\partial x_i}(x)\, dx\\
 &=-\int_{\R^n} T f(x) \frac{\partial h}{\partial x_i}(x)\, dx,
\end{aligned}
\end{equation}
where we have used the fact that the convergence of $Tf_j$ to $Tf$ in $L^{\Phi_r}(\R^n)$ also implies weak convergence. From \eqref{weak_derivative} we see that $g$ is the weak gradient of $T f$, and the fact that $\partial T f/\partial x_i$ is locally integrable is a direct consequence of the fact that it belongs to $L^{{\Psi_r}}(\R^n)$.  The second claim is proved.

Next, we claim that
\begin{equation}\label{ineq_norm_weak_limit}
\big\|\|\nabla T f\|\big\|_\lphir\leq\liminf_{j\to\infty} \big\|\|\nabla  T f_j\|\big\|_\lphir.
\end{equation}
With this goal in mind, we first prove that if $u\in (L^{\Phi_r}(\R^n))^n$, then
\begin{equation}\label{norm_weak_limit}
\big\|\|u\|\big\|_\lphir=\sup \left\{\left| \int_{\R^n} u(x)\cdot h(x)\,dx\right|: h\in (L^{{\Psi_r}}(\R^n))^n, \big\|\|h\|\big\|'_{L^{{\Psi_r}}(\R^n)}\leq 1\right\}.
\end{equation}
A similar formula in the scalar case is proved in \cite[(10), Proposition~10, Section~3.4]{RaoRen}; when applied with $f$ and $\Phi$ there replaced by $\|u\|$ and $\Phi_r$, respectively, it becomes
\begin{equation}\label{norm_weak_limit2M}
	\big\|\|u\|\big\|_\lphir=\sup \left\{\left| \int_{\R^n} \|u(x)\|\, v(x)\,dx\right|: v\in {\mathcal M}^{{\Psi_r}}, \|v\|'_{L^{{\Psi_r}}(\R^n)}\leq 1\right\},
\end{equation}
where ${\mathcal M}^{{\Psi_r}}$ is the closure of the span of all step functions in $L^{{\Psi_r}}(\R^n)$. (We warn the reader that in \cite{RaoRen},  $\|u\|_\lphi$ and $\|u\|'_{\lphi}$ are denoted by $N_\phi(u)$ and $\|u\|_\Phi$, respectively.)  We have already seen that $L^{\Psi_r}(\R^n)=L^{1}(\R^n)\cap L^{\infty}(\R^n)$ and that the two spaces have equivalent norms, so ${\mathcal M}^{{\Psi_r}}=L^{{\Psi_r}}(\R^n)$ due to Lemma \ref{lemM}. Thus,  \eqref{norm_weak_limit2M} becomes
\begin{equation}\label{norm_weak_limit2}
\big\|\|u\|\big\|_\lphir=\sup \left\{\left| \int_{\R^n} \|u(x)\|\, v(x)\,dx\right|: v\in L^{{\Psi_r}}(\R^n), \|v\|'_{L^{{\Psi_r}}(\R^n)}\leq 1\right\}.
\end{equation}
Let $S_1$ and $S_2$ denote the right-hand sides of \eqref{norm_weak_limit} and \eqref{norm_weak_limit2}, respectively.  The Cauchy-Schwartz inequality yields $S_1\leq S_2$. We can restrict the supremum in \eqref{norm_weak_limit2} to nonnegative $v$.  If $v\in L^{{\Psi_r}}(\R^n)$, $v\geq0$, $\|v\|'_{L^{{\Psi_r}}(\R^n)}\leq1$, and
\[
h(x)=
\begin{cases}
v(x) \frac{u(x)}{\|u(x)\|},&  \text{if $u(x)\neq0$},\\
0,&  \text{if $u(x)=0$},
\end{cases}
\]
then $h\in (L^{{\Psi_r}}(\R^n))^n$,  $\big\|\|h\|\big\|'_{L^{{\Psi_r}}(\R^n)}\leq1$, and
\[
S_1\geq \left| \int_{\R^n} u(x)\cdot h(x)\,dx\right|=\left| \int_{\R^n} \|u(x)\| v(x)\,dx\right|.
\]
This proves that $S_1\geq S_2$ and concludes the proof of \eqref{norm_weak_limit}.  Now \eqref{norm_weak_limit} with $u$ replaced by $\nabla T f$ and $\nabla T f_j$ gives
\begin{align*}
\big\|\|\nabla T f\|\big\|_\lphir
  &=\sup\left\{\left| \int_{\R^n} \nabla T f(x)\cdot h(x)\,dx\right|:
    h\in (L^{{\Psi_r}}(\R^n))^n,
    \big\|\|h\|\big\|'_{L^{{\Psi_r}}(\R^n)}\leq 1\right\}
    \\
  &=\sup\left\{\lim_{j\to\infty}\left|  \int_{\R^n} \nabla T f_j(x)\cdot h(x)\,dx\right|:
    h\in (L^{{\Psi_r}}(\R^n))^n, \big\|\|h\|\big\|'_{L^{{\Psi_r}}(\R^n)}\leq 1\right\}
    \\
&\leq\liminf_{j\to\infty} \sup\left\{\left|  \int_{\R^n} \nabla T f_j(x)\cdot h(x)\,dx\right|:
    h\in (L^{{\Psi_r}}(\R^n))^n, \big\|\|h\|\big\|'_{L^{{\Psi_r}}(\R^n)}\leq 1\right\}
    \\
&=\liminf_{j\to\infty} \big\| \|\nabla T f_j\|\big\|_\lphir.
\end{align*}
This proves \eqref{ineq_norm_weak_limit}.

From \eqref{lim_grad_approx}, \eqref{ineq_grad_approx}, and \eqref{ineq_norm_weak_limit}, we conclude that
\begin{equation}\label{PS_in_terms_of_norms}
 \big\|\|\nabla T f\|\big\|_\lphir\leq\big\|\|\nabla f\|\big\|_\lphir.
\end{equation}

Our fourth claim is that \eqref{eq_polyaszego_w1phi} holds when $\Phi$ is replaced by $\Phi_r$.  To see this, note that \eqref{PS_in_terms_of_norms} holds if $\Phi_r$ is replaced by $a\,\Phi_r$  for any $a>0$, because all the preceding arguments are valid with this replacement, due to  $\int_{\R^n}a\, \Phi_r(\| \nabla f(x)\|)\,dx<\infty$.  If we choose $a$ so that $\int_{\R^n}a\,\Phi_r(\|\nabla f(x)\|)\,dx=1$, then $\big\|\|\nabla f\|\big\|_{L^{a\,\Phi_r}(\R^n)}=1$, by \cite[Proposition~2.1.10(4)]{ES} and the fact that $\|\nabla f\|\in H^{a\Phi_r}(\R^n)$. Thus \eqref{PS_in_terms_of_norms} becomes
\[
\big\|\|\nabla T f\|\big\|_{L^{a\Phi_r}(\R^n)}\leq1.
\]
By \cite[Proposition~2.1.10(2)]{ES}, the previous inequality holds if and only if
\begin{equation}\label{eq:Phir_fine}
 \int_{\R^n}a\,\Phi_r(\|\nabla Tf(x)\|)\, dx\leq 1= \int_{\R^n}a\,\Phi_r(\|\nabla f(x)\|)\, dx.
\end{equation}
This proves the fourth claim.

Since the nonnegative function $\Phi_r$ increases to $\Phi$ as $r\to \infty$, one may apply the monotone convergence theorem to both sides of \eqref{eq:Phir_fine} and obtain  \eqref{eq_polyaszego_w1phi} when $\cH^n(\supp{f})<\infty$.
To remove the latter restriction, let $f\in W^{1,1}_{loc}(\R^n)\cap\cV(\R^n)$ and let $f_c=\max\{f-c,0\}$ for $c>0$. Then $f_c$ satisfies the hypotheses of the theorem and $\cH^n(\supp{f_c})<\infty$.  Moreover, $\nabla f_c=\nabla f$ on $\{x: f(x)> c\}$ and $Tf_c=\max \{Tf -c,0\}$, by Proposition~\ref{propoct14} with $\varphi(t)=(t-c)^+$, so $\nabla T f_c=\nabla T f$ on $\{x: T f(x)> c\}$. By \eqref{eq_polyaszego_w1phi} with $f$ replaced by $f_c$, we have
\begin{eqnarray*}
 \int_{\{x: T f(x)> c\}}\Phi(\|\nabla Tf(x)\|)\, dx&=&\int_{\R^n}\Phi(\|\nabla Tf_c(x)\|)\, dx\\
 &\leq &\int_{\R^n}\Phi(\|\nabla f_c(x)\|)\, dx=\int_{\{x: f(x)> c\}}\Phi(\|\nabla f(x)\|)\, dx.
\end{eqnarray*}
Letting $c\to0$ and applying the monotone convergence theorem, we obtain \eqref{eq_polyaszego_w1phi}.
\end{proof}

\begin{lem}\label{lem1apr21}
Let $X=\cS(\R^n)$ or $\cV(\R^n)$ and let $T:X\to X$ be a smoothing rearrangement.  If $f\in W^{1,\infty}(\R^n)\cap X$, then $Tf\in W^{1,\infty}(\R^n)$ and
\begin{equation}\label{eq_pInfty}
\esssup_{x\in \R^n}\|\nabla Tf(x)\|\leq \esssup_{x\in \R^n}\|\nabla f(x)\|.
\end{equation}
\end{lem}

\begin{proof}
If $f\in W^{1,\infty}(\R^n)\cap X$, then by \cite[Proposition~3.17]{Baer}, $f$ coincides $\cH^n$-almost everywhere with a Lipschitz function in $X$.  We may therefore assume that $f$ is Lipschitz, and then, as in the proof of Theorem~\ref{polyaszego_lip}, also assume that $Tf$ is Lipschitz. By  \cite[Corollary~3.4]{Baer}, \eqref{eq_pInfty} is equivalent to $L_1\le L_2$, where $L_1$ and $L_2$ are the Lipschitz constants of $Tf$ and $f$, respectively. Since
$$L_1=\sup_{d>0} \omega_d(Tf)/d\le \sup_{d>0} \omega_d(f)/d=L_2,$$
the proof is complete.
\end{proof}

\begin{cor}\label{polyaszego_w12}
Let $T:\cV(\R^n)\to \cV(\R^n)$ be a smoothing rearrangement and let $1\le p\le \infty$.  If $f\in W^{1,p}(\R^n)\cap \cV(\R^n)$, then $Tf\in W^{1,p}(\R^n)$ and
\begin{equation}\label{eq_polyaszego_w12}
\big{\|} \|\nabla Tf(x)\| \big{\|}_p\leq \big{\|} \|\nabla f(x)\| \big{\|}_p,
\end{equation}
where $\|\cdot\|_p$ denotes the $L^p$ norm when $1\le p<\infty$ and the essential supremum over $\R^n$ when $p=\infty$.
\end{cor}

\begin{proof}
The case when $p=\infty$ corresponds to Lemma~\ref{lem1apr21}.  Suppose that $1\le p<\infty$ and let $f\in W^{1,p}(\R^n)\cap \cV(\R^n)$. As $W^{1,p}(\R^n)\subset W^{1,1}_{loc}(\R^n)$,  Theorem~\ref{polyaszego_w12 new} with $\Phi(t)=t^p$ gives \eqref{eq_polyaszego_w12}. Our assumptions on $f$ show that the right-hand side of \eqref{eq_polyaszego_w12} is finite, implying that $Tf\in W^{1,p}(\R^n)$.
\end{proof}

The following example shows that the assumption that $T$ is smoothing in Theorems~\ref{polyaszego_lip} and~\ref{polyaszego_w12 new} and Corollary~\ref{polyaszego_w12} cannot be dropped in general.  It also shows that when $\Phi(t)=t$ in Theorem~\ref{polyaszego_lip} or~\ref{polyaszego_w12 new}, or $p=1$ in Corollary~\ref{polyaszego_w12}, the smoothing property is not necessary for the stated inequalities to hold.

\begin{ex}\label{exaug721}
{\rm  Let $K\in {\mathcal K}^n_{(o)}$ and $\cH^n(K)=\kappa_n$ and let $T:{\cV}(\R^n)\to{\cV}(\R^n)$ be the rearrangement defined in Example~\ref{exaug61}(ii).  We make the following two claims.

\smallskip

\noindent{(i)} Inequalities \eqref{eqoct171} and \eqref{eq_polyaszego_w1phi} for $\Phi(t)=t$, and \eqref{eq_polyaszego_w12} for $p=1$, each hold if and only if $K=x+B^n$ for some $x\in D^n$.

\noindent{(ii)} Inequalities \eqref{eqoct171} and \eqref{eq_polyaszego_w1phi} for strictly convex real-valued $\Phi$, and \eqref{eq_polyaszego_w12} for $1<p\le\infty$, each hold if and only if $K=B^n$.

\smallskip

Suppose these claims are true.  If $K$ is not a ball, then $T$ is not smoothing by Example~\ref{exaug61}(ii) and Theorems~\ref{polyaszego_lip} and~\ref{polyaszego_w12 new} (for real-valued strictly convex $\Phi$), and Corollary~\ref{polyaszego_w12} fail by (i) and (ii).  If $K=x+B^n$ for some $o\neq x\in D^n$, then $T$ is not smoothing by Example~\ref{exaug61}(ii) but nevertheless Theorems~\ref{polyaszego_lip} and~\ref{polyaszego_w12 new} with $\Phi(t)=t$, and Corollary~\ref{polyaszego_w12} with $p=1$, hold by (i).

To prove the two claims, we note first that if $f\in {\cV}(\R^n)$ and
$\alpha_{f,t}=\cH^n(\{x:f(x)\ge t\})$ for $t\ge 0$, the definition of $T$ yields
\begin{equation}\label{eq823212}
\{x:Tf(x)\ge t\}=\left(\frac{\alpha_{f,t}}{\kappa_n}\right)^{1/n}K
\end{equation}
for $t>0$.

Now let $M\in {\mathcal K}^n_{(o)}$ and $\cH^n(M)=\kappa_n$, and let $f_M(x)=(1-h_{M^\circ}(x))^+$ for $x\in \R^n$, where $M^\circ$ is the polar body of $M$ and $s^+$ is the nonnegative part of $s\in \R$. Then $\{x:f_M(x)\ge t\}=(1-t)^+M$ for $t>0$.  Since $\alpha_{f_M,t}=\left((1-t)^+\right)^n\kappa_n$, \eqref{eq823212} with $f=f_M$ implies that
\begin{equation}\label{eq823213}
Tf_M=f_K.
\end{equation}
When $M=B^n$, we have $f_{B^n}(x)=(1-\|x\|)^+$ for $x\in \R^n$ and hence
\begin{equation}\label{eqnormBn}
\|\nabla f_{B^n}(x)\|=
\begin{cases}
1, & \text{if $x\in D^n\setminus \{o\}$,}\\
0, & \text{if $x\not\in B^n$}.
\end{cases}
\end{equation}

The coarea formula for Lipschitz functions $f$ on $\R^n$ (see \cite[Theorem~4.19]{Baer}, \cite[Theorem~1, p.~112]{EG}) states that
\begin{align}\label{eqCAF}
\int_{\R^n}\|\nabla f(x)\|\,dx= \int_0^\infty  \cH^{n-1}(\{x:f(x)=t\})\,dt.
\end{align}		

Suppose that $T$ satisfies \eqref{eqoct171} or \eqref{eq_polyaszego_w1phi} with $\Phi(t)=t$, or \eqref{eq_polyaszego_w12} with $p=1$.  Then, using \eqref{eqCAF} with $f=f_K$, \eqref{eq823213} with $M=B^n$, either \eqref{eqoct171}, or \eqref{eq_polyaszego_w1phi}, or \eqref{eq_polyaszego_w12} with $f=f_{B^n}$, and \eqref{eqnormBn}, we obtain
\begin{eqnarray}\label{arr824}
\cH^{n-1}(\partial K)&=&n\int_0^1 (1-t)^{n-1}\cH^{n-1}(\partial K)\,dt=
n\int_0^\infty\cH^{n-1}(\{x:f_{K}(x)=t\})\,dt\\
&=&n\int_{\R^n}\|\nabla f_K(x)\|\,dx= n\int_{\R^n}\|\nabla Tf_{B^n}(x)\|\,dx\nonumber\\
&\le & n\int_{\R^n}\|\nabla f_{B^n}(x)\|\,dx=n\kappa_n=\cH^{n-1}(\partial B^n)\nonumber.
\end{eqnarray}
Because $\cH^n(K)=\cH^n(B^n)$, equality must hold in the isoperimetric inequality and consequently $K=x+B^n$ for some $x\in D^n$.
Conversely, suppose that $K=x+B^n$ for some $x\in D^n$ and $f\in W^{1,1}(\R^n)\cap \cV(\R^n)$.  From \eqref{eq823212} and the fact that the decreasing function $t\mapsto \alpha_{f,t}$ can only have countably many discontinuities, we get
$$
\{x:Tf(x)=t\}=\{x:Tf(x)\ge t\}\setminus \bigcup_{s>t}\{x:Tf(x)\ge s\}\subset \left(\frac{\alpha_{f,t}}{\kappa_n}\right)^{1/n}\partial K
$$
for almost all $t>0$.
Since $K$ is a translate of $B^n$, this and \eqref{eqCAF} imply that
\begin{align*}
\int_{\R^n}\|\nabla Tf(x)\|\, dx&= \int_0^\infty  \cH^{n-1}(\{x:Tf(x)=t\})\,dt\le
\int_0^\infty  \cH^{n-1}\left(\left(\frac{\alpha_{f,t}}{\kappa_n}\right)^{1/n}\partial K\right)\,dt
\\&=\int_0^\infty  \cH^{n-1}\left(\left(\frac{\alpha_{f,t}}{\kappa_n}\right)^{1/n}\partial B^n\right)\,dt=
\int_{\R^n}\|\nabla f^{\#}(x)\| dx.	
\end{align*}
Hence, $T$ satisfies \eqref{eqoct171} and \eqref{eq_polyaszego_w1phi} with $\Phi(t)=t$, and \eqref{eq_polyaszego_w12} with $p=1$, as the Schwarz rearrangement does so.  This proves (i).

For (ii), assume first that $\Phi$ is real-valued and strictly convex, and that \eqref{eqoct171} or \eqref{eq_polyaszego_w1phi} holds.  Let $\Phi(1)=c>0$.  Note that the measure on $K$ with differential $dx/\kappa_n$ is a probability measure, since $\cH^n(K)=\kappa_n$.  We use \eqref{eqnormBn}, either \eqref{eqoct171} or \eqref{eq_polyaszego_w1phi}, Jensen's inequality (see, e.g.~\cite[p.~62, Proposition~5]{RaoRen}), \eqref{arr824}, and the isoperimetric inequality, to obtain
\begin{eqnarray*}
c&=&\int_{\R^n}\Phi\left(\|\nabla f_{B^n}(x)\|\right) \frac{dx}{\kappa_n}\ge
\int_{\R^n}\Phi\left(\|\nabla Tf_{B^n}(x)\|\right) \frac{dx}{\kappa_n}\ge
\int_{K}\Phi\left(\|\nabla Tf_{B^n}(x)\|\right) \frac{dx}{\kappa_n}\\
&\ge & \Phi\left(\int_{K}\|\nabla Tf_{B^n}(x)\| \frac{dx}{\kappa_n}\right)=\Phi\left( \frac{1}{n\kappa_n}\cH^{n-1}(\partial K)\right)\ge
\Phi\left(\frac{1}{n\kappa_n}\cH^{n-1}(\partial B^n)\right)=\Phi(1)=c.
\end{eqnarray*}
It follows that there is equality in the isoperimetric inequality, giving $K=x+B^n$ for some $x\in D^n$, as before. But now equality also holds in Jensen's inequality with the strictly convex function $\Phi$, so $\|\nabla Tf_{B^n}(\cdot)\|=\|\nabla f_{K}(\cdot)\|$ must be constant $\cH^n$-almost everywhere on $K$. This is only possible when $x=o$. If we assume instead that \eqref{eq_polyaszego_w12} holds with $1<p<\infty$, we can apply the same argument with $\Phi(t)=t^p$.  Conversely, when $K=B^n$, the rearrangement $T$ is the Schwarz rearrangement and therefore the P\'olya-Szeg\H{o} inequality holds.

Finally, for (ii) when $p=\infty$, take $f=f_{B^n}$ in \eqref{eq_polyaszego_w12}.  We have $Tf_{B^n}=f_K$ by \eqref{eq823213}, and it is clear that
$\esssup_{x\in \R^n}\|\nabla f_{B^n}(x)\|=1$, while $\esssup_{x\in \R^n}\|\nabla f_K(x)\|>1$ if and only if $K\ne B^n$.  This completes the proof of (ii).
}
\end{ex}

\section{The anisotropic case}\label{anisotropic}

The following result generalizes Theorem~\ref{polyaszego_lip}, which corresponds to the case when $K=B^n$.

\begin{thm}\label{anisotropic polyaszego_lip}
Let $X={\mathcal{S}}(\R^n)$ or $\cV(\R^n)$, let $K\in {\mathcal K}^n_{(o)}$, let $T:X\to X$ be a rearrangement, and let $\Phi$ be a Young function.  If $T$ is $K$-smoothing and $f\in X$ is Lipschitz, then $Tf$ coincides with a Lipschitz function $\cH^n$-almost everywhere on $\R^n$, and
\begin{equation}\label{anisotropic eqoct171}
\int_{\R^n}\Phi\left(h_{-K}(\nabla Tf(x))\right)\, dx\leq \int_{\R^n}\Phi\left(h_{-K}(\nabla f(x))\right)\, dx,
\end{equation}
where the integrals may be infinite.
\end{thm}

\begin{proof}
We can argue as in the proof of Theorem~\ref{polyaszego_lip}, with very few changes.  With arguments as at the beginning of that proof, we may assume without loss of generality that $\Phi$ is a nontrivial real-valued Young function.  Let $f\in X$ be Lipschitz with Lipschitz constant $L$.  Choosing $0<r\le R$ such that $rB^n\subset K\subset RB^n$, we may use Lemma~\ref{cor1may21} to conclude that $Tf$ coincides ${\mathcal{H}}^{n}$-almost everywhere with a Lipschitz function with Lipschitz constant at most $LR/r$.  Inequality \eqref{eqoct111} follows as before.  Instead of \eqref{eqoct111a} we have
$$\max\left\{\|\nabla f(x)\|,\|\nabla Tf(x)\|\right\}\le LR/r$$
for ${\mathcal{H}}^{n}$-almost all $x\in \R^n$. Then, if
$C$ is the convex body from Lemma~\ref{lemoct20} corresponding to $M=LR/r$ and $\Phi$, we have
$$h_C(y,1)-1=b\,\Phi(h_K(y))=b\,\Phi(h_{-K}(-y))$$
for $y\in \R^n$ with $h_K(y)\le M$ and some $b>0$.  As before, this leads to  \eqref{anisotropic eqoct171}.
\end{proof}

Finally, we present in Theorem~\ref{AnisotropicW11} an anisotropic version of Theorem~\ref{polyaszego_w12 new}, which again corresponds to the case when $K=B^n$.  We shall need the following lemma.  Recall that a convex body is {\em smooth} if all its boundary points are regular and {\em strictly convex} if it does not contain a line segment in its boundary; see \cite[pp.~83, 87]{Sch14}.

\begin{lem}\label{Mar24lemma}
Let $L\in {\mathcal K}^n_{(o)}$ be a smooth and strictly convex body, and let $u:\R^n\to \R^n$ and $v:\R^n\to \R$ be measurable.  Then there is a measurable $w:\R^n\to \R^n$ such that

\noindent{\rm{(i)}} $h_{L^\circ}(w(x))=v(x)$ for all $x$ such that $u(x)\neq o$, and

\noindent{\rm{(ii)}} $u(x)\cdot w(x)=h_L(u(x))\, h_{L^\circ}(w(x))$ for all $x\in \R^n$.
\end{lem}

\begin{proof}
If $A=\{x: u(x)\neq 0\}$, then $A$ is $\cH^n$-measurable.  As is observed in \cite[Remark~1.7.14]{Sch14}, it follows easily from \eqref{polarrad} that since $L$ is smooth and strictly convex, the same is true of $L^{\circ}$.  Let $n(L^\circ, y)$ denote the unit outer normal to $L^{\circ}$ at $y\in \partial L^\circ$.  Note that $u/h_L(u)=\rho_{L^\circ}(u)u\in \partial L^{\circ}$ by \eqref{polarrad}, and define $\eta:S^{n-1}\to S^{n-1}$ by
$$\eta(u)=n\left(L^\circ, u/h_L(u)\right).$$
The map $\eta$ is continuous, since $L^{\circ}$ is a convex body of class $C^1$ by \cite[Theorem~2.2.4]{Sch14}.  The map $f:\R^n\setminus\{o\}\to S^{n-1}$ defined by $f(z)=z/\|z\|$ is also continuous.  The composition $\eta\circ f\circ u:A\to S^{n-1}$ of measurable functions is therefore also measurable, and so is its composition with the continuous support function $h_{L^\circ}:\R^{n}\to \R$  (see \cite[p.~115]{Sch14}).   Define
$$
w(x)=\begin{cases}
v(x)\frac{\eta(f(u(x)))}{h_{L^\circ}(\eta(f(u(x))))},& {\text{if $x\in A$,}}\\
0,& {\text{if $x\not\in A$}}.
\end{cases}
$$
Then $w:\R^n\to \R^n$ is measurable and clearly satisfies (i).  By its definition, $w(x)$ is an outer normal to $L^\circ$ at $f(u(x))/h_L(f(u(x)))=u(x)/h_L(u(x))$, so (ii) holds due to the equality condition for \eqref{CS_polar} stated immediately after it.
\end{proof}

\begin{thm}\label{AnisotropicW11}
Let $K\in {\mathcal K}^n_{(o)}$, let $T:\cV(\R^n)\to \cV(\R^n)$ be a $K$-smoothing rearrangement, and let $\Phi$ be a Young function.  If $f\in W^{1,1}_{loc}(\R^n)\cap\cV(\R^n)$ and $\int_{\R^n}\Phi\left(h_{-K}(\nabla f(x))\right)\, dx<\infty$, then $Tf\in W^{1,1}_{loc}(\R^n)$ and
\begin{equation}\label{eq_anisotropic_w1phi}
\int_{\R^n}\Phi\left(h_{-K}(\nabla Tf(x))\right)\, dx\leq \int_{\R^n}\Phi\left(h_{-K}(\nabla f(x))\right)\, dx.
\end{equation}
\end{thm}

\begin{proof}
The proof follows that of Theorem~\ref{polyaszego_w12 new}, up to the end of the proof of the second claim, verbatim except that the role of Theorem~\ref{polyaszego_lip} is now played by Theorem~\ref{anisotropic polyaszego_lip} and that $\|\nabla f\|$ and $\|\nabla Tf\|$ are replaced by $h_{-K}(\nabla f)$ and $h_{-K}(\nabla T f)$, respectively.  Instead of the third claim, that \eqref{ineq_norm_weak_limit} holds, we claim that
\begin{equation}\label{new ineq_norm_weak_limit}
\big\|h_{-K}(\nabla T f)\big\|_\lphir\leq\liminf_{j\to\infty} \big\|h_{-K}(\nabla  T f_j)\big\|_\lphir.
\end{equation}
The argument is similar to the one given for \eqref{ineq_norm_weak_limit}.  We first prove that if $u\in (L^{\Phi_r}(\R^n))^n$, then
\begin{equation}\label{new norm_weak_limit}
\big\|h_{-K}(u)\big\|_\lphir=\sup \left\{\left| \int_{\R^n} u(x)\cdot w(x)\,dx\right|: w\in (L^{{\Psi_r}}(\R^n))^n, \big\|h_{-K^\circ}(w)\big\|'_{L^{{\Psi_r}}(\R^n)}\leq 1\right\},
\end{equation}
where $-K^\circ=(-K)^\circ$ is the polar body of $-K$; see \eqref{polarbody}. For this, we apply \cite[(10), Proposition~10, Section~3.4]{RaoRen} to $h_{-K}(u)$, which yields
\begin{equation}\label{new norm_weak_limit2}
\big\|h_{-K}(u)\big\|_\lphir=\sup \left\{\left| \int_{\R^n} h_{-K}(u(x))\, v(x)\,dx\right|: v\in L^{{\Psi_r}}(\R^n), \|v\|'_{L^{{\Psi_r}}(\R^n)}\leq 1\right\}.
\end{equation}
(Here we have used the fact that ${\mathcal M}^{{\Psi_r}}=L^{{\Psi_r}}(\R^n)$, where ${\mathcal M}^{{\Psi_r}}$ is the closure of the span of all linear step functions in $L^{{\Psi_r}}(\R^n)$; this was explained after \eqref{norm_weak_limit2M}, along with a warning about the different notation employed in \cite{RaoRen}.)  Let $S_1$ and $S_2$ denote the right-hand sides of \eqref{new norm_weak_limit} and \eqref{new norm_weak_limit2}, respectively.  From \eqref{CS_polar} with $K$, $x$, and $y$ replaced by $-K$, $u$, and $w$, respectively, we obtain $S_1\leq S_2$.

For the converse, we can restrict the supremum in \eqref{new norm_weak_limit2} to nonnegative $v$.  Let $u\in (L^{\Phi_r}(\R^n))^n$ and let $v\in L^{{\Psi_r}}(\R^n)$, $v\geq0$, and $\|v\|'_{L^{{\Psi_r}}(\R^n)}\leq1$.  By \cite[Theorem~2.7.1]{Sch14}, the set of smooth and strictly convex bodies is dense in $\cK^n$ with the Hausdorff metric, while the compact convex sets strictly contained in $K$ and strictly containing $aK$ for a fixed $0<a<1$ form an open set in $\cK^n$.  We may therefore choose a sequence $(K_m)$ of smooth and strictly convex bodies converging to $K$ as $m\to \infty$ in the Hausdorff metric and such that
\begin{equation}\label{eqmar247}
\frac12 K\subset K_m\subset K
\end{equation}
for $m\in \N$. Let $w_m:\R^n\to S^{n-1}$ be the measurable vector field supplied by Lemma~\ref{Mar24lemma} with $L=-K_m$.  Since $v\in L^{\Psi_r}(\R^n)$, we have $h_{-K_m^\circ}(w_m)\in L^{\Psi_r}(\R^n)$; this is equivalent to $\|w_m\|\in L^{\Psi_r}(\R^n)$ and hence $w_m\in (L^{\Psi_r}(\R^n))^n$.  By \eqref{eqmar247}, $-K^\circ\subset - K_m^\circ$, so from Lemma~\ref{Mar24lemma}(i), we get
$$\|h_{-K^\circ}(w_m)\|'_{L^{\Psi_r}(\R^n)}\le  \|h_{-K_m^\circ}(w_m)\|'_{L^{\Psi_r}(\R^n)} =\|v\|'_{L^{\Psi_r}(\R^n)}\le 1.$$
Using (i) and (ii) of Lemma~\ref{Mar24lemma} with $L=-K_m$, we obtain
\begin{equation}\label{eqmar257}
S_1\geq \left| \int_{\R^n} u(x)\cdot w_m(x)\,dx\right|=\left| \int_{\R^n} h_{-K_m}(u(x))\, h_{-K_m^\circ}(w_m(x))\,dx\right|=\left| \int_{\R^n} h_{-K_m}(u(x))\, v(x)\,dx\right|.
\end{equation}
By \eqref{eqmar247},
$$0\le h_{-K_m}(u(x))\,v(x)\le h_{-K}(u(x))\,v(x)$$
and the function on the right-hand side is integrable.  Taking the limit as $m\to\infty$ in \eqref{eqmar257}, the dominated convergence theorem yields
$$S_1\ge \lim_{m\to\infty}\int_{\R^n} h_{-K_m}(u(x))\, v(x)\,dx=\int_{\R^n} h_{-K}(u(x))\, v(x)\,dx.$$
This proves that $S_1\geq S_2$ and concludes the proof of \eqref{new norm_weak_limit}.  Now \eqref{new ineq_norm_weak_limit} follows from \eqref{new norm_weak_limit} by the same argument that showed that \eqref{ineq_norm_weak_limit} follows from \eqref{norm_weak_limit}.

The remainder of the proof of the theorem is a repetition of the last part of the proof of Theorem~\ref{polyaszego_w12 new}, from the point where \eqref{ineq_norm_weak_limit} has been established onwards.
\end{proof}

Suppose that $K\subset \R^n$ is an $o$-symmetric convex body.  Then $$h_K(x)=h_{-K}(x)=\|x\|_{K^\circ}$$
for $x\in \R^n$, where $\|\cdot\|_{K^\circ}$ is the norm for which the unit ball is $K^\circ$, the polar body of $K$, defined by \eqref{polarbody}.  In this case, \eqref{anisotropic eqoct171} and \eqref{eq_anisotropic_w1phi} may be rewritten in the form
$$
\int_{\R^n}\Phi\left(\|\nabla Tf(x)\|_{K^\circ}\right)\, dx\leq \int_{\R^n}\Phi\left(\|\nabla f(x)\|_{K^\circ}\right)\, dx.
$$

\section{Appendix}\label{Appendix}

The purpose of this appendix is to provide a proof of Proposition~\ref{contraction}. This follows easily from Lemma~\ref{lemoct14} below, first proved by Crowe, Zweibel, and Rosenbloom \cite[Theorem~3]{CZR} for Schwarz rearrangement without the assumption that $F(s,0)$ and $F(0,t)$ decrease with $s\ge 0$ and $t\ge 0$, respectively. Variants of Proposition~\ref{contraction} are stated for general rearrangements by Brock and Solynin \cite[Theorem~3.1]{BS} and by Van Schaftingen and Willem \cite[Corollary~1]{VSW}, whose approaches to rearrangements differ from ours; see \cite[Appendix]{BGGK} for a comparison.  Brock and Solynin refer to \cite{CZR} for a proof, but do not explain why it should apply to general rearrangements, while \cite[Corollary~1]{VSW} is stated with the extra assumption that the function $j$ is even.  The proof of \cite[Corollary~1]{VSW} is based on that of \cite[Proposition~3.3.9]{VSPhD}, which does not assume that $j$ is even, or that it is nonnegative, but which requires a considerable amount of preliminary observations and terminology.  For this reason, we prefer to follow the argument in \cite{CZR}.

\begin{lem}\label{lemoct14}
Let $F:\R^2\to \R$ be continuous with $F(0,0)=0$ and such that $F(s,0)$ and $F(0,t)$ decrease with $s\ge 0$ and $t\ge 0$, respectively.  Suppose that for all coordinate rectangles $R=[a,b]\times [c,d]$, where $a\le b$, $c\le d$,
\begin{equation}\label{eqoct140}
G(R)=F(b,d)+F(a,c)-F(b,c)-F(a,d)\ge 0.
\end{equation}
Let $T:\cV(\R^n)\to \cV(\R^n)$ be a rearrangement.  If $f,g\in\cV(\R^n)$, then
\begin{equation}\label{eqoct141}
\int_{\R^n}F(f(x), g(x))\,dx \le \int_{\R^n}F(Tf(x), Tg(x))\,dx.
\end{equation}
\end{lem}

\begin{proof}
The function $G$ is additive on coordinate rectangles, that is, if $R$, $S$, and $R\cup S$ are non-overlapping coordinate rectangles, then $G(R\cup S)=G(R)+G(S)$.  This allows $G$ to be extended to a measure $\nu$ on $\R^2$ such that each coordinate rectangle $R$ is $\nu$-measurable and $\nu(R)=G(R)$; see \cite[pp.~64--68]{Saks} (where the union of sets is denoted by $+$).

Let $H$ denote the Heaviside function, i.e., $H(x)=1$ if $x\ge 0$ and $H(x)=0$ if $x<0$.  Then for $b,d,s,t\ge 0$,
\begin{equation*}\label{eq2}
1_{[0,b]\times [0,d]}(s,t)=H(b-s)H(d-t).
\end{equation*}
It follows that
\begin{eqnarray*}
\int_0^{\infty}\int_0^{\infty}H(b-s)H(d-t)\,d\nu(s,t)&=&
\int_0^{\infty}\int_0^{\infty}1_{[0,b]\times [0,d]}(s,t)\,d\nu(s,t)\\
&=&\nu([0,b]\times [0,d])=G([0,b]\times [0,d])\\
&=&F(b,d)-F(b,0)-F(0,d).
\end{eqnarray*}
From this we obtain
$$
F(b,d)=F(b,0)+F(0,d)+\int_0^{\infty}\int_0^{\infty}H(b-s)H(d-t)\,d\nu(s,t).
$$
On setting $b=f(x)$ and $d=g(x)$ and integrating, this gives
\begin{equation}\label{eqoct143}
\int_{\R^n}F(f(x),g(x))\,dx=\int_{\R^n}F(f(x),0)\,dx+\int_{\R^n}F(0,g(x))\,dx
+I(f,g),
\end{equation}
where by Fubini's theorem,
\begin{eqnarray}\label{eqoct146}
I(f,g)&=&\int_{\R^n}\int_0^{\infty}\int_0^{\infty}H(f(x)-s)H(g(x)-t)\,d\nu(s,t)\,dx
\nonumber\\
&=&\int_0^{\infty}\int_0^{\infty}\int_{\R^n}H(f(x)-s)H(g(x)-t)\,dx\,d\nu(s,t).
\end{eqnarray}
Our assumptions on $F$ imply that $-F(r,0)\ge 0$ increases with $r\ge 0$.  We can therefore apply \eqref{eqoct149}, with $\varphi(r)=-F(r,0)$ for $r\ge 0$, to obtain $-F(Tf(x),0)=T(-F(f(x),0))$, for each $x\in \R^n$.  With this, the layer-cake representation formula, and the equimeasurability of $T$, we conclude that
\begin{eqnarray}\label{eqoct144}
\int_{\R^n}F(Tf(x),0)\,dx&=&-\int_{\R^n}T(-F(f(x),0))\,dx\nonumber\\
&=&-\int_0^{\infty}{\cH}^n(\{x:T(-F(f(x),0))>t\})\,dt\nonumber\\
&=&-\int_0^{\infty}{\cH}^n(\{x:-F(f(x),0)>t\})\,dt=\int_{\R^n}F(f(x),0)\,dx.
\end{eqnarray}
Similarly,
\begin{equation}\label{eqoct145}
\int_{\R^n}F(0,Tg(x))\,dx=\int_{\R^n}F(0,g(x))\,dx.
\end{equation}
Since $\di_T$ is monotonic, we have
\begin{equation}\label{eqoct148}
{\cH}^n(\di_T(A\cap B))\le {\cH}^n((\di_TA)\cap \di_TB)
\end{equation}
whenever $A,B\in {\cL}^n$.  Consequently, using the measure-preserving property of $\di_T$, \eqref{eqoct148}, and \eqref{eqnov11}, we obtain
\begin{eqnarray*}\label{eq7}
\int_{\R^n}H(f(x)-s)H(g(x)-t)\,dx&=&{\cH}^n(\{x:f(x)\geq s\}\cap\{x:g(x) \geq  t\})\nonumber\\
&=&{\cH}^n(\di_T(\{x:f(x)\geq s\}\cap\{x:g(x) \geq t\}))\nonumber\\
&\le&{\cH}^n((\di_T\{x:f(x)\geq s\})\cap\di_T\{x:g(x)\geq t\})\nonumber\\
&=&{\cH}^n(\{x:Tf(x) \geq s\}\cap\{x:Tg(x)\geq t\})\nonumber\\
&=&\int_{\R^n}H(Tf(x)-s)H(Tg(x)-t)\,dx.
\end{eqnarray*}
By \eqref{eqoct146}, this yields $I(f,g)\le I(Tf,Tg)$.  Then \eqref{eqoct141} follows from \eqref{eqoct143}, \eqref{eqoct143} with $f$ and $g$ replaced by $Tf$ and $Tg$, respectively, \eqref{eqoct144}, and \eqref{eqoct145}.
\end{proof}

\noindent{\em Proof of Proposition~\ref{contraction}.}  Let $F(s,t)=-j(s-t)$ for $s,t\ge 0$.  Since $j$ is convex, for $r\in \R$ and $s,t\ge 0$, we have
$$j(r)-j(r-s)\le j(r+t)-j(r-s+t).$$
If $r=b-d$, $s=b-a$, and $t=d-c$, this gives
$$j(b-d)-j(a-d)\le j(b-c)-j(a-c),$$
yielding \eqref{eqoct140}.  Moreover, $F(r,0)=-j(r)$ and $F(0,r)=-j(-r)$ both decrease with $r\ge 0$ since $j\ge 0$ and $j(0)=0$.  Applying Lemma~\ref{lemoct14} with this choice of $F$, we obtain \eqref{eq_contraction}.

The $L^p$-contracting property results from taking $j(r)=|r|^p$, $p\ge 1$.
\qed

\end{document}